\documentclass[12pt,oneside,reqno]{amsart}
\usepackage{color}
\usepackage{esint,amssymb}
\usepackage{graphicx}
\usepackage{MnSymbol}
\usepackage{mathtools}
\usepackage{microtype}
\usepackage{amsmath}
\usepackage[colorlinks=true, pdfstartview=FitV, linkcolor=blue, citecolor=blue, urlcolor=blue,pagebackref=false]{hyperref}
\usepackage[margin=1in]{geometry}
\usepackage{slashed}
\usepackage[normalem]{ulem}
\usepackage{relsize}
\usepackage{cancel}
\usepackage{xifthen}
\usepackage{verbatim}
\DeclareFontEncoding{LS1}{}{}
\DeclareFontSubstitution{LS1}{stix}{m}{n}
\DeclareSymbolFont{stixsymbols}{LS1}{stixscr}{m}{n}
\SetSymbolFont{stixsymbols}{bold}{LS1}{stixscr}{b}{n}
\DeclareMathSymbol{\kay}{\mathalpha}{stixsymbols}{"6B}
\definecolor{darkgreen}{rgb}{0,0.5,0}
\definecolor{darkblue}{rgb}{0,0,0.7}
\definecolor{darkred}{rgb}{0.9,0.1,0.1}

\newcommand{\pfstep}[1]{\smallskip \noindent {\it #1.}}

\newtheorem{theorem}{Theorem}
\newtheorem{proposition}[theorem]{Proposition}
\newtheorem{lemma}[theorem]{Lemma}
\newtheorem{corollary}[theorem]{Corollary}

\theoremstyle{definition}
\newtheorem{remark}[theorem]{Remark}

\newtheorem{definition}[theorem]{Definition}

\newcommand{\cref}[1]{Corollary~\ref{c.#1}}

\numberwithin{equation}{section}
\numberwithin{theorem}{section}

\newcommand{\R}{\mathbb{R}}

\newcommand{\eqs}{=_{\scalebox{.6}{\mbox{S}}}}

\newcommand{\ep}{\varepsilon}

\newcommand{\test}[1][]{%
\ifthenelse{\equal{#1}{}}{omitted}{given}%
}

\renewcommand{\d}[1]{\ensuremath{\operatorname{d}\!{#1}}}
\newcommand{\norm}[1]{\left\lVert {#1} \right\rVert}

\newcommand{\jap}[1]{\left\langle {#1} \right\rangle}
\newcommand{\brk}[1]{\left\langle {#1} \right\rangle}

\renewcommand{\bar}{\overline}
\renewcommand{\tilde}{\widetilde}

\renewcommand{\part}{\partial}

\newcommand{\red}[1]{{\color{red} #1}}

\usepackage{dsfont}

\newcommand{\calD}{\mathcal D}
\newcommand{\calE}{\mathcal E}

\newcommand{\calI}{\mathcal I}
\newcommand{\calJ}{\mathcal J}

\newcommand{\calL}{\mathcal L}

\newcommand{\calT}{\mathcal T}

\newcommand{\calX}{\mathcal X}
\newcommand{\calY}{\mathcal Y}
\newcommand{\calZ}{\mathcal Z}

\def\f {\frac}
\def\rd {\partial}
\def\ls {\lesssim}
\def\de {\delta}
\def\i {\infty}
\def\alp {\alpha}
\def\bt {\beta}

\def\ep {\epsilon}

\def\om {\omega}

\newcommand{\ud}{\mathrm{d}}

\def\RR {\mathbb R}
\def\bbR {\mathbb R}

\def\srd {\slashed{\rd}}

\begin{document}

\title[Linear stability of traveling Maxwellians]{Linear stability of traveling Maxwellians for \\Landau equation with very soft potentials}
\begin{abstract}
	Consider the Landau equation with a Coulomb potential linearized around traveling Maxwellians. We prove that the microscopic part of the linear solution decays with an inverse polynomial rate, while the macroscopic part remains bounded but in general does not decay. In particular, the long-time dynamics are different from the linearization around either (1) traveling Maxwellians for moderately soft or hard potentials or (2) global (non-traveling) Maxwellians (for any potentials).
\end{abstract}

\author[Sanchit Chaturvedi]{Sanchit Chaturvedi}
\address[Sanchit Chaturvedi]{Cornell University, Department of Mathematics, 310 Malott Hall, Ithaca,
NY 14853, USA}
\email{sc3653@cornell.edu}
\author[Jonathan Luk]{Jonathan Luk}
\address[Jonathan Luk]{Department of Mathematics, Stanford University, 450 Jane Stanford Way, Stanford, CA 94305, USA}
\email{jluk@stanford.edu}

\maketitle

\vspace{-2ex}
\begin{center}{\it\large Dedicated to Professor Yan Guo, on the occasion of his 60th birthday}
\end{center}

\section{Introduction}
Consider the Landau equation on the whole space, i.e., the particle density $F: [0,T) \times \bbR^3_x \times \bbR^3_v \to \bbR_{\geq 0}$ satisfies the equation
\begin{equation}\label{eq:Landau}
\begin{split}
\rd_tF + v_i\rd_{x_i} F&=Q(F,F),\\
F(0,x,v)&=F_0(x,v),
\end{split}
\end{equation}
where 
\begin{equation}\label{eq:Q}
Q(F,G) = \rd_{v_i}\int_{\R^3}\phi_{ij}(v-v_*)[F(v_*)\rd_{v_j}G(v)-G(v)\rd_{v_j} F(v_*)]\ud v_*
\end{equation}
and $\phi$ is the following anisotropic matrix
\begin{equation}\label{eq:phi}
\phi_{ij}(y)=\left(\delta_{ij}-\f{y_i y_j}{|y|^2}\right)|y|^{2+\gamma}.
\end{equation}
Here, and below, we use the Einstein summation convention where repeated lower case Latin indices are summed over $i,j=1,2,3$. For the remainder of the paper, we fix $\gamma = -3$, which is the physical Coulomb case.

Our goal is to consider the linear stability of the traveling Maxwellian 
\begin{equation}\label{eq:mu}
\mu=e^{-\alpha|v|^2-\beta|x-tv|^2}
\end{equation}
for some $\alp,\bt>0$. Observe that $F = \mu$ is an explicit solution to \eqref{eq:Landau} with initial data $F_0(x,v) = e^{-\alp|v|^2 - \bt |x|^2}$ localized in phase space. 

To our knowledge, the stability of the traveling Maxwellian \eqref{eq:mu} has not been previously studied for the Landau equation. However, it has been studied for the Boltzmann equation in both cutoff \cite{AlGa09, BaGaGoLe16, Go97, To88} and non-cutoff cases \cite{HeLi2023}. In particular, the non-cutoff Boltzmann equation shares many similarities with the Landau equation. Regardless of whether the kernel is cutoff or non-cutoff, all the previous works concern the moderately soft potentials, Maxwellian molecules, and hard potentials cases (corresponding to $\gamma \in (-2,0)$, $\gamma = 0$ and $\gamma \in (0,1]$, respectively for \eqref{eq:phi} in the Landau case). The restriction is related to decay properties of solutions. Under such a restriction, the dispersion from the transport operator $\rd_t + v_i \rd_{x_i}$ dominates the dynamics and the collision term can be viewed as perturbative from the point of view of long-time asymptotics. As a result the traveling Maxwellian is orbitally stable but not asymptotically stable, and as $t\to \infty$, the solution approaches a solution to the linear transport equation which is in general not a traveling Maxwellian. This is most striking in the hard sphere Boltzmann case, where a full scattering theory in a neighborhood of the traveling Maxwellians (under an additional smallness assumption) has been established \cite{BaGaGoLe16}. In other words, not only are traveling Maxwellians not asymptotically stable, but in fact any nearby solutions to the transport equation can be achieved as a scattering state as $t\to \infty$. Due to the dominance of the transport dynamics, the mathematical study of the stability of the traveling Maxwellian is closely related to that of the stability of vacuum $f_{\mathrm{vac}} \equiv 0$, as has been explicitly pointed out in \cite{HeLi2023}. In view of the works \cite{sC2020a, jL2019} for the stability of vacuum for the Landau equation when $\gamma > -2$, one expects that traveling Maxwellians are also orbitally but not asymptotically stable for this equation.

By contrast, the dynamics is drastically different in a neighborhood of a global (non-traveling) Maxwellian, i.e., \eqref{eq:mu} with $\alp >0$ but $\bt =0$. In that case, the collisional term is no longer perturbative but instead drives the dynamics. For the Landau equation in the full range $\gamma \in [-3,1]$, as well as for the Boltzmann equation, the entropic dissipation from the collisional term gives a decay mechanism and the global Maxwellians are in fact nonlinearly asymptotically stable (see \cite{AMUXY12.3, AMUXY12, AMUXY12.2, Guo02, Ukai74} and further references in Section~\ref{sec:stability.glo.Max}).

We now return to our case of interest, i.e., the stability of traveling Maxwellian (\eqref{eq:mu} with $\alp,\bt>0$) for $\gamma = -3$. Due to the very soft potential, the techniques in \cite{AlGa09, BaGaGoLe16, sC2020a, HeLi2023, jL2019, Go97, To88} which treat the equation as a perturbation of the linear transport equation are no longer applicable, as one would need to deal with terms which are not time-integrable. In fact, we show that the \textbf{linearized dynamics is different both from the $\gamma \in (-2,1]$ case and the global Maxwellian case}. Specifically, we show that the linear perturbation converges as $t\to \infty$ to a solution to the linear transport equation; in general it does not converge to $0$. Moreover, after decomposing the solution into the macroscopic and microscopic parts, we prove that the microscopic part decays inverse polynomially in time. In other words, most of the solution decays due to entropic dissipation, but there is a residue part for which entropic dissipation is too weak and the dynamics is still dominated by transport.

Compared to the linearized perturbations of traveling Maxwellians for $\gamma > -2$, the difference we see in our case is that the linearized collisional term plays an important role so that the microscopic part of the solution decays. This is also different from the case of linearized perturbations of global (non-traveling) Maxwellians. In that case, the collisional term likewise causes the microscopic part to decay, but then the macroscopic part also decays through its interaction with the microscopic part. (At a technical level, this follows from suitable elliptic estimates.) However, in our case, the interaction between the microscopic and macroscopic parts is too weak so that the macroscopic part in general does not decay.

To precisely describe our result, we first write the perturbation around traveling Maxwellian $\mu$ as 
$$F=\mu+\sqrt{\mu}f.$$
Then \eqref{eq:Landau} reduces to the following equation for $f$
\begin{equation}\label{eq:Landau-f}
\rd_tf+v_i\rd_{x_i}f+L f=\Gamma(f,f),
\end{equation} 
where 
\begin{equation}\label{eq:L}
\begin{split}
- Lf&=\mu^{-\f 12}Q(\mu, \mu^{\f 12} f)+\mu^{-\f 12} Q(\mu^{\f 12} f,\mu)=:A f+Kf
\end{split}
\end{equation}
and 
\begin{equation}\label{eq:Gamma}
\Gamma=\mu^{-\f 12} Q(\mu^{\f 12} f,\mu^{\f 12} f).
\end{equation}

Dropping the nonlinear terms, we will study the linear equation
\begin{equation}\label{eq:linear.intro}
\rd_tf+v_i\rd_{x_i}f+L f=0.
\end{equation} 
For every $t,x$, the operator $L$ has a $5$-dimensional kernel given by 
\begin{equation}\label{eq:L.ker.intro}
	\mathrm{ker}(L) = \mathrm{span}\{\sqrt{\mu}, z_i \sqrt{\mu}, |z|^2 \sqrt{\mu} \}
\end{equation}
for 
\begin{equation}
	z = \sqrt{\alpha+\beta t^2}v - \f{x\beta t}{\sqrt{\alpha+\beta t^2}}.
\end{equation}
The kernel \eqref{eq:L.ker.intro} plays an important role for the dynamics. To capture this, denote by $\Pi$ the orthogonal projection to $\mathrm{ker}(L)$ with respect to $\ud v$. We will further denote $\Pi f = a(t,x) \sqrt{\mu} + b_i(t,x) z_i \sqrt{\mu} + c(t,x) |z|^2 \sqrt{\mu}$; $\Pi f$ will be called the macroscopic part while $(I - \Pi)f$ will be called the microscopic part.

For the remainder of the paper, we write $a \ls b$ if there exists $c>0$ such that $a\leq cb$. The constant $c>0$ typically depends only on $\alp$, $\bt$ and $K_{0}$, but will be specified in the context. We also denote $\jap{y} = (1+|y|^2)^{\f 12}$. The following is our main theorem concerning the linear equation:
\begin{theorem}\label{thm:main}
Let $t_0\geq 1$ and $K_{0} \in \mathbb Z_{\geq 3}$. Fix the background solution $\mu$ in \eqref{eq:mu} and consider the initial value problem for the linearized equation \eqref{eq:linear.intro} about $\mu$ with initial data
\begin{equation}
	f\restriction_{t=t_0} = f_{t_0}
\end{equation}
for some smooth $f_{t_0}:\bbR^3_x \times \bbR^3_v \to \bbR$. For every $K \in \mathbb Z_{\geq 3}$, define the initial energy $E_K$ by \eqref{eq:EK.def} below. Assume $f_{t_0}$ verifies
$$E_{K_{0}} <\infty.$$

Then the solution exists globally in the time interval $[t_0,\infty)$ and the following holds for some implicit constants which depend only on $\alp$, $\bt$ and $K_{0}$ (and are in particular independent of $t_0$):
\begin{enumerate}
\item \label{item:part.1}(Boundedness of weighted energy) For $\alp, \bt >0$ as in \eqref{eq:mu} and $\overline{\bt}$ sufficiently small depending on $\alp$, $\bt$, the following weighted energy is bounded:
\begin{align}
    \| e^{\f{\alp \bar{\bt} |x|^2}{2(\alp + \bar{\bt} t^2)}} \jap{x/t}^{K_{0}-1} f \|_{L^2_{x,v}} \ls E_{K_{0}}^{\f 12}, \label{eq:boundedness.main.thm} \\
    \| e^{\f 12 \alp |v|^2 + \f 12\bar{\bt} |x-tv|^2} \jap{x/t}^{K_{0}} f \|_{L^2_{x,v}} \ls t^{0.55} E_{K_{0}}^{\f 12} . \label{eq:growth.main.thm}
\end{align}
\item \label{item:part.2}(Decay of microscopic part) The microscopic part $(I - \Pi)f$ satisfies the following estimates
\begin{equation}\label{eq:decay.main.thm}
	\| \jap{z}^{-\f 12} \jap{x/t}^{K_{0}-2} (I - \Pi) f \|_{L^2_{x,v}}(t) \ls t^{-\f 12} E_{K_{0}}^{\f 12}.
\end{equation}
\item \label{item:part.3}(Improved decay of microscopic part in localized region) In the region $|x|\ls t$, the microscopic part $(I - \Pi)f$ obeys decay estimates with an improved rate. More precisely,
\begin{equation}\label{eq:improved.decay.main.thm}
	\| e^{-\f{\alp \bt |x|^2}{\alp + \bt t^2}} \jap{z}^{-1} \jap{x/t}^{K_{0}-3} (I - \Pi) f \|_{L^2_{x,v}}(t) \ls t^{-1} E_{K_{0}}^{\f 12}.
\end{equation}
\item \label{item:part.4}(Existence of limit as $t\to \infty$) Let $f^\sharp(t,x,v) = f(t,x+tv,v)$. Then, there exists $f^\sharp_\infty: \bbR^3_x \times \bbR^3_v\to \bbR$ such that 
\begin{equation}\label{eq:convergence.main.thm}
\| \jap{x}^{-\f 12} (f^\sharp(t,x,v) - f^\sharp_\infty (x,v)) \|_{L^2_{x,v}}(t) \ls t^{-\f 12} E_{3}^{\f 12}.
\end{equation}
\item \label{item:part.5}(Generic non-decay of the macroscopic part) For $t_0$ sufficiently large, there exists a large class of choices of $f_{t_0}$ such that the limit $f^\sharp_\infty$ is not identically $0$.
\end{enumerate}
\end{theorem}

The proofs of the five parts can be found in Proposition~\ref{prop:f.final}, Theorem~\ref{thm:microscopic}, Theorem~\ref{thm:microscopic.improved}, Theorem~\ref{thm:limit.def}, and Theorem~\ref{thm:nonvanishing}, respectively.



\begin{remark}
    Our theorem is only stated for $\gamma = -3$, which is the physical case. However, we expect that the phenomenon holds more generally for very soft potentials. In particular, a version of the theorem should in principle hold for the range $\gamma\in [-3,-2)$ (with a different $t$ rate).
\end{remark}

\begin{remark}\label{rmk:weight.improved.intro}
    Even though \eqref{eq:growth.main.thm} is not a boundedness estimate (as it grows with $t$), it is a crucial step in establishing the other bounds. Moreover, we included the bound here since it encodes some information not implied by the other estimates. For instance, interpolating this with \eqref{eq:decay.main.thm}, one can control higher $\jap{z}$ moments and obtain for every $\ell \in \mathbb Z_{\geq 0}$ and $\ep >0$,
    \begin{equation}\label{eq:weight.improved.intro}
        \| \jap{z}^{\ell} \jap{x/t}^{K_{0}-2} (I - \Pi) f \|_{L^2_{x,v}}(t) \ls_{\ep,\ell} t^{-\f 12+\ep} E_{K_{0}}^{\f 12}.
    \end{equation}
    We will prove this in Proposition~\ref{prop:weight.improved}.
\end{remark}

\begin{remark}[Estimates for higher derivatives]
	We expect higher derivative estimates similar to those in Theorem~\ref{thm:main}, e.g., for higher derivatives taken with respect to $\rd_x$ and $t \rd_x + \rd_v$, can also be proven with our methods. In fact, in order to obtain our results, it is already necessary to simultaneously prove some higher-order estimates with respect to $t \rd_x + \rd_v$. 
\end{remark}

\begin{remark}[Nonlinear dynamics]
Our theorem is purely about the linear dynamics. However, in view of the non-decay of the macroscopic parts of $f$, one expects the question of nonlinear stability --- or instability --- to hinge upon the precise structure of the nonlinear terms.
\end{remark}

\begin{remark}[Relation to the stability of vacuum for very soft potentials]
Our result does not directly address the question of stability of vacuum for very soft potentials for the Landau equation with $\gamma = -3$. However, our result applies in particular to small traveling Maxwellians backgrounds to show a new decay mechanism. Thus even if the stability of vacuum were true in this case, the proof would be different from \cite{jL2019, sC2020a} in that the dynamics are not dominated by transport alone as $t\to \infty$.
\end{remark}

\subsection{Ideas of the proof}

\subsubsection{Basic energy identity}

Solutions to the equation \eqref{eq:linear.intro} admit a basic energy identity
\begin{equation}\label{eq:energy.intro}
\f 12 \|f\|_{L^2_{x,v}}^2(t) + \int_{t_0}^t \langle f, Lf \rangle_{L^2_{x,v}}(s) \ud s = \f 12 \|f\|_{L^2_{x,v}}^2(t_0).
\end{equation}
The bulk term is non-negative, and in fact controls $(I-\Pi) f$ with degenerate weights. In particular, 
\begin{equation}\label{eq:coercive.intro}
\langle f, Lf \rangle_{L^2_v}(s) \gtrsim \| e^{-\f{\alp\bt|x|^2}{2(\alp+\bt t^2)}} \jap{z}^{-\f 12} (I-\Pi) f \|_{L^2_{v}}^2 + t^{-2} \| e^{-\f{\alp\bt|x|^2}{2(\alp+\bt t^2)}} \jap{z}^{-\f 32} \nabla_v (I-\Pi) f \|_{L^2_{v}}^2,
\end{equation}
where $z = \sqrt{\alp + \bt t^2} v - \f{x \bt t}{\sqrt{\alp + \bt t^2}}$. (In fact sharper estimates hold; see Lemma~\ref{lem:lower-bound-L}.) In particular, the combination of \eqref{eq:energy.intro} and \eqref{eq:coercive.intro} already suggests that $(I-\Pi)f$ obeys some weak averaged decay bound with degeneration for $|x| \gtrsim t$. To proceed, our first goal is to upgrade \eqref{eq:energy.intro} to conclude stronger transport-type estimates for $f$.

\subsubsection{Propagation of transport estimates}\label{sec:intro.transport.est}

Because the full solution $f$ does not decay, the best one can hope for is to prove estimates consistent with those capturing dispersion of the linear transport equation. A natural $L^2$-based norm for solutions $f_{\mathrm{lin}}$ to the linear transport equation
\begin{equation}\label{eq:linear.transport}
	\rd_t f_{\mathrm{lin}} + v_i \rd_{x_i} f_{\mathrm{lin}} = 0
\end{equation}
is 
\begin{equation}\label{eq:transport.norm}
	\| \jap{x-tv}^\ell \jap{v}^{j} e^{\mathfrak a|v|^2} e^{\mathfrak b |x-tv|^2} f_{\mathrm{lin}}\|_{L^2_{x,v}},
\end{equation}
which is independent of time. This is simply due to the facts that $T := \rd_t + v_i \rd_{x_i}$ satisfies $T v =0$ and $T(x-tv) = 0$.

There are various issues in propagating the norm \eqref{eq:transport.norm} for solutions to \eqref{eq:linear.intro}. 
\begin{enumerate}
\item \label{item:issue.1} There is a limit on $\mathfrak a$ and $\mathfrak b$ imposed by the collisional operator. This is already the case in the stability theory for global (non-traveling) Maxwellians, where one cannot put arbitrary Gaussian weights in $v$ in the energies; see \cite{StGu08}.
\item \label{item:issue.2} More seriously, whenever there is a weight $w$ depending on $v$ (such as $\jap{x-tv}^\ell$, $\jap{v}^\kay$, and/or $e^{\mathfrak a|v|^2} e^{\mathfrak b |x-tv|^2}$), $\langle w^2 f, Lf \rangle$ is no longer coercive as it mixes $\Pi f$ and $(I-\Pi)f$. Combined with the fact that $\Pi f$ does not decay (as we will prove in part \eqref{item:part.5} of Theorem~\ref{thm:main}), we in fact cannot expect energies with these weights to be bounded. 
\end{enumerate}

A crucial observation is as follows. After using the triangle inequality, the boundedness of \eqref{eq:transport.norm} implies, for a suitable set of indices, the boundedness of
\begin{equation}\label{eq:transport.weak.norm}
    \| \jap{x/t}^\kay e^{\mathfrak a \mathfrak b |x|^2/(\mathfrak a + \mathfrak b t^2)} f_{\mathrm{lin}}\|_{L^2_{x,v}}.
\end{equation}
The $|x|/t$ weights are less natural from the point of view of the transport equation, and indeed the norm \eqref{eq:transport.weak.norm} is not conserved for solutions to \eqref{eq:linear.transport}. Nonetheless, because the weights are $v$-independent, we can apply the coercivity estimate \eqref{eq:coercive.intro}, and it is now possible to bound solutions to \eqref{eq:linear.intro} in the norm \eqref{eq:transport.weak.norm} uniformly in $t$; see the main estimate \eqref{eq:boundedness.main.thm}:
\begin{equation}\label{eq:boundedness.main.thm.2}
    \| e^{\f{\alp \bar{\bt} |x|^2}{2(\alp + \bar{\bt} t^2)}} \jap{x/t}^{K_{0}-1} f \|_{L^2_{x,v}} \ls E_{K_{0}}^{\f 12}.    
\end{equation}

As we already mentioned, even for the linear transport equation, the norm \eqref{eq:transport.weak.norm} is not conserved. Thus, when we apply \eqref{eq:energy.intro} to derive \eqref{eq:boundedness.main.thm.2}, the estimate does not close on its own. Instead, we need to simultaneously bound a hierarchy of norms in addition to \eqref{eq:boundedness.main.thm.2}; in particular, we combine this with \eqref{eq:growth.main.thm}:
\begin{equation}\label{eq:growth.main.thm.2}
    \| e^{\f 12 \alp |v|^2 + \f 12\bar{\bt} |x-tv|^2} \jap{x/t}^{K_{0}} f \|_{L^2_{x,v}} \ls t^{0.55} E_{K_{0}}^{\f 12}. 
\end{equation}
Observe that the norm in \eqref{eq:growth.main.thm.2} is allowed to grow, and as we argued above in issue \eqref{item:issue.2}, this growth is necessary. Nonetheless, importantly, when \eqref{eq:growth.main.thm.2} is used in the derivation of \eqref{eq:boundedness.main.thm.2}, there are extra polynomially decaying $t$-weights so that \eqref{eq:growth.main.thm.2} is still useful to control the error terms.

Without getting into too many details at this point, we make two additional comments. First, observe that \eqref{eq:boundedness.main.thm.2} contains the weight $\jap{x/t}^{K_{0}-1}$ while \eqref{eq:growth.main.thm.2} contains the weight $\jap{x/t}^{K_{0}}$. In other words, one needs to be careful with the moment loss when these estimates are combined. In fact, we need to simultaneously also bound $\|  \jap{x/t}^{K_{0}} f \|_{L^2_{x,v}} \ls E_{K_{0}}^{\f 12}$ when deriving \eqref{eq:boundedness.main.thm.2} and \eqref{eq:growth.main.thm.2}. Second, the exponential weights $e^{\f{\alp \bar{\bt} |x|^2}{2(\alp + \bar{\bt} t^2)}}$ and $e^{\f 12 \alp |v|^2 + \f 12\bar{\bt} |x-tv|^2}$ are chosen judiciously. As mentioned in issue \eqref{item:issue.1} above, $\alp$ and $\bar{\bt}$ cannot be arbitrary. The choice that $\bar{\bt}$ is small will ensure that we can control the collision term in the presence of these weights. It turns out that these weights are also sufficient in the subsequent arguments. In particular, it is important that for any choice of $\bar{\bt} >0$, the weight $e^{\f{\alp \bar{\bt} |x|^2}{2(\alp + \bar{\bt} t^2)}}$ is comparable to the weight $e^{\f{\alp \bt |x|^2}{2(\alp + \bt t^2)}}$ that naturally occurs in the collisional kernel (see \eqref{eq:coercive.intro}) in the region $\{|x|\leq t^2\}$.

\subsubsection{Estimates for $Lf$}

To obtain the remaining parts of the theorem, the key is to control $Lf$. Since $L\Pi f = 0$, the term $Lf$ only depends on the microscopic part of $f$; proving decay of $Lf$ is therefore consistent with non-decay of the macroscopic part. Once we prove sufficiently strong decay for $Lf$, we obtain the other parts of the theorem as follows.
\begin{enumerate}
    \item Recalling the equation \eqref{eq:linear.intro}, the integrable-in-time decay of $Lf$ implies that as $t\to \infty$, the solution behaves like a solution to the transport equation. This proves part \eqref{item:part.4} of Theorem~\ref{thm:main}. (To show part \eqref{item:part.5} of the theorem, i.e., the non-vanishing of the limit, we prescribe data at large $t_0$ corresponding to non-trivial macroscopic part but trivial microscopic part. Using the proven estimates we can deduce that the limit is also non-trivial if $t_0$ is sufficiently large.)
    \item We prove elliptic estimates that control $(I-\Pi)f$ by $Lf$. The decay of $Lf$ then implies the decay of $(I-\Pi)f$, proving parts \eqref{item:part.2} and \eqref{item:part.3} of Theorem~\ref{thm:main}.
\end{enumerate}

The relevant energies that we will control are 
$$\calE_{L,1}(t):= t \langle e^{\f{\alp \bar{\bt}|x|^2}{\alp + \bar{\bt} t^2}} \jap{x/t}^{2(K_{0}-2)} f, Lf \rangle_{L^2_{x,v}},\quad \calE_{L,2}:= t^2 \| e^{\f{\alp \bar{\bt}|x|^2}{2(\alp + \bar{\bt} t^2)}} \jap{x/t}^{K_{0}-3} Lf\|_{L^2_{x,v}}^2.$$

Modulo commutator terms (which are important but to be discussed below), the main contribution in 
\begin{align}
    \f{\ud}{\ud t} \calE_{L,1}(t) = &\:\langle e^{\f{\alp \bar{\bt}|x|^2}{\alp + \bar{\bt} t^2}} \jap{x/t}^{2(K_{0}-2)} f, Lf \rangle_{L^2_{x,v}} - 2t\| e^{\f{\alp \bar{\bt}|x|^2}{2(\alp + \bar{\bt} t^2)}} \jap{x/t}^{(K_{0}-2)}  Lf \|_{L^2_{x,v}}^2 + \cdots ,\label{eq:EL.1} \\
    \f{\ud}{\ud t} \calE_{L,2}(t) = &\:\ 2t \| e^{\f{\alp \bar{\bt}|x|^2}{2(\alp + \bar{\bt} t^2)}} \jap{x/t}^{K_{0}-3}  Lf \|_{L^2_{x,v}}^2 - 2t^2\langle e^{\f{\alp \bar{\bt}|x|^2}{\alp + \bar{\bt} t^2}} \jap{x/t}^{2(K_{0}-3)} Lf, L^2f \rangle_{L^2_{x,v}} + \cdots . \label{eq:EL.2}
\end{align}
In both \eqref{eq:EL.1} and \eqref{eq:EL.2}, the second terms have a favorable sign; for \eqref{eq:EL.2}, the sign comes from \eqref{eq:coercive.intro} (with $f$ replaced by $e^{\f{\alp \bar{\bt}|x|^2}{2(\alp + \bar{\bt} t^2)}} \jap{x/t}^{K_{0}-3}  Lf$). In \eqref{eq:EL.1}, the first term can be bounded by the analogue of \eqref{eq:energy.intro} with $e^{\f{\alp \bar{\bt}|x|^2}{2(\alp + \bar{\bt} t^2)}} \jap{x/t}^{(K_{0}-2)}$ weights and is thus already controlled by the arguments in Section~\ref{sec:intro.transport.est}. On the other hand, in \eqref{eq:EL.2}, the first term can be controlled by the good term in \eqref{eq:EL.1}. Thus, as long as we control the error terms, we can bound $\calE_{L,1}(t)$ and $\calE_{L,2}(t)$ to prove decay for $Lf$. Notice that the good term on the right-hand side of \eqref{eq:EL.2} gives that $Lf$ is, at least in an $L^2_t$-averaged sense, integrable in time, justifying what we said above.

The main difficulty in this strategy is to control the error terms. For this, we need to (a) carefully decompose the solution into the macroscopic and microscopic parts and (b) control the commutator $[L,T]$. The decomposition is necessary in view of the fact that the macroscopic parts do not decay. For the commutator $[L,T]$, we will be able to show that it is perturbative, but among other things, this relies on stronger transport estimates. In particular, we will need to first control $Y_l f$ for $Y_l = t \rd_{x_l} + \rd_{v_l}$ and to show analogues of \eqref{eq:boundedness.main.thm.2} and \eqref{eq:growth.main.thm.2} with $f$ replaced by $Y_l f$. (We observe that since $[T,Y_l] = 0$, for the linear transport equation, the norm \eqref{eq:transport.norm} is also conserved for $f_{\mathrm{lin}}$ replaced by $Y_lf_{\mathrm{lin}}$.) In this process, we will in turn need to control the commutator $[Y_l, L]$. It turns out that all of the commutator terms involved either have sufficient decay, or otherwise can be controlled by the microscopic terms, allowing one to finally bound $\calE_{L,1}$ and $\calE_{L,2}$.

\subsection{Related works}
In this section, we give a non-exhaustive list of works on the Landau equation and the Boltzmann equation. Global regularity remains a wide open problem except for spatially homogeneous data. For general large data, there are local results as well as results on continuation criteria. There is also a large body of literature on the construction of global solutions which are perturbations of global or traveling Maxwellians, or of vacuum.

\subsubsection{The spatially homogeneous case}

In the spatially homogeneous case, global regularity has been recently proven for very soft potentials in the breakthrough works of Guillen--Silvestre \cite{GuSi2025} for Landau and Imbert--Silvestre--Villani \cite{ImSiVi2026} for noncutoff Boltzmann. The case of other values of $\gamma$ has been understood earlier; see \cite{DeVi00, DeVi00.2, mGnG2019, Si17, Wu14} for the Landau equation and \cite{lA1972.1, lA1972.2, tC1957, lDcM2009, nF2006, nFcM2009, lbH2012, gTcV1999, sMbW1999} for the Boltzmann equation.

\subsubsection{Local existence and continuation criteria for general data} Surprisingly, the local existence for the Landau equation with general data is quite recent \cite{HeSn17,HeSnTa17, cHsSaT2019,sC2023}.

For spatially inhomogeneous solutions, there has been a lot of activity in proving regularity of solutions to the Landau equation assuming a priori pointwise control of hydrodynamic quantities (the mass density, energy density and entropy density); see \cite{CaSiSn18,GoImMoVa16, ImSi2021, ImSi2022, OuSi2024}. Particular types of singularities, such as self-similar singularities and hydrodynamic implosion type singularities, have been ruled out by Bedrossian--Gualdani--Snelson \cite{jBmpGsS2022} (for $\gamma \in [-3, -2]$) and Golding--Henderson--Silvestre \cite{GoHeSi2026} (for $\gamma \in [-3,1]$), respectively. The latter work also obtained a refined continuation criterion based on weighted control for hydrodynamic quantities. We note also that hydrodynamic implosion type singularities have been recently constructed for ``very hard'' potentials ($\gamma\in(\sqrt{3},2]$) \cite{bedrossian2026finite}.

\subsubsection{Stability of global Maxwellians}\label{sec:stability.glo.Max} This is a dissipation-dominated regime. The problem of stability of global Maxwellians has a rich history with initial works focusing on the spatially homogeneous regimes. The first result regarding stability of these equilibria for cutoff Boltzmann in spatially inhomogeneous settings goes back to work of Ukai \cite{Ukai74}.

Stability of Maxwellians for a long-range model was first proven in the seminal work of Guo \cite{Guo02}. He used a nonlinear energy method to prove global stability of Maxwellians for the Landau equation on a torus. The result has been extended and generalized in various directions, including other collisional models \cite{AMUXY12.3,AMUXY12,AMUXY12.2,jwJrmS2022,GrSt11, Guo02.2,Guo03.2,Guo03,Guo12,StGu04,StGu06,StGu08,Str12}, more general function spaces (including very low-regularity settings) \cite{CaMi17, CaTrWu17, CaTrWuErratum17, rjDsqLsSrmS2021, wGmpGaL2024, cHsSaT2025}, and weakly collisional settings \cite{BeCoDo2024, BeZhZi2025, VPL}.

\subsubsection{Stability of vacuum and traveling Maxwellians} This is a dispersion dominated regime. Illner--Shinbrot, in \cite{IlSh84}, were the first ones to prove stability of vacuum result for cutoff Boltzmann with hard spheres interaction. Toscani in \cite{To87} proved that one can find data near vacuum for cutoff Boltzmann with hard potentials such that the global solutions converge neither to a traveling Maxwellian nor back to vacuum. For more results on stability of vacuum for cutoff Boltzmann see \cite{BeTo85, Ha85, ToBe84, Po88, To86, Guo01,Ar11, HeJi17}. 

For long-range models, the second author in \cite{jL2019} proved stability of vacuum for Landau equations with moderately soft potentials. The proof proceeds via an energy approach combined with a vector field method and a special null structure of the nonlinear collisional term. Since then the first author proved stability of vacuum for Landau equations with hard potentials in \cite{sC2020a} and for Boltzmann equations with moderately soft potentials in \cite{sC2020b}.

The stability of traveling Maxwellians for moderately soft potentials and hard potentials follow a similar story. Dispersion is the main mechanism and one can find data such that the solution does not converge back to the traveling Maxwellian. In fact, a whole scattering theory is known for the cutoff Boltzmann case thanks to \cite{BaGaGoLe16}. See also \cite{AlGa09, Go97, To88} for more results pertaining to stability of traveling Maxwellians for cutoff Boltzmann. Recently, He--Li in \cite{HeLi2023}, proved a similar result in the case of noncutoff Boltzmann with moderately soft potentials. We emphasize that the situation is quite different at least linearly for very soft potentials and that is the focus of this paper.

\subsubsection{Commuting vector field method in kinetic models} The commuting vector field method has been widely used to capture dispersion in collisionless models including Vlasov--Poisson, Einstein--Vlasov and Vlasov--Maxwell systems in \cite{Bi17,FaJoSm17.1, Sm16, Wa18.1, Wa18.2, Wa18.3, Wo18, hLmT20, Ta17}. The first time the vector field method was used for a collisional model was in the work of the second author in \cite{jL2019} and has been used for stability of vacuum in works of first author \cite{sC2020a, sC2020b}. 

In \cite{VPL}, the two authors with Nguyen, married Guo's energy method with hypocoercivity and the commuting vector field method to simultaneously understand dissipation, hypocoercivity and Landau damping. Many papers including two by the authors have used this method to capture Landau damping and phase mixing since; see for example \cite{BeCoDo2024, BeZhZi2025, sCjL2022, sCjL2026}. 

\subsection{Outline of the paper} The remainder of the paper is structured as follows. In \textbf{Section~\ref{sec:notations}}, we first introduce the necessary notations. In \textbf{Section~\ref{sec:collision}}, we prove estimates for the Landau collisional kernel. In \textbf{Section~\ref{sec:eng-est-lin}}, we prove the basic energy estimates. In \textbf{Section~\ref{sec:commutator}}, we estimate the various commutators. 

The next few sections prove the main estimates. In \textbf{Section~\ref{sec:lower.order.energy}}, we prove the estimates for $f$ and $Yf$. In \textbf{Section~\ref{sec:L.energy}}, we prove energy estimates involving $Lf$. Using these bounds, we prove in \textbf{Section~\ref{sec:limit}} that a $t\to \infty$ limit exists and is in general non-trivial. Finally, in \textbf{Section~\ref{sec:decay}}, we use elliptic estimates to prove decay of the microscopic part of the solution.

\subsection*{Acknowledgement} SC acknowledges support by the Simons Foundation Award 1141490. JL gratefully acknowledges the support of a National Science Foundation Grant under DMS-2304445. 

\subsection*{AI Disclosure} We used Anthropic Claude to proofread earlier versions of the paper. The paper is completely written by the authors and the mathematical content is the authors' own.

\section{Notations}\label{sec:notations}

We introduce the notations we use in the remainder of the paper.


\subsection*{The linear operator}

The linear operator $L$ can be written as 
\begin{equation}
    L = - A - K
\end{equation}
as in \eqref{eq:L}. We further introduce the notations 
\begin{equation}\label{eq:sigma}
\begin{split}
\sigma_{ij}=\phi_{ij}*\mu,\quad
\sigma_i =\sigma_{ij}z_j,
\end{split}
\end{equation}
with
\begin{equation}\label{eq:def-z}
z:=\sqrt{\alpha+\beta t^2}v - \f{x\beta t}{\sqrt{\alpha+\beta t^2}}.
\end{equation}
The notation is important as the linear operator $A$ can be written as
\begin{equation}
    A g=\rd_{v_i}(\sigma_{ij}\rd_{v_j}g)+\sqrt{\alpha+\beta t^2}(\rd_{v_i}\sigma_i) g-(\alpha+\beta t^2)\sigma_{ij}z_iz_j g,
\end{equation}
which will be proven in Lemma~\ref{lem:A}. For $z$ in \eqref{eq:def-z}, we will also denote the projection to the direction of $z$ as follows
\begin{equation}\label{eq:proj-spacetime}
P_{z} g:= \f{[z\cdot g]z}{|z|^2}.
\end{equation}
This will be relevant for the analysis of the linear operator; see for instance Lemma~\ref{lem:sigma-bounds}.

\subsection*{Vector fields}

The following vector fields will play important roles:
\begin{align}
T := &\:\rd_t + v_i \rd_{x_i} \\
Y_l := &\: t \rd_{x_l} + \rd_{v_l},\quad l=1,2,3.
\end{align}

We will also use the following norms for vector fields:
\begin{equation}
    |\rd_v g|^2:= \sum_{i=1}^3 |\rd_{v_i} g|^2,\quad |Y g|^2 := \sum_{i=1}^3 |Y_i g|^2.
\end{equation}

\subsection*{Weight functions}

A key ingredient of our proof is to use a hierarchy of weight functions. We define
\begin{equation}\label{eq:M}
w_{(\vartheta,\kay,\iota)}^2(t,x,v) = \jap{\frac{|x|}{t}}^{2\kay} \exp(B t^{-0.1}) \exp(\vartheta (\alpha|v|^2+ \bar{\bt}(t)|x-tv|^2)) \exp\Big(\frac{\iota \alpha\bar{\bt}(t) |x|^2}{\alpha+\bar{\bt}(t) t^2}\Big).
\end{equation}
Here, $\kay \in \mathbb Z_{\geq 0}$, $\vartheta \in \{0,1\}$ and $\iota \in \{0,1\}$. Moreover, only one of $\vartheta$ and $\iota$ will be chosen to be $=1$. The parameters $\alp$, $\bt$ are as in \eqref{eq:mu}, while $B$ is large (to be fixed at the end of Section~\ref{sec:eng-est-lin}) and 
\begin{equation}
    \bar{\bt}(t)=\beta_0(1+t^{-0.1}),
\end{equation}
where $\bt_0 \in (0, \bt/6)$ will be chosen to be small (to be fixed after the proof of Proposition~\ref{prop:weighted-A-bound}). Notice in particular that $\bar{\bt} < \f 13\bt$.

Moreover, it is often useful to only focus on part of the weights and define
\begin{equation}\label{eq:wv}
w_v^2(t,x,v) = \exp( \alpha|v|^2+ \bar{\bt}(t)|x-tv|^2).
\end{equation}

For convenience, we will often write $w_{(\vartheta,\kay,\iota)}^2 = w_{(\vartheta,\kay,\iota)}^2(t,x,v)$, $w_v^2 = w_v^2(t,x,v)$ and $\bar{\bt} = \bar{\bt}(t)$.

\subsection*{Macroscopic and microscopic parts}
\begin{definition}\label{def:macro-quant}
Define $\Pi g$ to be the orthogonal projection of $g$ to $\mathrm{span}\{ \sqrt{\mu}, z_i \sqrt{\mu}, |z|^2 \sqrt{\mu}\}$ with respect to $\ud v$. Denote
$$\Pi g=a^g(t,x) \sqrt{\mu}+b_i^g(t,x) z_i \sqrt{\mu}+c^g(t,x)|z|^2\sqrt{\mu}.$$
\end{definition}

\subsection*{Norms}

The collisional operator is naturally associated with an anisotropic dissipation norm. We introduce both a weighted and an unweighted version.
\begin{definition}\label{def:dissipation-norm}
Define
\begin{align*}
\norm{g}_{\Delta_v(0)}^2&=  e^{-\f{\alpha \beta |x|^2}{\alpha+\beta t^2}}\left[ t^{-2}\norm{\jap{z}^{-3/2} P_{z}\rd_{v}g}^2_{L^2_v} + \norm{\jap{z}^{-1/2}\Big(t^{-1} |(I-P_{z})\rd_{v}g| + |g|\Big)}^2_{L^2_v} \right]
\end{align*}
and
\begin{align*}
\norm{g}_{\Delta_v(1)}^2&=  e^{-\f{\alpha \beta |x|^2}{\alpha+\beta t^2}}\left[t^{-2}\norm{w_v\jap{z}^{-3/2} P_{z}\rd_{v}g }^2_{L^2_v} + \norm{w_v\jap{z}^{-1/2} \Big(t^{-1} |(I-P_{z})\rd_{v}g| + |g|\Big)}^2_{L^2_v}\right],
\end{align*}
where $z$ is as in \eqref{eq:def-z}. We will also denote $\Delta_v = \Delta_v(0)$.
\end{definition}

We define also the following norms in both $x$ and $v$.
\begin{definition}\label{def:calD}
    Define 
    \begin{equation}
        \| g\|_{\calD_{x,v}(\vartheta,\kay,\iota)} := \Big\| \| w_{(\vartheta,\kay,\iota)} g \|_{\Delta_v(0)} \Big\|_{L^2_x}.
    \end{equation}
\end{definition}

\subsection*{Initial energy}
For any $K \in \mathbb Z_{\geq 3}$, define the initial energy by
\begin{equation}\label{eq:EK.def}
\begin{split}
    E_K = &\: \sum_{|\om|\leq 1} \Big(\| w_{(0,K-|\om|,0)} Y^\om f_{t_0} \|_{L^2_{x,v}}^2 + t_0^{-1.1} \| w_{(1,K-|\om|,0)} Y^\om f_{t_0} \|_{L^2_{x,v}}^2 + \| w_{(0,K-|\om|-1,1)} Y^\om f_{t_0} \|_{L^2_{x,v}}^2 \Big) \\
    &\: + \Big( t_0 \langle w^2_{(0,K-2,1)} f_{t_0} , L f_{t_0} \rangle_{L^2_{x,v}} + t_0^2 \|w_{(0,K-3,1)} Lf_{t_0} \|_{L^2_{x,v}}^2\Big).
\end{split}
\end{equation}

\section{Landau collisional operator}\label{sec:collision}

From this section onwards, we begin the proof of the main theorem. In particular, we will from now on assume that $t \geq t_0 \geq 1$. 

In this section, unless explicitly stated otherwise, all implicit constants will be allowed to depend on $\alp$, $\bt$ and $\bar{\bt}$.

\subsection{Representation formulas}

We now prove representation formulas for $L$. In analyzing the collision operator, it is convenient to introduce the notation $g_* = g(v_*)$ and $z_* = \sqrt{\alp + \bt t^2} v_* - \f{x \bt t}{\sqrt{\alp + \bt t^2}}$.
\begin{lemma}\label{lem:A}
The operator $A$ takes the following form
\begin{equation}\label{eq:A}
\begin{split}
A g&=\rd_{v_i}(\sigma_{ij}\rd_{v_j}g)+\sqrt{\alpha+\beta t^2}(\rd_{v_i}\sigma_i) g-(\alpha+\beta t^2)\sigma_{ij}z_iz_j g,
\end{split}
\end{equation}
\end{lemma}
\begin{proof}
Recall from \eqref{eq:L} that $A g=\mu^{-1/2}Q(\mu,\mu^{1/2}g)$.
To proceed, we use that 
$$\mu = e^{-\alpha|v|^2}e^{-\beta |x-tv|^2}=e^{-\f{\alpha\beta |x|^2}{\alpha+\beta t^2}}e^{-|z|^2}.$$

Using the definition, that the factor $e^{-\f{\alpha\beta |x|^2}{\alpha+\beta t^2}}$ is independent of $v$ and that $$\rd_{v_l} z_m = \sqrt{\alp + \bt t^2} \de_{lm},\quad  \phi_{ij}(u)u_j=0,$$ we get
\begin{align*}
A g&=e^{-\f{\alpha\beta |x|^2}{\alpha+\beta t^2}}e^{|z|^2/2}\rd_{v_i}\left\{[\phi_{ij}*e^{-|z|^2}]e^{-|z|^2/2}\left[\rd_{v_j}g-\sqrt{\alpha+\beta t^2}z_j g\right]\right\}\\
&\qquad +2e^{-\f{\alpha\beta |x|^2}{\alpha+\beta t^2}}\sqrt{\alpha+\beta t^2}e^{|z|^2/2}\rd_{v_i}\left\{[\phi_{ij}*(z_j e^{-|z|^2})]e^{-|z|^2/2} g\right\} \\
&=e^{-\f{\alpha\beta |x|^2}{\alpha+\beta t^2}}e^{|z|^2/2}\rd_{v_i}\left\{[\phi_{ij}*e^{-|z|^2}]e^{-|z|^2/2}\left[\rd_{v_j}g-\sqrt{\alpha+\beta t^2}z_j g\right]\right\}\\
&\qquad +2e^{-\f{\alpha\beta |x|^2}{\alpha+\beta t^2}}\sqrt{\alpha+\beta t^2}e^{|z|^2/2}\rd_{v_i}\left\{[\phi_{ij}* e^{-|z|^2}]e^{-|z|^2/2}z_j g\right\} \\
&=e^{-\f{\alpha\beta |x|^2}{\alpha+\beta t^2}}e^{|z|^2/2}\rd_{v_i}\left\{[\phi_{ij}*e^{-|z|^2}]e^{-|z|^2/2}\left[\rd_{v_j}g+\sqrt{\alpha+\beta t^2}z_j g\right]\right\} \\
&=e^{-\f{\alpha\beta |x|^2}{\alpha+\beta t^2}}\left(\rd_{v_i}\{[\phi_{ij}*e^{-|z|^2}]\rd_{v_j}g\}+e^{|z|^2/2}\rd_{v_i}(e^{-|z|^2/2})\{[\phi_{ij}*e^{-|z|^2}]\rd_{v_j}g\}\right)\\
&\qquad +e^{-\f{\alpha\beta |x|^2}{\alpha+\beta t^2}}\sqrt{\alpha+\beta t^2}\left(\rd_{v_i}\left\{[\phi_{ij}*e^{-|z|^2}]z_j g\right\}+e^{|z|^2/2}\rd_{v_i}(e^{-|z|^2/2})\left\{[\phi_{ij}*e^{-|z|^2}]z_j g\right\}\right)\\
&=e^{-\f{\alpha\beta |x|^2}{\alpha+\beta t^2}}\left\{\rd_{v_i}([\phi_{ij}*e^{-|z|^2}]\rd_{v_j}g) + \sqrt{\alpha+\beta t^2}\rd_{v_i}[(\phi_{ij}*e^{-|z|^2}) z_j]g-(\alpha+\beta t^2)(\phi_{ij}*e^{-|z|^2})z_iz_j g\right\}\\
&=\rd_{v_i}(\sigma_{ij}\rd_{v_j}g)+\sqrt{\alpha+\beta t^2}\rd_{v_i}\sigma_i g-(\alpha+\beta t^2)\sigma_{ij}z_iz_j g.
\end{align*} \qedhere
\end{proof}

\begin{lemma}
	The operator $K$ takes the following form:
		\begin{equation}\label{eq:K.representation}
\begin{split}
 Kg = &\: 8\pi  \mu g + 4(\alp + \bt t^2) \mu^{\f 12}   \int_{\R^3}\Big( \phi_{ij}(v-v_*) z_{*i} z_{*j} - |v-v_*|^{-1} \Big) \mu_*^{\f 12} {g}_* \ud v_*.
\end{split}
\end{equation}
\end{lemma}
\begin{proof}
Recalling from \eqref{eq:L} that $Kg = \mu^{-\f 12} Q(\mu^{\f 12} g,\mu)$, and proceeding as in for $A g$, we get
\begin{equation}\label{eq:K}
    \begin{split}
        K g&= - e^{-\f{\alpha\beta |x|^2}{\alpha+\beta t^2}}e^{|z|^2/2}\rd_{v_i}\left\{2\sqrt{\alpha+\beta t^2}z_j e^{-|z|^2} \left[\phi_{ij}*(e^{-|z|^2/2}g)\right]\right\}\\
&\qquad - e^{-\f{\alpha\beta |x|^2}{\alpha+\beta t^2}}e^{|z|^2/2}\rd_{v_i}\left\{  e^{-|z|^2}[\phi_{ij}*(e^{-|z|^2/2}[\rd_{v_j}g-\sqrt{\alpha+\beta t^2}z_j g])] \right\} \\
&=- e^{-\f{\alpha\beta |x|^2}{\alpha+\beta t^2}}e^{|z|^2/2}\rd_{v_i}\left\{ e^{-|z|^2} [\phi_{ij}*(e^{-|z|^2/2}[\rd_{v_j}g+\sqrt{\alpha+\beta t^2}z_j g])] \right\} \\
&=- \mu^{-\f 12}\rd_{v_i}\left\{ \mu \left[\phi_{ij}*(\mu^{\f 12}[\rd_{v_j}g+\sqrt{\alpha+\beta t^2}z_j g])\right]\right\},
    \end{split}
\end{equation}
where we have used $ (v-v_*)_j \phi_{ij}(v- v_*) =0$.
	Expanding the $\rd_{v_i}$ derivative, we thus obtain
\begin{equation}
\begin{split}
 Kg=&\: - \mu^{\f 12}\int_{\R^3}\rd_{v_i} \phi_{ij}(v-v_*)\mu^{\f 12}(v_*) [\rd_{v_{*j}}{g}_*+\sqrt{\alp+\bt t^2} z_{*j}{g}_*]\ud v_*\\
&\:+ 2 \sqrt{\alp + \bt t^2} z_i \mu^{\f 12} \int_{\R^3}\phi_{ij}(v-v_*)\mu^{\f 12}(v_*)[\rd_{v_{*j}}{g}_*+\sqrt{\alp+\bt t^2} z_{*j}{g}_*]\ud v_*.
\end{split}
\end{equation}
Integrating by parts away the $\rd_{v_{*j}}$, and noting that $\rd_{y_i} \rd_{y_j} \phi_{ij}(y) = -8\pi \de$ as distributions (where $\de$ is the Dirac delta), we obtain
\begin{equation}
\begin{split}
 Kg=&\: 8\pi  \mu g +  \mu^{\f 12}\int_{\R^3}\rd_{v_i} \phi_{ij}(v-v_*)\rd_{v_{*j}} \mu^{\f 12}_* {g}_* \ud v_*\\
&\:- \mu^{\f 12}\int_{\R^3}\rd_{v_i} \phi_{ij}(v-v_*)\mu^{\f 12}_* \sqrt{\alp + \bt t^2} z_{*j} {g}_*\ud v_*\\
&\: - 2 \sqrt{\alp + \bt t^2} z_i \mu^{\f 12} \int_{\R^3}\rd_{v_{*j}} \phi_{ij}(v-v_*)\mu^{\f 12}_* {g}_* \ud v_*\\
&\: - 2 \sqrt{\alp + \bt t^2} z_i \mu^{\f 12} \int_{\R^3}\phi_{ij}(v-v_*)(\rd_{v_{*j}} \mu^{\f 12}_*) {g}_* \ud v_*\\
&\:+  2 \sqrt{\alp + \bt t^2} z_i \mu^{\f 12} \int_{\R^3}\phi_{ij}(v-v_*)\mu^{\f 12}_* \sqrt{\alp + \bt t^2} z_{*j} {g}_*\ud v_* \\
=&\: 8\pi  \mu g +2  \sqrt{\alp + \bt t^2} \mu^{\f 12}\int_{\R^3}\rd_{v_i} \phi_{ij}(v-v_*) (z_j - z_{*j}) \mu^{\f 12}_* {g}_* \ud v_* \\
&\: + 4 (\alp + \bt t^2)  \mu^{\f 12} \int_{\R^3}\phi_{ij}(v-v_*)z_{*i}z_{*j} \mu^{\f 12}_* {g}_* \ud v_*.
\end{split}
\end{equation}
Observe that $y_j \rd_{y_i} \phi_{ij}(y) = -\f 2{|y|}$. Hence,
$$\rd_{v_i} \phi_{ij}(v-v_*) (z_j - z_{*j}) = \sqrt{\alp + \bt t^2} \rd_{v_i} \phi_{ij}(v-v_*) (v_j - v_{*j}) = -\f{2\sqrt{\alp + \bt t^2} }{|v-v_*|},$$
which gives the conclusion. \qedhere
\end{proof}

\subsection{Auxiliary results}
We first prove an auxiliary result.
\begin{lemma}\label{lem:int-exp}
Let $\zeta \geq 0$, $a > 0$. Let $h:\mathbb R^3\to [0,\infty)$ be an $L^1_{\mathrm{loc}}$ function such that $h(y) \ls |y|^{-\zeta}$ for $|y|\geq 1$. Then, the following holds for all $c \in \bbR^3$, 
$$ 
\int_{\R^3} h(y) e^{-a|y-c|^2} \ud y \ls \jap{c}^{-\zeta},$$
where the implicit constant may depend on $h$ and $a$ but is independent of $c$. In particular, for $\zeta \in [0,3)$,
$$\int_{\R^3} \f{e^{-a|y-c|^2}}{|y|^\zeta}\ud y \ls \jap{c}^{-\zeta}.$$
\end{lemma}
\begin{proof}
If $|c|\leq 10$, we have $e^{-a|y-c|^2} \leq e^{-\f a2|y|^2+3a|c|^2}$ so that $h(y)e^{-a|y-c|^2} \leq \begin{cases} e^{3a|c|^2} h(y) & \hbox{if $|y|\leq 1$} \\ e^{-\f a2|y|^2+3a|c|^2} & \hbox{if $|y| \geq 1$} \end{cases}$, which clearly implies the lemma. Thus, we focus on the case $|c|\geq 10$.
\begin{align*}
\int_{\R^3} h(y) e^{-a|y-c|^2} \ud y&= \int_{|y|\leq |c|/2} h(y) e^{-a|y-c|^2} \ud y+ \int_{|y| > |c|/2} h(y) e^{-a|y-c|^2} \ud y\\
&\lesssim e^{-\f{a|c|^2}{10}}\int_{|y|\leq |c|/2} h(y) \ud y +|c|^{-\zeta}\int_{|y|> |c|/2} e^{-a|y-c|^2}\ud y\\
&\lesssim e^{-\f{a|c|^2}{10}}|c|^{3}+ |c|^{-\zeta} \lesssim \jap{c}^{-\zeta}.
\end{align*}

\end{proof}

\begin{proposition}\label{prop:aux-integral}
For any $\mathfrak{a}, \mathfrak{b}>0$, we have
\begin{align}
\int_{\R^3}e^{-\mathfrak{a}|v_*|^2}e^{-\mathfrak{b}|x-tv_*|^2}\d v_* =&\:  \pi^{3/2} (\mathfrak{a}+\mathfrak{b}t^2)^{-3/2} e^{-\f{\mathfrak{ab} |x|^2}{\mathfrak{a}+\mathfrak{b} t^2}}, \label{eq:int-no-phi} \\
\int_{\R^3}|v-v_*|^{-1} e^{-\mathfrak{a}|v_*|^2}e^{-\mathfrak{b}|x-tv_*|^2}\ud v_*\lesssim &\: e^{-\f{\mathfrak{ab} |x|^2}{\mathfrak{a}+\mathfrak{b} t^2}}t^{-2}\jap{\sqrt{\mathfrak{a}+\mathfrak{b}t^2}v-\f{x\mathfrak{b}t}{\sqrt{\mathfrak{a}+\mathfrak{b}t^2}}}^{-1}, \label{eq:int-phi} 
\end{align}
where the implicit constants depend only on $\mathfrak{a}, \mathfrak{b}$. Moreover, the same estimates hold even if there are polynomial factors of $\sqrt{\mathfrak{a}+\mathfrak{b}t^2}v_*-\f{x\mathfrak{b}t}{\sqrt{\mathfrak{a}+\mathfrak{b}t^2}}$ in the integral (with implicit constants also depending on the polynomial).
\end{proposition}
\begin{proof}
We first note that 
$$e^{-\mathfrak{a}|v_*|^2}e^{-\mathfrak{b}|x-tv_*|^2}=e^{-\f{\mathfrak{ab}|x|^2}{\mathfrak{a}+\mathfrak{b}t^2}}e^{-\left|\sqrt{\mathfrak{a}+\mathfrak{b}t^2}v_*-\f{x\mathfrak{b}t}{\sqrt{\mathfrak{a}+\mathfrak{b}t^2}}\right|^2}.$$

To prove \eqref{eq:int-no-phi} we use the change of coordinates $y=\sqrt{\mathfrak{a}+\mathfrak{b}t^2}v_*-\f{x\mathfrak{b}t}{\sqrt{\mathfrak{a}+\mathfrak{b}t^2}}$
\begin{equation}\label{eq:int-no-phi.1}
\begin{split}
e^{-\f{\mathfrak{ab}|x|^2}{\mathfrak{a}+\mathfrak{b}t^2}}\int_{\R^3}e^{-\left|\sqrt{\mathfrak{a}+\mathfrak{b}t^2}v_*-\f{x\mathfrak{b}t}{\sqrt{\mathfrak{a}+\mathfrak{b}t^2}}\right|^2}\ud v_*&=(\mathfrak{a}+\mathfrak{b}t^2)^{-3/2}e^{-\f{\mathfrak{ab}|x|^2}{\mathfrak{a}+\mathfrak{b}t^2}}\int_{\R^3}e^{-|y|^2}\ud y\\
&=\pi^{3/2}(\mathfrak{a}+\mathfrak{b}t^2)^{-3/2}e^{-\f{\mathfrak{ab}|x|^2}{\mathfrak{a}+\mathfrak{b}t^2}}.
\end{split}
\end{equation}

For \eqref{eq:int-phi}, we have 
\begin{equation}
\begin{split}
\int_{\R^3}& |v-v_*|^{-1} e^{-\mathfrak{a}|v_*|^2}e^{-\mathfrak{b}|x-tv_*|^2}\d v_*\\
&= e^{-\f{\mathfrak{ab}|x|^2}{\mathfrak{a}+\mathfrak{b}t^2}}\int_{\R^3}e^{-\left|\sqrt{\mathfrak{a}+\mathfrak{b}t^2}v_*-\f{x\mathfrak{b}t}{\sqrt{\mathfrak{a}+\mathfrak{b}t^2}}\right|^2}\cdot\f{1}{|v-v_*|}\ud v_*\\
&=e^{-\f{\mathfrak{ab}|x|^2}{\mathfrak{a}+\mathfrak{b}t^2}}\int_{\R^3}e^{-\left|\sqrt{\mathfrak{a}+\mathfrak{b}t^2}(v_*-v)+\sqrt{\mathfrak{a}+\mathfrak{b}t^2}v-\f{x\mathfrak{b}t}{\sqrt{\mathfrak{a}+\mathfrak{b}t^2}}\right|^2}\cdot\f{1}{|v-v_*|}\ud v_*.
\end{split}
\end{equation}
Now we use the change of coordinates $y=(v_*-v)\sqrt{\mathfrak{a}+\mathfrak{b}t^2}$ and use Lemma~\ref{lem:int-exp} to get \eqref{eq:int-phi} as follows
\begin{equation}\label{eq:int-phi.1}
\begin{split}
e^{-\f{\mathfrak{ab}|x|^2}{\mathfrak{a}+\mathfrak{b}t^2}}&\int_{\R^3}e^{-\left|\sqrt{\mathfrak{a}+\mathfrak{b}t^2}(v_*-v)+\sqrt{\mathfrak{a}+\mathfrak{b}t^2}v-\f{x\mathfrak{b}t}{\sqrt{\mathfrak{a}+\mathfrak{b}t^2}}\right|^2}\cdot\f{1}{|v-v_*|}\ud v_*\\
&=(\mathfrak{a}+\mathfrak{b}t^2)^{-1}e^{-\f{\mathfrak{ab}|x|^2}{\mathfrak{a}+\mathfrak{b}t^2}}\int_{\R^3}e^{-\left|y+\sqrt{\mathfrak{a}+\mathfrak{b}t^2}v-\f{x\mathfrak{b}t}{\sqrt{\mathfrak{a}+\mathfrak{b}t^2}}\right|^2}\cdot\f{1}{|y|}\ud y\\
&\lesssim (\mathfrak{a}+\mathfrak{b}t^2)^{-1}e^{-\f{\mathfrak{ab}|x|^2}{\mathfrak{a}+\mathfrak{b}t^2}}\jap{\f{xt\mathfrak{b}}{\sqrt{\mathfrak{a}+\mathfrak{b}t^2}}-\sqrt{\mathfrak{a}+\mathfrak{b}t^2}v}^{-1}.
\end{split}
\end{equation}

Finally, suppose there are additional polynomial factors of $\mathfrak z_* :=\sqrt{\mathfrak{a}+\mathfrak{b}t^2}v_*-\f{x\mathfrak{b}t}{\sqrt{\mathfrak{a}+\mathfrak{b}t^2}}$, say, we consider $\int_{\R^3}e^{-\mathfrak{a}|v_*|^2}e^{-\mathfrak{b}|x-tv_*|^2} |\mathfrak z_*|^k \d v_*$ or 
$\int_{\R^3}|v-v_*|^{-1} e^{-\mathfrak{a}|v_*|^2}e^{-\mathfrak{b}|x-tv_*|^2} |\mathfrak z_*|^k \d v_*$ instead of \eqref{eq:int-no-phi}, \eqref{eq:int-phi}. Then in \eqref{eq:int-no-phi.1} and \eqref{eq:int-phi.1}, we will instead have integrals
$$\int_{\R^3}e^{-|y|^2} |y|^k \ud y,\quad \int_{\R^3}e^{-\left|y+\sqrt{\mathfrak{a}+\mathfrak{b}t^2}v-\f{x\mathfrak{b}t}{\sqrt{\mathfrak{a}+\mathfrak{b}t^2}}\right|^2} \Big| y+\sqrt{\mathfrak{a}+\mathfrak{b}t^2}v-\f{x\mathfrak{b}t}{\sqrt{\mathfrak{a}+\mathfrak{b}t^2}}\Big|^k \cdot\f{1}{|y|}\ud y,$$
which satisfy the same bound as before and we can dominate them up to a constant by 
$$\int_{\R^3}e^{-\f 12|y|^2} \ud y,\quad \int_{\R^3}e^{-\f 12\left|y+\sqrt{\mathfrak{a}+\mathfrak{b}t^2}v-\f{x\mathfrak{b}t}{\sqrt{\mathfrak{a}+\mathfrak{b}t^2}}\right|^2}  \cdot\f{1}{|y|}\ud y.$$
\qedhere
\end{proof}

\subsection{Preliminaries about weight functions}

\begin{lemma}\label{lem:z.x-tv.com}
The following estimates hold:
\begin{equation}
	\jap{z} \ls  \jap{x-tv} + \jap{\f x{t^2}},\quad  \jap{x-tv} \ls \jap{z} + \jap{\f x{t^2}}.
\end{equation}
In particular,
\begin{equation}\label{eq:z.x-tv.com.restricted}
	\jap{z} \ls  \jap{x-tv} \ls \jap{z} \quad \hbox{whenever $|x| \leq t^2$}.
\end{equation}
\end{lemma}
\begin{proof}
	This is immediate from
	$\sqrt{\alpha+\beta t^2} z + (\f \alp t+ \bt t) (x-tv) = \f{\alp x}{t}.$ \qedhere
\end{proof}

\begin{lemma}\label{lem:weight.compare}
The following holds for any $\kay,\ell \in \mathbb Z_{\geq 0}$:
\begin{equation}
	\jap{z}^\ell w_{(0,\kay,1)} \ls w_{(1,\kay,0)}  \jap{\f{x}{t^2}}^\ell.
\end{equation}
In particular,
\begin{equation}
	\jap{z}^\ell w_{(0,\kay,1)} \ls w_{(1,\kay+\ell,0)}.
\end{equation}
\end{lemma}
\begin{proof}
First observe, by writing out the weight functions, that
\begin{equation}\label{eq:weight.compare.prelim}
	w_{(0,\kay,1)} = w_{(1,\kay,0)} e^{-\f 12 |\sqrt{\alp + \bar{\bt}t^2} v - \f{\bar{\bt} tx}{\sqrt{\alp + \bar{\bt} t^2}}|^2}.
\end{equation}
To put in extra polynomial powers of $\jap{z}^\ell$, observe that 
\begin{equation}
\begin{split}
	&\: \Big|z -  \f{\sqrt{\alp + \bt t^2}}{\sqrt{\alp + \bar{\bt}t^2}}(\sqrt{\alp + \bar{\bt}t^2} v - \f{\bar{\bt} tx}{\sqrt{\alp + \bar{\bt} t^2}})\Big| 
	=  \Big| \f{\bt tx}{\sqrt{\alp + \bt t^2}}- \f{\bar{\bt} \sqrt{\alp + \bt t^2} tx}{\alp + \bar{\bt} t^2}\Big| \\
	= &\: \f{\Big(\bt (\alp + \bar{\bt} t^2) - \bar{\bt} (\alp + \bt t^2) \Big) t x}{\sqrt{\alp + \bt t^2} (\alp + \bar{\bt} t^2) } = \f{(\bt - \bar{\bt}) \alp  t x}{\sqrt{\alp + \bt t^2} (\alp + \bar{\bt} t^2) } \ls \f{|x|}{t^2}.
\end{split}
\end{equation}
Thus 
$$\jap{z} \ls |\sqrt{\alp + \bar{\bt}t^2} v - \f{\bar{\bt} tx}{\sqrt{\alp + \bar{\bt} t^2}}| + \jap{\f{x}{t^2}}.$$
Now, polynomials in $|\sqrt{\alp + \bar{\bt}t^2} v - \f{\bar{\bt} tx}{\sqrt{\alp + \bar{\bt} t^2}}|$ can be dominated by the exponential factor in \eqref{eq:weight.compare.prelim}, and the result follows. \qedhere
\end{proof}

\subsection{Rescaling arguments}

\begin{proposition}\label{prop:normalize}
Let $L_1$ denote the linearized operator about the global Maxwellian $\mu_1 = e^{-|v|^2}$, i.e.,
$$- L_1g=\mu_1^{-1/2}Q(\mu_1, \mu^{1/2}_1 g)+\mu^{-1/2}_1 Q(\mu^{1/2}_1 g,\mu_1).$$
Then 
\begin{equation}
L g(t,x,v) = e^{-\f{\alp\bt|x|^2}{\alp + \bt t^2}} (L_1 \widetilde{g}) \Big(t,x,\sqrt{\alp+\bt t^2} (v- \f{\bt t x}{\alp + \bt t^2})\Big),
\end{equation}
where
\begin{equation}
\widetilde{g}(t,x,w) = g \Big(t,x,\f{w}{\sqrt{\alp+\bt t^2}} + \f{\bt t x}{\alp + \bt t^2} \Big).
\end{equation}
\end{proposition}
\begin{proof}
Write
$$\mu = e^{-\f{\alp\bt|x|^2}{\alp + \bt t^2}} e^{-\left|\sqrt{\alp + \bt t^2}v-\f{\bt t x}{\sqrt{\alp + \bt t^2}}\right|^2}$$
and then carry out a straightforward change of variables. \qedhere
\end{proof}

Using a similar rescaling argument, we also obtain a precise description of $\sigma_{ij}$:
\begin{lemma}\label{lem:sigma-bounds}
The functions $\sigma_{ij}$ defined as in \eqref{eq:sigma} are smooth and satisfy
\begin{equation}\label{eq:sigma.decomposition}
\sigma_{ij}=\f{e^{-\f{\alpha \beta |x|^2}{\alpha+\beta t^2}}}{\alp + \bt t^{2}}\left[\lambda_1(t,x,v)\f{z_iz_j}{|z|^2}+\lambda_2(t,x,v)\left(\delta_{ij}-\f{z_iz_j}{|z|^2}\right)\right]
\end{equation}
where $\lambda_1(t,x,v) = \lambda^1_1(z),\,\lambda_2(t,x,v) = \lambda^1_2(z) >0$ with $\lambda^1_1(v)$, $\lambda^1_2(v)$ given by 
\begin{align}
\lambda^1_1(v) = &\: \int_{\R^3} |u|^{-1} \Big( 1 - \f{u_1^2}{|u|^2} \Big) e^{-((u_1-|v|)^2+ u_2^2+ u_3^2)}  \, \ud u, \label{eq:sigma.decomposition.DL.1} \\
\lambda^1_2(v) = &\: \int_{\R^3} |u|^{-1} \Big( 1 - \f{u_2^2+u_3^2}{2|u|^2} \Big) e^{-((u_1-|v|)^2+ u_2^2+ u_3^2)}  \, \ud u. \label{eq:sigma.decomposition.DL.2}
\end{align}
In particular, $\lambda_1$, $\lambda_2$ are functions of $z$ alone and there exist $c_1, c_2 >0$ such that as $|z|\to \infty$
$$\jap{z}^{3} \lambda_1\to c_{1} ,\qquad \jap{z} \lambda_2\to c_{2} .$$
In particular,
\begin{align}
    \sigma_{ij}g_ig_j=&\:\f{e^{-\f{\alpha \beta |x|^2}{\alpha+\beta t^2}}}{\alp + \bt t^{2}}[\lambda_1(t,x,v) |P_{z}g|^2+\lambda_2(t,x,v) | (I-P_{z})g|^2], \\
    \sigma_{ij} z_i z_j = &\: \f{e^{-\f{\alpha \beta |x|^2}{\alpha+\beta t^2}}}{\alp + \bt t^{2}}\lambda_1(t,x,v)|z|^2.
\end{align}
Moreover, the functions $\sigma_i$ defined as in \eqref{eq:sigma} are smooth and satisfy
\begin{equation}\label{eq:sigma.bound}
    |\rd_{v}^\gamma \sigma_i| \ls e^{-\f{\alpha \beta |x|^2}{\alpha+\beta t^2}}t^{-2+|\gamma|} \jap{z}^{-1-|\gamma|}.
\end{equation}

\end{lemma}

\begin{proof}
Recall from \eqref{eq:phi} and the definition of $\sigma_{ij},$
\begin{align*}
\sigma_{ij}=\int_{\R^3}\left(\delta_{ij}-\f{(v-v_*)_i(v-v_*)_j}{|v-v_*|^2}\right)|v-v_*|^{-1} e^{-\alpha|v_*|^2}e^{-\beta|x-tv_*|^2}\ud v_*.
\end{align*}            
Next we write
$$e^{-\alpha |v_*|^2}e^{-\beta |x-tv_*|^2}=e^{-\f{\alpha\beta |x|^2}{\alpha+\beta t^2}}e^{-\left|\sqrt{\alpha+\beta t^2}(v_*-v)+\sqrt{\alpha+\beta t^2}v-\f{x\beta t}{\sqrt{\alpha+\beta t^2}}\right|^2}.$$
Then, using the change of coordinates $w=\sqrt{\alpha+\beta t^2}(v_*-v)$, we get 
\begin{equation}\label{eq:sigma-integral}
\sigma_{ij}=\f{1}{\alpha+\beta t^2}e^{-\f{\alpha\beta |x|^2}{\alpha+\beta t^2}}\int_{\R^3} \left(\delta_{ij}-\f{w_iw_j}{|w|^2}\right)|w|^{-1} e^{-\left|w-\left(\f{x\beta t}{\sqrt{\alpha+\beta t^2}}-\sqrt{\alpha+\beta t^2}v\right)\right|^2}\ud w.
\end{equation}

We now compare this with the corresponding quantity in the normalized case
\begin{equation}\label{eq:sigma1-integral}
    \begin{split}
        \sigma_{ij}^1:= &\: \int_{\R^3}\left(\delta_{ij}-\f{(v-v_*)_i(v-v_*)_j}{|v-v_*|^2}\right)|v-v_*|^{-1} e^{-|v_*|^2}\ud v_* \\
        = &\: \int_{\R^3}\left(\delta_{ij}-\f{w_iw_j}{|w|^2}\right)|w|^{-1} e^{-|w+v|^2}\ud w.
    \end{split}
\end{equation}
Therefore, by comparing \eqref{eq:sigma-integral} with \eqref{eq:sigma1-integral}, we obtain the representation formula \eqref{eq:sigma.decomposition} after using the formulas near the end of \cite[Lemma~4]{StGu08} (and taking into account the different normalization constants). The needed estimates then follow from those in \cite[Lemma~4]{StGu08}. Finally, \eqref{eq:sigma.bound} can be proven with a similar rescaling argument after using \cite[Lemma~4]{StGu08}; we omit the details. \qedhere

\end{proof}

In view of Lemma~\ref{lem:sigma-bounds}, it is easy to check that the norms in Definition~\ref{def:dissipation-norm} can be equivalently defined as follows (after also using Poincar\'e's inequality on a finite $z$ region):
\begin{lemma}\label{lem:dissipation.norm.equivalence}
    Define
    \begin{align}
        \| g\|_{\tilde{\Delta}_v(0)}^2 = &\: \int_{\R^3} \sigma_{ij} \rd_{v_i} g \rd_{v_j} g \ud v + (\alp + \bt t^2) \int_{\R^3} \sigma_{ij} z_i z_j g^2 \ud v, \\
        \| g\|_{\tilde{\Delta}_v(1)}^2 = &\: \int_{\R^3} w_v^2\sigma_{ij} \rd_{v_i} g \rd_{v_j} g \ud v + (\alp + \bt t^2) \int_{\R^3} w_v^2\sigma_{ij} z_i z_j g^2 \ud v.
    \end{align}
    Then, for $\vartheta \in \{0,1\}$,
    \begin{equation}
        \| g\|_{\tilde{\Delta}_v(\vartheta)} \ls \| g\|_{\Delta_v(\vartheta)} \ls \| g\|_{\tilde{\Delta}_v(\vartheta)}.
    \end{equation}
\end{lemma}

Next we give a lower bound for $\int_{\R^3}g Lg\ud v $.
\begin{lemma}\label{lem:lower-bound-L}
There is a constant $\delta_0 >0$ such that 
\begin{align*}
\int_{\R^3} g\cdot Lg\ud v\geq \delta_0 \| (I-\Pi) g\|_{\Delta_v}^2.
\end{align*}
\end{lemma}
\begin{proof}
By the result in \cite[Theorem~1.2]{cMrmS07}, the corresponding estimate holds for the normalized operator $L_1$ (see Proposition~\ref{prop:normalize}). 
Applying Proposition~\ref{prop:normalize}, we then obtain the desired bound. \qedhere
\end{proof}

\subsection{Weighted estimates for the collisional operators}

In this section we prove weighted $L^2_v$ estimates for the linear collisional operator. Up to this point we have been estimating the collisional operator after transforming to a problem for $L_1$. If we take this approach, the natural weights to introduce should involve, e.g., factors of $e^{\f{\alp\bt|x|^2}{\alp+\bt t^2}}$ and $e^{\de\left| \f{\bt t x}{\sqrt{\alp+\bt t^2}} - \sqrt{\alp + \bt t^2} v\right|^2}$. However, these weights do not adapt to the transport operator and lead to a loss of moments. We have instead introduced weights (recall \eqref{eq:M}) that are adapted to the transport operator and for this reason we need additional new weighted estimates for the collision operators.

We will begin with weighted estimates for $K$. We will in fact prove a slightly stronger estimate than is needed for $K$ in Proposition~\ref{prop:main.prelim.for.K}, which will be useful later for the commutator estimates. The key estimate will be given in Proposition~\ref{prop:weighted-K-bound}.

Before we proceed, recall the Hardy--Littlewood--Sobolev inequality.
\begin{lemma}[Hardy--Littlewood--Sobolev inequality]\label{lem:HLS}
Suppose $\f 2p + \f{\lambda}3 = 2$ with $\lambda \in [0,3)$. Then 
$$\Big| \int_{\R^3} \int_{\R^3} \f{h_1(v) h_2(v_*)}{|v-v_*|^\lambda} \, \ud v \ud v_* \Big|\ls \| h_1\|_{L^p} \| h_2\|_{L^p} .$$
\end{lemma}

\begin{proposition}\label{prop:main.prelim.for.K}
	Define
    \begin{equation}
        w = \begin{cases}
            w_v  & \hbox{if $\vartheta =1$} \\
            1 & \hbox{if $\vartheta =0$}.
        \end{cases}
    \end{equation} For $\ell_1,\ell_2 \in \mathbb Z_{\geq 0}$, define
	\begin{align}
\calJ_1 = &\: \int_{\R^3} w^2|z|^{\ell_2} \mu |g_1 g_2| \, \ud v, \label{eq:calJ1.def}\\
\calJ_2 = &\: t^2 \int_{\R^3}\int_{\R^3} |v-v_*|^{-1} w^2 \mu^{\f 12} \mu^{\f 12}_* |z_*|^{\ell_1} |z|^{\ell_2} |{g_1}_* g_2| \,  \ud v_* \ud v. \label{eq:calJ3.def}
\end{align}
	Then for any $\ell \in \mathbb Z_{\geq 0}$, the following holds for any sufficiently regular functions $g_1$, $g_2$ with implicit constants depending on $\ell_1$, $\ell_2$, $\ell$:
	\begin{enumerate}
\item If $\vartheta = 1$, then
\begin{equation}\label{eq:weighted-K-1}
\begin{split}
\sum_{n=1}^2 \calJ_n \ls &\: \norm{g_1}_{L^2_v} \norm{g_2}_{L^2_v} + t^{-2} \norm{wg_1}_{L^2_v} \norm{wg_2}_{L^2_v}.
\end{split}
\end{equation}
\item If $\vartheta = 0$, then 
\begin{equation}\label{eq:weighted-K-0}
\begin{split}
\sum_{n=1}^2 \calJ_n 
\ls &\:  e^{-\f{\alp \bt |x|^2}{\alp + \bt t^2}} \norm{\jap{z}^{-\ell} g_1}_{L^2_v} \norm{\jap{z}^{-\ell} g_2}_{L^2_v}.
\end{split}
\end{equation}
\end{enumerate}	
\end{proposition}
\begin{proof}

We begin with the $\vartheta = 1$ case in Steps~1 and 2. We then turn to the $\vartheta = 0$ case in Step~3.

\pfstep{Step~1: First bound of the left-hand side of \eqref{eq:weighted-K-1}} We now consider the $\vartheta = 1$ case and so $w = w_v$. Our main goal in this step is to establish
\begin{equation}\label{eq:K.est.sharper}
\sum_{n=1}^2 \calJ_n \ls \jap{x/t^2}^{\ell_1+\ell_2} \Big\| e^{-\alp|v|^2-\f{ (\beta+\bar{\bt})}{8}|x-tv|^2} w g_1 \Big\|_{L^2_v} \Big\| e^{-\f{(\beta-\bar{\bt})}{8}|x-tv|^2} w g_2 \Big\|_{L^{2}_v}.
\end{equation}
In Step~2 we will show that this is sufficient for our desired conclusion.

We first look at $\calJ_1$ in \eqref{eq:calJ1.def}. Using $\mu = e^{-\alp|v|^2 - \bt|x-tv|^2}$ and Lemma~\ref{lem:z.x-tv.com}, we bound $\calJ_1$ as follows:
\begin{equation}\label{eq:K.est.J1}
\begin{split}
\calJ_1 = &\: \int_{\R^3} w^2|z|^{\ell_2} \mu |g_1 g_2| \, \ud v \leq \jap{x/t^2}^{\ell_2} \| e^{-\alp|v|^2- \f \bt 2 |x-tv|^2} w g_1 \|_{L^2_v} \| e^{- \f \bt 4 |x-tv|^2} w g_2 \|_{L^2_v},
\end{split}
\end{equation}
where we used a bit of the exponentially decaying factor to dominate $\jap{x-tv}^{\ell_2}$. The bound \eqref{eq:K.est.J1} is better than needed in \eqref{eq:K.est.sharper} (after recalling that $\bar{\bt} \in (0,\bt)$). 

For the term $\calJ_2$ in \eqref{eq:calJ3.def}, observe that
\begin{equation}\label{eq:K.est.weights.manipulation}
w^2 |z_*|^{\ell_1} |z|^{\ell_2} \mu^{\f 12} \mu^{\f 12}_* = w w_* |z_*|^{\ell_1} |z|^{\ell_2}  e^{-\f{(\beta-\bar{\bt})}{2}|x-tv|^2} e^{-\alpha|v_*|^2}e^{-\f{(\beta+\bar{\bt})}{2}|x-tv_*|^2}.
\end{equation}
We use Lemma~\ref{lem:z.x-tv.com} to obtain
\begin{align}
e^{-\f{(\beta-\bar{\bt})}{2}|x-tv|^2} |z|^{\ell_2}  \ls&\: \jap{x/t^2}^{\ell_2} e^{-\f{(\beta-\bar{\bt})}{4}|x-tv|^2},\label{eq:K.est.weights.manipulation.1}\\
e^{-\alpha|v_*|^2}e^{-\f{(\beta+\bar{\bt})}{2}|x-tv_*|^2} |z_*|^{\ell_1} \ls &\: \jap{x/t^2}^{\ell_1} e^{-\alp |v_*|^2 - \f{(\beta+\bar{\bt})}{4}|x-tv_*|^2}. \label{eq:K.est.weights.manipulation.2}
\end{align}

Hence, after controlling the weights, we have
$$\calJ_2 \ls \calJ_2',$$
where
\begin{align}
\calJ_2' = &\: t^2 \jap{x/t^2}^{\ell_1+\ell_2} \int_{\R^3}\int_{\R^3} |v-v_*|^{-1}e^{-\f{(\beta-\bar{\bt})}{4}|x-tv|^2}e^{-\alp|v_*|^2-\f{(\beta+\bar{\bt})}{4}|x-tv_*|^2} |w_* {g_1}_*  w g_2| \,  \ud v_* \ud v.
\end{align}
By the Hardy--Littlewood--Sobolev inequality (Lemma~\ref{lem:HLS} with $p = \f 65$ and $\lambda = 1$), we have
\begin{equation*}
\begin{split}
\calJ_2' \ls &\:  t^2\jap{x/t^2}^{\ell_1+\ell_2}\Big\| e^{-\alp|v|^2-\f{(\beta+\bar{\bt})}{4}|x-tv|^2} w g_1 \Big\|_{L^{\f 65}_v} \Big\| e^{-\f{(\beta-\bar{\bt})}{4}|x-tv|^2} w g_2 \Big\|_{L^{\f 65}_v} \\
\ls &\: t^2\jap{x/t^2}^{\ell_1+\ell_2} \Big\| e^{-\f{(\beta-\bar{\bt})}{8}|x-tv|^2} \Big\|_{L^3_v} \Big\| e^{-\f{(\beta+\bar{\bt})}{8}|x-tv|^2}  \Big\|_{L^3_v}\Big\| e^{- \alp|v|^2-\f{ (\beta+\bar{\bt})}{8}|x-tv|^2} w g_1 \Big\|_{L^2_v} \Big\| e^{-\f{(\beta-\bar{\bt})}{8}|x-tv|^2} w g_2 \Big\|_{L^{2}_v}\\
\ls &\: \jap{x/t^2}^{\ell_1+\ell_2} \Big\| e^{- \alp |v|^2-\f{ (\beta+\bar{\bt})}{8}|x-tv|^2} w g_1 \Big\|_{L^2_v} \Big\| e^{-\f{(\beta-\bar{\bt})}{8}|x-tv|^2} w g_2 \Big\|_{L^{2}_v},
\end{split}
\end{equation*}
where we used that $\| e^{-\f{(\beta-\bar{\bt})}{8}|x-tv|^2} \|_{L^3_v} ,\,\| e^{-\f{(\beta+\bar{\bt})}{8}|x-tv|^2} \|_{L^3_v} \ls t^{-1}$.

Combining the bounds for $\calJ_1$ and $\calJ_2$, we have thus proven \eqref{eq:K.est.sharper}.


 \pfstep{Step~2: Proof of \eqref{eq:weighted-K-1}} Our remaining goal is to show that the right-hand side of \eqref{eq:K.est.sharper} can be bounded by the right-hand side of \eqref{eq:weighted-K-1}. For this purpose, we split into two cases $|x| \geq t^2$ and $|x| \leq t^2$. 
 
 If $|x| \geq t^2$, then we can use
 $$e^{-\alp|v|^2-\f{ (\beta+\bar{\bt})}{8}|x-tv|^2} = e^{-\f{\alp(\beta+\bar{\bt})|x|^2}{8\alp+ (\beta+\bar{\bt})t^2}} e^{-|\sqrt{\alp + \f 18(\beta+\bar{\bt})t^2} v - \f{\f 18(\beta+\bar{\bt})xt}{\sqrt{\alp + \f 18(\beta+\bar{\bt})t^2}}|^2} \leq e^{-\f{c|x|^2}{t^2}}$$
for some $c>0$ depending on $\alp$, $\bt$, $\bar{\bt}$. Hence,
\begin{equation}\label{eq:K.est.large.x.main}
\begin{split}
&\: \jap{x/t^2}^{\ell_1+\ell_2} \Big\| e^{-\alp|v|^2-\f{ (\beta+\bar{\bt})}{8}|x-tv|^2} w g_1 \Big\|_{L^2_v} \Big\| e^{-\f{(\beta-\bar{\bt})}{4}|x-tv|^2} w g_2 \Big\|_{L^{2}_v} 
\ls  \jap{x/t^2}^{\ell_1+\ell_2} e^{-\f{c|x|^2}{t^2}} \Big\|  w g_1 \Big\|_{L^2_v} \Big\| w g_2 \Big\|_{L^{2}_v} .
\end{split}
\end{equation}
Now observe that for $|x| \geq t^2$, we have 
$$\jap{x/t^2}^{\ell_1+\ell_2} e^{-\f{c|x|^2}{t^2}} \ls  e^{-\f{c|x|^2}{2 t^2}} (\sup_{m \leq \ell_1+\ell_2} \jap{y}^m e^{-\f{c|y|^2}{2}}) \ls e^{-\f{ct^2}{2}} \ls \brk{t}^{-2}.$$
Plugging this back into \eqref{eq:K.est.large.x.main}, we show that the term can be bounded above by 
$$\hbox{\eqref{eq:K.est.large.x.main}} \ls t^{-2}\Big\|  w g_1 \Big\|_{L^2_v} \Big\| w g_2 \Big\|_{L^{2}_v} ,$$
which is acceptable.

We then turn to the $|x| \leq t^2$ case. We see that the $\jap{x/t^2}^{\ell_1+\ell_2}$ factor is bounded. 
We now study the weights in \eqref{eq:K.est.sharper}. First, observe that for any $\mathfrak b >0$, the following holds when $|x|\leq t^2$
\begin{equation}\label{eq:general.weight.computation.1}
    \begin{split}
        e^{\f{\alp \bt |x|^2}{2(\alp + \bt t^2)}} e^{-\f 12 \alp |v|^2 - \mathfrak b |x-tv|^2} = &\: e^{\f{\alp \bt |x|^2}{2(\alp + \bt t^2)}} e^{-\f{\alp \mathfrak b |x|^2}{\alp + 2\mathfrak b t^2}} e^{-\f 12 |\sqrt{\alp + 2\mathfrak b t^2} v - \f{2x\mathfrak b t}{\sqrt{\alp + 2\mathfrak b t^2}}|^2} \\
        = &\: e^{\f{\alp^2 (\bt - 2\mathfrak b) |x|^2 }{2(\alp + 2\mathfrak b t^2)(\alp + \bt t^2)}} e^{-\f 12|\sqrt{\alp + 2\mathfrak b t^2} v - \f{2x\mathfrak b t}{\sqrt{\alp + 2\mathfrak b t^2}}|^2} \ls_{\mathfrak b} 1,
    \end{split}
\end{equation}
which implies
\begin{equation}
\begin{split}
e^{\f{\alp\bt |x|^2}{2(\alp + \bt t^2)}} e^{-\alp|v|^2-\f{ (\beta+\bar{\bt})}{8}|x-tv|^2} w =&\:  e^{\f{\alp\bt |x|^2}{2(\alp + \bt t^2)}} e^{-\f \alp 2|v|^2 -(\f{\bt}{8}- \f{3\bar{\bt}}{8})|x-tv|^2} \ls 1.
\end{split}
\end{equation}
For the weight $e^{-\f{\alp\bt |x|^2}{2(\alp + \bt t^2)}} e^{-\f{(\beta-\bar{\bt})}{4}|x-tv|^2} w$ on $g_2$, we compute that 
\begin{equation}\label{eq:annoying.weight}
\begin{split}
\log \Big( e^{-\f{\alp\bt |x|^2}{2(\alp + \bt t^2)}} e^{-\f{(\beta-\bar{\bt})}{4}|x-tv|^2} w \Big) = &\: \f \alp 2|v|^2 - \f{\alp \bt |x|^2}{2(\alp + \bt t^2)} - \f{\bt - \bar{\bt}}{4}|x-tv|^2 + \f {\bar{\bt}}2 |x-tv|^2 \\
= &\: \Big( \f{\alp}{2t^2} - \f{\bt - 3\bar{\bt}}{4}\Big) |x-tv|^2 + \f{\alp x\cdot (vt-x)}{t^2} + \f{\alp|x|^2}{2t^2}  - \f{\alp \bt |x|^2}{2(\alp + \bt t^2)} \\
= &\: \Big( \f{\alp}{2t^2} - \f{\bt - 3\bar{\bt}}{4}\Big) |x-tv|^2 + \f{\alp x\cdot (vt-x)}{t^2} +\f{\alp^2  |x|^2}{2t^2(\alp + \bt t^2)}.
\end{split}
\end{equation}
Since $|x|/t^2\leq 1$, for $t_*\geq \sqrt{\f{4 \alp}{\bt - 3\bar{\bt}}}$, we can complete the square to deduce that $\hbox{\eqref{eq:annoying.weight}} \ls 1$ when $t \geq t_*$. On the other hand, when $t \leq t_*$, we have $1\ls t^{-2}$. It follows that 
\begin{equation}\label{eq:stupid.weights}
    e^{-\f{\alp\bt |x|^2}{2(\alp + \bt t^2)}} e^{-\f{(\beta-\bar{\bt})}{4}|x-tv|^2} w \ls 1 + t^{-2} w.
\end{equation}

Altogether, when $|x| \leq t^2$, 
\begin{align}
e^{\f{\alp\bt |x|^2}{2(\alp + \bt t^2)}} \Big\| e^{-\alp|v|^2-\f{ (\beta+\bar{\bt})}{8}|x-tv|^2} w g_1 \Big\|_{L^2_v}  \ls &\: \|  g_1 \|_{L^2_v}, \\
e^{-\f{\alp\bt |x|^2}{2(\alp + \bt t^2)}} \Big\| e^{-\f{(\beta-\bar{\bt})}{4}|x-tv|^2} w g_2 \Big\|_{L^2_v} \ls  &\: \|  g_2 \|_{L^2_v} + t^{-2} \|w g_2\|_{L^2_v}
\end{align} 
so that the right-hand side of \eqref{eq:K.est.sharper} also obeys acceptable bounds.

\pfstep{Step~3: Proof of \eqref{eq:weighted-K-0}} We now consider the $\vartheta = 0$ case, which is simpler. Writing $\mu^{\f 12} = e^{-\f{\alp\bt |x|^2}{2(\alp + \bt t^2)}} e^{-|z|^2/2}$, we have
$$\calJ_1 \ls e^{-\f{\alp\bt |x|^2}{\alp + \bt t^2}} \int_{\R^3} w^2  \jap{z}^{-2\ell} |g_1 g_2| \, \ud v,$$
which gives the desired bound after using Cauchy--Schwarz. For $\calJ_2$, we observe that 
\begin{equation}
|z_*|^{\ell_1} |z|^{\ell_2} \jap{z}^{\ell} \jap{z_*}^{\ell} \mu^{\f 12} \mu^{\f 12}_* 
= |z_*|^{\ell_1} |z|^{\ell_2} \jap{z}^{\ell} \jap{z_*}^{\ell} e^{-\f{\alp \bt |x|^2}{\alp + \bt t^2}} e^{-|z|^2/2} e^{-|z_*|^2/2}.
\end{equation} 
so that 
\begin{equation}
    \begin{split}
        \calJ_2 \ls &\: t^2 e^{-\f{\alp \bt |x|^2}{\alp + \bt t^2}} \int_{\R^3}\int_{\R^3} |v-v_*|^{-1} e^{-|z|^2/4} e^{-|z_*|^2/4} |\jap{z_*}^{-\ell} {g_1}_* \jap{z}^{-\ell} g_2| \,  \ud v_* \ud v.
    \end{split}
\end{equation}
Using the Hardy--Littlewood--Sobolev inequality (Lemma~\ref{lem:HLS} with $p = \f 65$ and $\lambda = 1$), and the fact that $\| e^{-|z|^2/4} \|_{L^3_v} \ls t^{-1}$, we obtain the desired bound. \qedhere
\end{proof}

We apply the previous proposition to the specific case of the operator $K$. Recalling \eqref{eq:K.representation}, the previous proposition implies
\begin{proposition}\label{prop:weighted-K-bound}
Let $\ell \in \mathbb Z_{\geq 0}$. Then the following holds for implicit constants depending on $\ell$:
\begin{align}
 \left|\langle w_v^2 Kg_1, g_2 \rangle_{L^2_{v}} \right| \ls &\: \norm{g_1}_{L^2_v} \norm{g_2}_{L^2_v} + t^{-2} \norm{w_v g_1}_{L^2_v} \norm{w_v g_2}_{L^2_v}, \label{eq:K.upper.bound.weighted} \\
 \left|\langle  Kg_1, g_2 \rangle_{L^2_{v}} \right| \ls &\: e^{-\f{\alp \bt |x|^2}{\alp + \bt t^2}} \norm{\jap{z}^{-\ell} g_1}_{L^2_v} \norm{\jap{z}^{-\ell} g_2}_{L^2_v}.\label{eq:K.upper.bound.unweighted}
\end{align}
\end{proposition}

We now turn to the estimates for $A$. It will be convenient to introduce a further decomposition
\begin{equation}\label{eq:A.decomposition}
    A = \bar{A} + \widetilde{A}, 
\end{equation}
where 
\begin{equation}
    \bar{A} = \rd_{v_i}(\sigma_{ij}\rd_{v_j}g)-(\alpha+\beta t^2)\sigma_{ij}z_iz_j g,\quad \widetilde{A} = \sqrt{\alpha+\beta t^2}\rd_{v_i}\sigma_i g.
\end{equation}

We first bound $\widetilde{A}$.
\begin{proposition}\label{prop:sigma-i-est}
For any $\eta>0$ small, there exists a $C(\eta)>0$ such that the following holds for any sufficiently regular functions $g_1$, $g_2$:

\begin{align}
    \left|\langle w_v^2 \widetilde{A} g_1, g_2 \rangle_{L^2_{v}} \right|
    \leq &\:  \Big(\eta^{\f 12} \norm{g_1}_{\Delta_v(1)} + C(\eta) \norm{g_1}_{L^2_v}\Big) \Big(\eta^{\f 12} \norm{g_2}_{\Delta_v(1)} + C(\eta) \norm{g_2}_{L^2_v}\Big) \nonumber \\
    &\: +  t^{-2} (\norm{w_v g_1}_{L^2_v}^2 + \norm{w_v g_2}_{L^2_v}^2) \label{eq:tildeA.1} \\
    \left|\langle  \widetilde{A} g_1, g_2 \rangle_{L^2_{v}} \right|
    \ls &\: e^{-\f{\alp \bt |x|^2}{\alp + \bt t^2}} \langle \jap{z}^{-2} |g_1|, |g_2| \rangle_{L^2_v}  . \label{eq:tildeA.2}
\end{align}

\end{proposition}
\begin{proof}
By Lemma~\ref{lem:sigma-bounds}, we have for $w = w_v$ or $w=1$,
\begin{equation}\label{eq:tildeA.CS}
\begin{split}
\left|\langle w^2 \widetilde{A} g_1, g_2 \rangle_{L^2_{v}} \right| \ls t \left|\int_{\R^3}w^2 \rd_{v_i}\sigma_i g_1 \cdot g_2\ud v\right|\lesssim &\: e^{-\f{\alpha\beta |x|^2}{\alpha+\beta t^2}} \int_{\R^3} \jap{z}^{-2} w^2|g_1|\cdot |g_2|\ud v .
\end{split}
\end{equation}
This immediately gives \eqref{eq:tildeA.2}. We consider \eqref{eq:tildeA.1} for the remainder of the proof. 
The key now is to notice that the $\|\cdot \|_{\Delta_v(1)}$ norm controls $e^{-\f{\alpha\beta |x|^2}{2(\alpha+\beta t^2)}}  \| \jap{z}^{-\f 12} w_v \cdot \|_{L^2_v}$ and so the discrepancy of $\jap{z}^{-\f 12}$ power gives the needed smallness. 

To proceed, we split into the cases $|x|\geq t^2$ and $|x|\leq t^2$. If $|x| \geq t^2$, we have $e^{-\f{\alpha\beta |x|^2}{\alpha+\beta t^2}} \ls t^{-2}$ (and $\brk{z}^{-1} \ls 1$), so that the term is acceptable.

Consider now the case $|x|\leq t^2$. We split the domain of integration into $|x-tv|\leq D$ and $|x-tv|\geq D$ for $D\geq 1$ sufficiently large to be chosen, i.e., 
\begin{equation}\label{eq:tildeA.large.x}
\begin{split}
&\: e^{-\f{\alpha\beta |x|^2}{\alpha+\beta t^2}} \int_{\R^3} \jap{z}^{-2} w^2|g_1|\cdot |g_2|\ud v  \\
\ls &\: e^{-\f{\alpha\beta |x|^2}{\alpha+\beta t^2}} \Big(\| \jap{z}^{-1} w_v g_1 \|_{L^2_v(|x-tv| \geq D)} + \| \jap{z}^{-1} w_v g_1 \|_{L^2_v(|x-tv| \leq D)} \Big) \\
&\: \qquad \quad \times \Big( \| \jap{z}^{-1} w_v g_2 \|_{L^2_v(|x-tv| \geq D)} + \| \jap{z}^{-1} w_v g_2 \|_{L^2_v(|x-tv| \leq D)} \Big).
\end{split}
\end{equation}
For $|x-tv| \geq D$, we have by \eqref{eq:z.x-tv.com.restricted} that $\jap{z}^{-\f 12} \ls \jap{x-tv}^{-\f 12} \ls D^{-\f 12}$. Hence,  we have
\begin{equation}\label{eq:tildeA.small.x.1}
e^{-\f{\alpha\beta |x|^2}{2(\alpha+\beta t^2)}} \| \jap{z}^{-1} w_v g_i \|_{L^2_v(|x-tv| \geq D)} \ls D^{-1/2} \norm{g_i}_{\Delta_v(1)}\quad i=1,2.
\end{equation}
For the region $|x-tv| \leq D$, observe $e^{-\f{\alpha\beta |x|^2}{2(\alpha+\beta t^2)}} w_v \leq e^{-\f{\alpha\beta |x|^2}{2(\alpha+\beta t^2)}} e^{-\f{\bt-\bar{\bt}}4 |x-tv|^2} w_v e^{\f{\bt-\bar{\bt}}4 D^2}$ so that by \eqref{eq:stupid.weights}, we have
\begin{equation}\label{eq:tildeA.small.x.2}
e^{-\f{\alpha\beta |x|^2}{2(\alpha+\beta t^2)}} \| \jap{z}^{-1} w_v g_i \|_{L^2_v(|x-tv| \leq D)} \ls e^{\f{\bt-\bar{\bt}}4 D^2} \Big( \norm{g_i}_{L^2_v} + t^{-2} \norm{w_v g_i}_{L^2_v} \Big) \quad i=1,2.
\end{equation}
Plugging \eqref{eq:tildeA.large.x}, \eqref{eq:tildeA.small.x.1}, \eqref{eq:tildeA.small.x.2} into \eqref{eq:tildeA.CS}, and choosing $D$ sufficiently large, we obtain \eqref{eq:tildeA.1}. \qedhere

\end{proof}
Next we prove weighted estimates for the term involving $Ag.$
\begin{proposition}\label{prop:weighted-A-bound}
For $\bt_0>0$ sufficiently small, the following holds for some $\kappa>0$:
\begin{align}
-\jap{w_v^2 Ag, g }_{L^2_v} \geq &\:  \kappa \norm{g}_{\Delta_v(1)}^2 - C \| g\|_{L^2_v}^2 -  Ct^{-2} \norm{w_v g}_{L^2_v}^2.
\end{align}
\end{proposition}
\begin{proof}
Since we need to capture the smallness of $\bt_0$, in this proof, we allow our implicit constants to depend on $\alp$, $\bt$ but not $\bar{\bt}$.

We first consider the term $- \langle w_v^2 g, \rd_{v_i} (\sigma_{ij} \rd_{v_j} g) \rangle_{L^2_v}$.
\begin{equation}
\begin{split}
	&\: - \langle w_v^2 g, \rd_{v_i} (\sigma_{ij} \rd_{v_j} g) \rangle_{L^2_v} = -\langle e^{\alp |v|^2 + \bar{\bt} |x-tv|^2} g, \rd_{v_i} (\sigma_{ij} \rd_{v_j} g) \rangle_{L^2_v} \\
	= &\: \langle e^{\alp |v|^2 + \bar{\bt} |x-tv|^2} \sigma_{ij} \rd_{v_i} g,  \rd_{v_j} g \rangle_{L^2_v} - 2 \bar{\bt} t \langle e^{\alp |v|^2 + \bar{\bt} |x-tv|^2} \sigma_{ij} (x-tv)_i g,  \rd_{v_j} g \rangle_{L^2_v} \\
	&\: + 2 \alp \langle e^{\alp |v|^2 + \bar{\bt} |x-tv|^2} \sigma_{ij} v_i g,  \rd_{v_j} g \rangle_{L^2_v} =: I + II + III.
\end{split}
\end{equation}

$I$ is the main term to be kept. For $II$, $III$, we use that
$$\sigma_{ij} (x-tv)_i = -\f{t \sigma_{ij} z_i}{\sqrt{\alp+\bt t^2}} + \f{\alp \sigma_{ij} x_i}{\alp + \bt t^2} ,\quad \sigma_{ij} v_i =\f{\sigma_{ij} z_i}{\sqrt{\alp + \bt t^2}} + \f{\bt t^2}{\alp + \bt t^2} \f{\sigma_{ij} x_i}{t}.$$
Moreover, by \eqref{eq:sigma.decomposition},
$$\sigma_{ij} z_i \rd_{v_j} g = \f{e^{-\f{\alp \bt |x|^2}{\alp + \bt t^2}}}{\alp + \bt t^2}  \lambda_1 |z| |P_z \rd_{v} g|.$$
From these we obtain
\begin{equation}
\begin{split}
&\: |II| + |III| \\
\ls &\: e^{-\f{\alp \bt |x|^2}{\alp + \bt t^2}} \Big( \bar{\bt} t^{-1} \| w_v^2 \jap{z}^{-2} g P_z \rd_v g\|_{L^1_v} + t^{-3} \| w_v^2 \jap{z}^{-2} g P_z \rd_v g\|_{L^1_v} \\
&\: +  t^{-2} \| w_v^2 \jap{z}^{-3} g (P_z \f xt) P_z \rd_v g\|_{L^1_v} + t^{-2} \| w_v^2 \jap{z}^{-1} g ((I-P_z) \f xt) (I-P_z) \rd_v g\|_{L^1_v} \Big) \\
 =: &\: IV + V + VI + VII.
\end{split}
\end{equation}

The terms $V$, $VI$, $VII$ do not have smallness, but are compensated by additional decay, which we will exploit. The terms $VI$ and $VII$ are slightly harder due to the extra $\f xt$ weights; we consider $VI$ as an example. Using Cauchy--Schwarz and Young's inequalities,
\begin{equation}
\begin{split}
|VI| = &\: e^{-\f{\alp \bt |x|^2}{\alp + \bt t^2}} t^{-2} \| w_v^2 \jap{z}^{-3} g (P_z \f xt) P_z \rd_v g\|_{L^1_v} \\
\leq&\: e^{-\f{\alp \bt |x|^2}{\alp + \bt t^2}} \f{|x|}{t} t^{-2}  \| w_v g \|_{L^2_v} \| w_v \jap{z}^{-\f 32} P_z \rd_v g \|_{L^2_v} \\
\leq &\: \eta e^{-\f{\alp \bt |x|^2}{\alp + \bt t^2}}  t^{-2} \| w_v \jap{z}^{-\f 32} P_z \rd_v g \|_{L^2_v}^2  + C(\eta) e^{-\f{\alp \bt |x|^2}{\alp + \bt t^2}} \f{|x|^2}{t^2} t^{-2} \| w_v g \|_{L^2_v}^2.
\end{split}
\end{equation}
The first term $\ls \eta \| g \|_{\Delta_v(1)}^2$, while for the second term, since $e^{-\f{\alp \bt |x|^2}{\alp + \bt t^2}} \f{|x|^2}{t^2} \ls 1$, the term is bounded by $C(\eta) t^{-2} \| w_v g \|_{L^2_v}^2$. The terms $V$, $VII$ can be treated similarly. Thus, altogether, we have 
\begin{equation}\label{eq:barA.lower.bound.derivative}
\begin{split}
	-\int_{\R^3}w_v^2  g \rd_{v_i} (\sigma_{ij} \rd_{v_j} g) \ud v \geq \kappa_0 \int_{\R^3} w_v^2 \sigma_{ij} \rd_{v_i} g \rd_{v_j} g \ud v - C (\bar{\bt} + \eta) \| g \|_{\Delta_v(1)}^2 - C(\eta) t^{-2} \| w_v g \|_{L^2_v}^2.
\end{split}
\end{equation}
Recalling the definition of $\bar{A}$ and the equivalence of norms in Lemma~\ref{lem:dissipation.norm.equivalence}, and choosing $\bt_0$ (hence $\bar{\bt}$) and $\eta$ small depending on $\kappa_0$ in \eqref{eq:barA.lower.bound.derivative}, it thus follows that 
\begin{equation}\label{eq:barA.lower.bound}
\begin{split}
	-\langle w_v^2  g, \bar{A} g\rangle_{L^2_v} \geq \kappa_1 \| g \|_{\Delta_v(1)}^2 - C(\eta) t^{-2} \| w_v g \|_{L^2_v}^2.
\end{split}
\end{equation} 
Recalling now the definition of $A$ in \eqref{eq:A.decomposition}, we combine \eqref{eq:barA.lower.bound} with \eqref{eq:tildeA.1} and choose $\eta$ in \eqref{eq:tildeA.1} to be sufficiently small to obtain the desired conclusion. \qedhere
\end{proof}
From now on we fix $\bt_0$ (and notice that it depends only on $\alp$ and $\bt$).

Combining Proposition~\ref{prop:weighted-A-bound} with the bound \eqref{eq:K.upper.bound.weighted} for $K$, we immediately obtain the following lower bound for $L$ in a weighted space.
\begin{corollary}\label{cor:weighted-L-bound}
There exists a $\kappa>0$ such that the following holds:
\begin{align*}
\langle  w_v^2 Lg, g \rangle_{L^2_v} &\geq  \kappa \norm{g}_{\Delta_v(1)}^2 - C  \norm{g}_{L^2_v}^2 - C t^{-2} \norm{w_v g}_{L^2_v}^2.
\end{align*}
\end{corollary}
We will also need upper bounds for $A$.
\begin{proposition}\label{prop:weighted-A-upper-bound}
The following estimates hold:
\begin{align}
     \left|\jap{ w_v^2 A g_1, g_2 }_{L^2_v} \right|
\ls &\: \norm{g_1}_{\Delta_v(1)}\norm{g_2}_{\Delta_v(1)} + t^{-2} \norm{w_v g_1}_{L^2_v}\norm{w_v g_2}_{L^2_v}, \label{eq:A.upper.bound.1}\\
    \left|\jap{ A g_1, g_2 }_{L^2_v} \right| 
\ls &\:  e^{-\f{\alp \bt |x|^2}{\alp + \bt t^2}} \Big( \jap{ \jap{z}^{-1}  |g_1|, |g_2|}_{L^2_v}  +  t^{-2}  \jap{ \jap{z}^{-3} | P_z \rd_v g_1|, | P_z \rd_v g_2 |}_{L^2_v} \nonumber \\
&\: \qquad \qquad +  t^{-2} \jap{ \jap{z}^{-1} |(I - P_z) \rd_v g_1|, |(I - P_z) \rd_v g_2|}_{L^2_v}  \Big).\label{eq:A.upper.bound.2}
\end{align}
\end{proposition}
\begin{proof}
    We begin with \eqref{eq:A.upper.bound.2}. For $\bar{A}$, the bound is immediate after integrating by parts in $v$ and using \eqref{eq:sigma.decomposition}. For $\widetilde{A}$, the upper bound follows from \eqref{eq:tildeA.2}.

    We now turn to \eqref{eq:A.upper.bound.1}. A similar argument as in the proof of Proposition~\ref{prop:weighted-A-bound} gives 
    \begin{equation}\label{eq:A.upper.bound.proof.1}
        \Big| \langle w_v^2  g_2, \bar{A} g_1\rangle_{L^2_v} \Big| \ls  \| g_1 \|_{\Delta_v(1)}\| g_2 \|_{\Delta_v(1)} + t^{-2} \| w_v g_1 \|_{L^2_v} \| w_v g_2 \|_{L^2_v}.
    \end{equation}
    For $\widetilde{A}$, we use \eqref{eq:tildeA.CS} to deduce that
    \begin{equation}\label{eq:A.upper.bound.proof.2}
        \Big| \langle w_v^2  g_2, \widetilde{A} g_1\rangle_{L^2_v} \Big| \ls  \| g_1 \|_{\Delta_v(1)}\| g_2 \|_{\Delta_v(1)}.
    \end{equation}
    Combining \eqref{eq:A.upper.bound.proof.1}--\eqref{eq:A.upper.bound.proof.2} yields the desired conclusion. \qedhere
\end{proof}

\section{Energy estimates for the linearized equations}\label{sec:eng-est-lin}

We will consider the following linear equation with a source:
\begin{equation}\label{eq:lin-with-source}
\partial_t h+v_i\partial_{x_i}h+L h=\mathfrak{S}.
\end{equation}
In this section we allow our implicit constants to depend on $\alp$, $\bt$ and $K_{0}$, but not on $B$. (Note that since $\bt_0$ is already fixed depending on $\alp$ and $\bt$, the constants can also depend on $\bt_0$.) It will be important to use the largeness of $B$ to close our estimates.

\subsection{Preliminary weighted estimate for the linearized equation}

\begin{lemma}\label{lem:general-eng-estimate}
Let $h$ satisfy \eqref{eq:lin-with-source}. Define $w =w_{(\vartheta,\kay,\iota)}$ as in \eqref{eq:M} with $\kay\in \mathbb Z_{\geq 0}$, and $\iota,\, \vartheta \in \{0,1\}$. 
Then the following estimate holds for any $B\geq 1$:
\begin{align*}
&\: \f{\ud}{\ud t} \norm{w h}_{L^2_{x,v}}^2 + B t^{-1.1} \norm{w h}_{L^2_{x,v}}^2 + \vartheta t^{-1.1} \norm{|x-tv| w h}_{L^2_{x,v}}^2+\int_{\R^3}\int_{\R^3}w^2 hLh\ud v\ud x\\
\lesssim &\: \kay^2 B^{-1} t^{-2.9}\norm{\jap{x/t}^{-1} |x-tv| w h}_{L^2_{x,v}}^2 + \iota B^{-1}  t^{-2.9}\norm{\jap{x/t} |x-tv| w h}_{L^2_{x,v}}^2 \\
&\: + \iota  t^{-3} \norm{\jap{x/t} w h}_{L^2_{x,v}}^2 +\Big| \int_{\R^3}\int_{\R^3} w^2h \mathfrak{S}\ud v\ud x \Big|.
\end{align*}

\end{lemma}
\begin{proof}
We multiply the equation \eqref{eq:lin-with-source} by $w^2 h$ to get
\begin{align*}
\frac{1}{2}(\rd_t+v_i\rd_{x_i})(w^2h^2)-\frac{1}{2}(\rd_t w^2+v_i\rd_{x_i}w^2)  h^2+w^2hLh=w^2\mathfrak{S}\cdot h.
\end{align*}
Now we note that for $T=\rd_t+v_i\rd_{x_i}$, we have $T v_j = T (x-tv)_j =0$ and thus
\begin{equation}\label{eq:Tw.exact.1}
-Tw^2=\Bigg( \f B{10} t^{-1.1} -\vartheta(\rd_t \bar{\bt}(t))|x-tv|^2 -2\kay T(|x|/t)\f{|x|}{t}\jap{x/t}^{-2}- \iota T\left(\frac{\alpha\bar{\bt}(t) |x|^2}{\alpha+\bar{\bt}(t) t^2}\right)\Bigg)w^2.
\end{equation}
Since 
\begin{align}
-\rd_t \bar{\bt}(t)=&\: - \beta_0\rd_t (1+t^{-0.1})=0.1 \beta_0 t^{-1.1}, \label{eq:Tw.exact.2}\\
-T(|x|/t)= &\: -(\rd_t +\frac{x_i}{t}\rd_{x_i})\Big(\f{|x|}{t} \Big)+\f{(x-tv)_i}{t}\partial_{x_i} \Big(\f{|x|}{t} \Big) =\f{(x-tv)_i x_i}{|x|t^2}, \label{eq:Tw.exact.3}
\end{align}
and for $\iota\in\{0,1\},$
\begin{equation}\label{eq:Tw.exact.4}
    \begin{split}
        &\: -\iota  T\left(\frac{\alpha\bar{\bt}(t) |x|^2}{\alpha+\bar{\bt}(t) t^2}\right) \\
=&\: -\iota(\rd_t +\frac{x_i}{t}\rd_{x_i})\left(\frac{\alpha\bar{\bt}(t) |x|^2}{\alpha+\bar{\bt}(t) t^2}\right)+\iota \f{(x-tv)_i}{t}\partial_{x_i}\left(\frac{\alpha\bar{\bt}(t) |x|^2}{\alpha+\bar{\bt}(t) t^2}\right)\\
= &\: -\f{2\iota \alp \bar{\bt}(t) |x|^2}{t(\alpha+\bar{\bt}(t)t^2) } \Big( 1 - \f{\bar{\bt}(t) t^2}{\alpha+\bar{\bt}(t)t^2}\Big) +\f{\iota \alp\bar{\bt}(t)  \bar{\bt}'(t)|x|^2 t^2 }{(\alpha+\bar{\bt}(t)t^2)^2} -\f{\iota \alp \bar{\bt}'(t) |x|^2}{\alpha+\bar{\bt}(t)t^2} + \f{2\iota \alpha\bar{\bt}(t)(x-tv)_ix_i}{t(\alpha+\bar{\bt}(t)t^2)}\\
=&\: -\f{2\iota \alp^2 \bar{\bt}(t) |x|^2}{t(\alpha+\bar{\bt}(t)t^2)^2 }  - \f{\iota \alp^2 \bar{\bt}'(t) |x|^2}{(\alpha+\bar{\bt}(t)t^2)^2 } + \f{2\iota \alpha\bar{\bt}(t)(x-tv)_ix_i}{t(\alpha+\bar{\bt}(t)t^2)}.
    \end{split}
\end{equation}

Using these observations and integrating in $x,v$, we get that 
\begin{equation}\label{eq:basic.EE.diff}
\begin{split}
&\: \f{\ud}{\ud t} \norm{wh}_{L^2_{x,v}}^2 + B t^{-1.1} \norm{w h}_{L^2_{x,v}}^2  + \vartheta t^{-1.1} \norm{|x-tv| w h}_{L^2_{x,v}}^2+\int_{\R^3}\int_{\R^3}w^2hLh\ud v\ud x\\
\lesssim&\: \int_{\R^3}\int_{\R^3} t^{-2} |x-tv| \Big( \kay \jap{x/t}^{-1} + \iota  \jap{x/t}\Big) w^2h^2\ud v\ud x+\iota \int_{\R^3}\int_{\R^3} t^{-3}\jap{x/t}^{2} w^2h^2\ud v\ud x\\
&\: \quad+ \Big| \int_{\R^3}\int_{\R^3} w^2 h \mathfrak{S}\ud v\ud x \Big|,
\end{split}
\end{equation}
where we dropped the non-negative term $- \f{\iota \alp^2 \bar{\bt}'(t) |x|^2}{(\alpha+\bar{\bt}(t)t^2)^2 }w^2 h^2$.

Using Young's inequality, for every $\eta >0$, the terms on the first line of the right-hand side of \eqref{eq:basic.EE.diff} can be controlled by
\begin{equation}
\begin{split}
&\: \eta B t^{-1.1}\norm{w h}_{L^2_{x,v}}^2 + C\eta^{-1} B^{-1} \kay^2  t^{-2.9}\norm{\jap{x/t}^{-1} |x-tv|w h}_{L^2_{x,v}}^2 \\
&\: + C\eta^{-1} B^{-1} \iota t^{-2.9}\norm{\jap{x/t} |x-tv| w h}_{L^2_{x,v}}^2 + C \iota t^{-3} \norm{\jap{x/t} w h}_{L^2_{x,v}}^2.
\end{split}
\end{equation}
Choosing $\eta>0$ sufficiently small, the first term can be absorbed by the second term on the left-hand side of \eqref{eq:basic.EE.diff}. Now fixing $\eta>0$ and putting all these together gives the desired estimate. \qedhere
\end{proof}

Now we will use Lemma~\ref{lem:general-eng-estimate} with different parameters. We already observe that $\kay \neq 0$ or $\iota \neq 0$ creates extra error terms on the right-hand side, while $\vartheta \neq 0$ creates difficulties in using the term $\langle w^2 h, Lh \rangle_{L^2_{x,v}}$ on the left-hand side.

We start with the case of $\iota=\vartheta=0.$
\begin{proposition}\label{prop:pure-poly-x-weights}
Let $h$ satisfy \eqref{eq:lin-with-source} then for $\kay \in \mathbb Z_{\geq 0}$ and $\iota=\vartheta=0$, we have the following estimate: 
\begin{equation}\label{eq:pure-poly-x-weights}
\begin{split}
&\: \f{\ud}{\ud t} \norm{w_{(0,\kay,0)} h}_{L^2_{x,v}}^2+\norm{ (I-\Pi)h}^2_{\calD_{x,v}(0,\kay,0)} 
+B t^{-1.1} \norm{w_{(0,\kay,0)} h}_{L^2_{x,v}}^2\\
\lesssim &\: B^{-1} t^{-2.9} \norm{w_{(1,\kay,0)} h}_{L^2_{x,v}}^2+ \Big| \int_{\R^3}\int_{\R^3} w_{(0,\kay,0)}^{2} h \mathfrak{S}\ud v\ud x \Big|.
\end{split}
\end{equation}
\end{proposition}
\begin{proof}
We use Lemma~\ref{lem:general-eng-estimate} with $\iota=\vartheta=0$. Thus, all the terms on the right-hand side with $\iota$ vanish. Further, we use Lemma~\ref{lem:lower-bound-L} to get a coercive bound for $\int_{\R^3}\int_{\R^3}w_{(0,\kay,0)}^2hLh\ud v\ud x$ which is possible since $\jap{x/t}$ is $v$-independent and so
$$\int_{\R^3}\int_{\R^3}w_{(0,\kay,0)}^2hLh\ud v\ud x=\int_{\R^3}\int_{\R^3}e^{Bt^{-0.1}}\jap{x/t}^{\kay}hL\left(h\jap{x/t}^{\kay}\right)\ud v\ud x.$$


Finally, we bound $t^{-2.9}\norm{|x-tv|w h}_{L^2_{x,v}}^2$ by $t^{-2.9} \norm{w_{(1,\kay,0)} h}_{L^2_{x,v}}^2$.
\end{proof}

Next, we allow for non-zero $\kay$ and $\vartheta$, but still require $\iota = 0$. Notice that because of an error term $\norm{w_{(0,\kay,0)} h}_{L^2_{x,v}}^2$ on the right-hand side of the estimate, when this estimate is applied, we will have to allow $\norm{w_{(1,\kay,0)} h}_{L^2_{x,v}}^2$ to grow.
\begin{proposition}\label{prop:Guassian-v-weights}
Let $h$ satisfy \eqref{eq:lin-with-source} then for $\kay \in\mathbb Z_{\geq 0}$, $\vartheta=1$, $\iota=0$, we have the following estimate:
\begin{align*}
\f{\ud}{\ud t} \norm{w_{(1,\kay,0)} h}_{L^2_{x,v}}^2+&\:\norm{h}_{\calD_{x,v}(1,\kay,0)}^2 + B t^{-1.1} \norm{ w_{(1,\kay,0)} h}_{L^2_{x,v}}^2 + t^{-1.1} \norm{|x-tv| w_{(1,\kay,0)} h}_{L^2_{x,v}}^2 \\
&\quad\lesssim \Big| \int_{\R^3}\int_{\R^3}w^2_{(1,\kay,0)} h \mathfrak{S}\ud v\ud x \Big| +\norm{w_{(0,\kay,0)} h}_{L^2_{x,v}}^2.
\end{align*}
\end{proposition} 
\begin{proof}
We again begin by an application of Lemma~\ref{lem:general-eng-estimate}. Since $\iota=0$, all the terms on the right-hand side with $\iota$ vanish. Hence
\begin{align*}
&\: \f{\ud}{\ud t} \norm{w_{(1,\kay,0)} h}_{L^2_{x,v}}^2 +B t^{-1.1} \norm{w_{(1,\kay,0)} h}_{L^2_{x,v}}^2\\
&\: + t^{-1.1} \norm{|x-tv| w_{(1,\kay,0)} h}_{L^2_{x,v}}^2+\int_{\R^3}\int_{\R^3}w_{(1,\kay,0)}^2 hLh\ud v\ud x \\
\lesssim &\: B^{-1} t^{-2.9}\norm{|x-tv|w_{(1,\kay,0)} h}_{L^2_{x,v}}^2 + \Big| \int_{\R^3}\int_{\R^3} w_{(1,\kay,0)}^2h \mathfrak{S}\ud v\ud x \Big| \\
\lesssim &\:  \Big| \int_{\R^3}\int_{\R^3} w_{(1,\kay,0)}^2 h \mathfrak{S}\ud v\ud x \Big|,
\end{align*}
where the last line is achieved after taking $B$ sufficiently large so that we can absorb the first term on the right-hand side by the second term on the left-hand side. 

Next, we use Corollary~\ref{cor:weighted-L-bound} to get that 
\begin{align*}
\f{\ud}{\ud t} \norm{w_{(1,\kay,0)} h}_{L^2_{x,v}}^2&\: +B t^{-1.1} \norm{w_{(1,\kay,0)} h}_{L^2_{x,v}}^2+ t^{-1.1} \norm{|x-tv| w_{(1,\kay,0)} h}_{L^2_{x,v}}^2+\norm{h}^2_{\calD_{x,v}(1,\kay,0)} \\
&\lesssim  \norm{w_{(0,\kay,0)} h}_{L^2_{x,v}}^2+t^{-2}\norm{w_{(1,\kay,0)}h}^2_{L^2_{x,v}} + \Big|\int_{\R^3}\int_{\R^3} w_{(1,\kay,0)}^2h \mathfrak{S}\ud v\ud x \Big|,
\end{align*}
Again, by choosing $B$ large enough, we can absorb the second term on the right-hand side by the second term on the left-hand side. This gives the desired bound. \qedhere

\end{proof}

Combining Proposition~\ref{prop:pure-poly-x-weights}  and Proposition~\ref{prop:Guassian-v-weights}, we obtain the following corollary.
\begin{corollary}\label{cor:pure-poly-x-weights}
Let $h$ satisfy \eqref{eq:lin-with-source}. Then for $\kay \in \mathbb Z_{\geq 0}$, $\iota=0$, the following estimate holds:
\begin{equation}\label{eq:combined.main.energy}
\begin{split}
&\f{\ud}{\ud t} \left(\norm{w_{(0,\kay,0)} h}_{L^2_{x,v}}^2 + t^{-1.1} \norm{w_{(1,\kay,0)}h}^2_{L^2_{x,v}}\right)+\norm{(I-\Pi)h}^2_{\calD_{x,v}(0,\kay,0)} + t^{-1.1} \norm{h}_{\calD_{x,v}(1,\kay,0)}^2 \\
&\quad+B t^{-1.1} \Big( \norm{w_{(0,\kay,0)} h}_{L^2_{x,v}}^2 + t^{-1.1} \norm{w_{(1,\kay,0)} h}_{L^2_{x,v}}^2 \Big)+ t^{-2.2} \norm{\jap{x-tv}w_{(1,\kay,0)} h}_{L^2_{x,v}}^2\\
&\quad\lesssim t^{-1.1} \Big| \int_{\R^3}\int_{\R^3} w_{(1,\kay,0)}^2 h \mathfrak{S}\ud v\ud x \Big| + \Big|\int_{\R^3}\int_{\R^3} w_{(0,\kay,0)}^2 h \mathfrak{S}\ud v\ud x\Big|.
\end{split}
\end{equation}
\end{corollary}
\begin{proof}
We first multiply the estimate in Proposition~\ref{prop:Guassian-v-weights} by $t^{-1.1}$. Observe that $[\f{\ud}{\ud t}, t^{-1.1}]$ generates a term with a good sign that we can drop. We then add it to the estimate in Proposition~\ref{prop:pure-poly-x-weights}. Hence, we get 
\begin{align*}
& \hbox{LHS of \eqref{eq:combined.main.energy}}\\
&\quad\lesssim t^{-1.1} \Big| \int_{\R^3}\int_{\R^3} w_{(1,\kay,0)}^2 h \mathfrak{S}\ud v\ud x \Big| + \Big|\int_{\R^3}\int_{\R^3} w_{(0,\kay,0)}^2 h \mathfrak{S}\ud v\ud x\Big| \\
&\qquad+ B^{-1} t^{-2.9} \norm{w_{(1,\kay,0)} h}_{L^2_{x,v}}^2 + t^{-1.1} \norm{w_{(0,\kay,0)} h}_{L^2_{x,v}}^2.
\end{align*}
We absorb the last two terms on the right-hand side by the terms $B t^{-1.1} \Big( \norm{w_{(0,\kay,0)} h}_{L^2_{x,v}}^2 + t^{-1.1} \norm{w_{(1,\kay,0)} h}_{L^2_{x,v}}^2 \Big)$ on the left-hand side after choosing $B$ large enough. This gives us the required estimate. \qedhere
\end{proof}




\subsection{Gaussian in $x/t$ weighted estimate}
In this subsection we prove estimates for $\iota=1$ and $\vartheta=0.$ Observe that there is a loss in $\jap{x/t}$ moments and thus the terms we control have different $\kay$ weights.
\begin{proposition}\label{prop:Gaussian-x-weights}
Let $h$ satisfy \eqref{eq:lin-with-source}. Then for $\kay \in \mathbb Z_{\geq 0}$, the following holds:
\begin{align*}
&\f{\ud}{\ud t} \left(\norm{w_{(0,\kay,1)} h}_{L^2_{x,v}}^2 + \norm{w_{(0,\kay+1,0)} h}_{L^2_{x,v}}^2 + t^{-1.1} \norm{w_{(1,\kay+1,0)}h}^2_{L^2_{x,v}}\right) \\
&\hspace{3em} +\norm{(I-\Pi)h}^2_{\calD_{x,v}(0,\kay,1)} + t^{-1.1} \norm{h}_{\calD_{x,v}(1,\kay+1,0)}^2\\
&\hspace{5em}\lesssim t^{-1.1} \Big| \int_{\R^3}\int_{\R^3} w_{(1,\kay+1,0)}^2 h \mathfrak{S}\ud v\ud x \Big| + \Big|\int_{\R^3}\int_{\R^3} w_{(0,\kay+1,0)}^2 h \mathfrak{S}\ud v\ud x\Big|\\
&\hspace{5em}\quad+ \Big| \int_{\R^3}\int_{\R^3} w_{(0,\kay,1)}^2h \mathfrak{S}\ud v\ud x \Big| .
\end{align*}
\end{proposition}
\begin{proof}
We first use Lemma~\ref{lem:general-eng-estimate} to get
\begin{equation}\label{eq:exp.xt.weight.prelim}
\begin{split}
&\: \f{\ud}{\ud t} \norm{w_{(0,\kay,1)} h}_{L^2_{x,v}}^2+B t^{-1.1} \norm{w_{(0,\kay,1)} h}_{L^2_{x,v}}^2+\int_{\R^3}\int_{\R^3}w_{(0,\kay,1)}^2 hLh\ud v\ud x\\
\lesssim &\: B^{-1} t^{-2.9} \norm{\jap{x/t} |x-tv| w_{(0,\kay,1)} h}_{L^2_{x,v}}^2 + t^{-3} \norm{\jap{x/t} w_{(0,\kay,1)} h}_{L^2_{x,v}}^2 + \Big| \int_{\R^3}\int_{\R^3} w_{(0,\kay,1)}^2h \mathfrak{S}\ud v\ud x \Big|.
\end{split}
\end{equation}
By Lemma~\ref{lem:lower-bound-L}, since $w_{(0,\kay,1)}$ is independent of $v$, the term on the left-hand side controls
\begin{equation}
	\int_{\R^3}\int_{\R^3}w^2_{(0,\kay,1)} hLh\ud v\ud x \gtrsim \norm{(I-\Pi)h}^2_{\calD_{x,v}(0,\kay,1)}.
\end{equation}
To bound the right-hand side of \eqref{eq:exp.xt.weight.prelim}, we consider \eqref{eq:combined.main.energy} with $\kay$ replaced by $\kay +1$. After dropping some non-negative terms on the left, we obtain
\begin{equation}\label{eq:exp.xt.weight.prelim.2}
\begin{split}
&\f{\ud}{\ud t} \left(\norm{w_{(0,\kay+1,0)} h}_{L^2_{x,v}}^2 + t^{-1.1} \norm{w_{(1,\kay+1,0)}h}^2_{L^2_{x,v}}\right)+B t^{-2.2}  \norm{w_{(1,\kay+1,0)} h}_{L^2_{x,v}}^2 \\
&\quad + t^{-1.1} \norm{h}_{\calD_{x,v}(1,\kay+1,0)}^2 + t^{-2.2} \norm{\jap{x-tv}w_{(1,\kay+1,0)} h}_{L^2_{x,v}}^2 \\
&\quad\lesssim t^{-1.1} \Big| \int_{\R^3}\int_{\R^3} w_{(1,\kay+1,0)}^2 h \mathfrak{S}\ud v\ud x \Big| + \Big|\int_{\R^3}\int_{\R^3} w_{(0,\kay+1,0)}^2 h \mathfrak{S}\ud v\ud x\Big|.
\end{split}
\end{equation}

We now sum \eqref{eq:exp.xt.weight.prelim} and \eqref{eq:exp.xt.weight.prelim.2}. Observe that the terms $B^{-1} t^{-2.9} \norm{\jap{x/t} |x-tv| w_{(0,\kay,1)} h}_{L^2_{x,v}}^2$, $t^{-3} \norm{\jap{x/t} w_{(0,\kay,1)} h}_{L^2_{x,v}}^2$ on the right-hand side of \eqref{eq:exp.xt.weight.prelim} are controlled, respectively, by the terms $t^{-2.2} \norm{\jap{x-tv}w_{(1,\kay+1,0)} h}_{L^2_{x,v}}^2$, $B t^{-2.2}  \norm{w_{(1,\kay+1,0)} h}_{L^2_{x,v}}^2$ on the left-hand side of \eqref{eq:exp.xt.weight.prelim.2} after taking $B$ large. After dropping some non-negative terms, this gives the desired bound. \qedhere

\end{proof}

At this point we fix the constant $B$ so that all the estimates above hold. From now on, all implicit constants are allowed to depend on $B$.

\section{Commutator estimates}\label{sec:commutator}

\subsection{Estimates for $[Y,A]$}
We will commute the equation with $Y_l=t\partial_{x_l}+\partial_{v_l}$. Thus, we need estimates on commutators that arise due to the Landau collisional terms. 

We start by proving the following preliminary lemma.
\begin{lemma}\label{lem:Y-der-exp}
Define 
$$\sigma_{(l)ij} := \int_{\RR^3} \phi_{ij}(v-v_*) z_{*,l} \mu(v_*)\ud v_*.$$
Then the following holds: 
\begin{align}
	Y_l \sigma_{ij} = &\: -\f{2\alp \bt x_l t}{\alp + \bt t^2} \sigma_{ij}   - \f{2\alp}{\sqrt{\alp+\bt t^2}} \sigma_{(l)ij},  \label{eq:Y.on.coeff.1}\\
	Y_l (\sigma_{ij} z_j) = &\: -\f{2\alp \bt x_l t}{\alp + \bt t^2} \sigma_{ij}z_j   - \f{2\alp}{\sqrt{\alp+\bt t^2}} \sigma_{(l)ij} z_j + \f{\alpha}{\sqrt{\alpha+\beta t^2}} \sigma_{ij} \de_{lj}, \label{eq:Y.on.coeff.2}\\
	Y_l (\sigma_{ij} z_i z_j) =&\:  -\f{2\alp \bt x_l t}{\alp + \bt t^2} (\sigma_{ij}z_i z_j)  - \f{2\alp}{\sqrt{\alp+\bt t^2}} \sigma_{(l)ij}z_i z_j +  \f{\alpha}{\sqrt{\alpha+\beta t^2}} \sigma_{ij} (z_i \de_{jl} + z_j \de_{il}). \label{eq:Y.on.coeff.3}
\end{align}
\end{lemma}
\begin{proof}
First, observe that
\begin{align}
Y_l \mu = &\: -2 \alp v_l \mu = -\f{2 \alp \bt x_l t}{\alp + \bt t^2} \mu - \f{2\alp z_l}{\sqrt{\alp+\bt t^2}} \mu, \label{eq:Y-der-mu}\\
Y_l z_j = &\: \f{\alpha}{\sqrt{\alpha+\beta t^2}} \de_{lj}. \label{eq:Y-der-z}
\end{align}
Using \eqref{eq:Y-der-mu} and \eqref{eq:def-z}, we obtain
\begin{equation}
	Y_l \sigma_{ij} =  \int_{\RR^3} \phi_{ij}(v-v_*) (Y_l\mu)(v_*) \ud v_*  = -\f{2\alp \bt x_l t}{\alp + \bt t^2}\sigma_{ij} - \f{2\alp}{\sqrt{\alp+\bt t^2}}\int_{\RR^3} \phi_{ij}(v-v_*) z_{*,l} \mu(v_*) \ud v_*,
\end{equation}
which gives \eqref{eq:Y.on.coeff.1}. Finally, \eqref{eq:Y.on.coeff.2} and \eqref{eq:Y.on.coeff.3} can be obtained by combining \eqref{eq:Y.on.coeff.1} and \eqref{eq:Y-der-z}. \qedhere

\end{proof}

Using the computations above, we have the following bounds.
\begin{proposition}\label{prop:Y.on.coeff.est}
	\begin{align}
		\Big| Y_l \sigma_{ij} +\f{2 \alp \bt x_l t}{\alp + \bt t^2} \sigma_{ij}\Big|, \,\Big| \Big(Y_l \sigma_{ij} +\f{2 \alp \bt x_l t}{\alp + \bt t^2} \sigma_{ij}\Big) z_i \Big|, \, \Big| \Big( Y_l \sigma_{ij} +\f{2 \alp \bt x_l t}{\alp + \bt t^2} \sigma_{ij} \Big) z_i z_j \Big| \ls &\: t^{-1} \times t^{-2} e^{-\f{\alp\bt |x|^2}{\alp + \bt t^2}} \jap{z}^{-1}, \label{eq:Y.on.coeff.est.1}\\
		\Big| Y_l \rd_{v_i} \sigma_{i} + \f{2 \alp \bt x_l t}{\alp + \bt t^2} \rd_{v_i}\sigma_{i}\Big| \ls &\: t^{-1} \times t^{-1} e^{-\f{\alp\bt |x|^2}{\alp + \bt t^2}} \jap{z}^{-1}, \label{eq:Y.on.coeff.est.2}\\
		\Big| Y_l (\sigma_{ij} z_i z_j) +\f{2 \alp \bt x_l t}{\alp + \bt t^2} (\sigma_{ij} z_i z_j) \Big|  \ls &\: t^{-1} \times t^{-2} e^{-\f{\alp\bt |x|^2}{\alp + \bt t^2}} \jap{z}^{-1}. \label{eq:Y.on.coeff.est.3}
	\end{align}
\end{proposition}
\begin{proof}
	By Proposition~\ref{prop:aux-integral},
	\begin{equation}
		|\sigma_{(l)ij}| \ls t^{-2} e^{-\f{\alp\bt|x|^2}{\alp + \bt t^2}} \jap{z}^{-1}.
	\end{equation}
	The first estimate in \eqref{eq:Y.on.coeff.est.1} then immediately follows from Lemma~\ref{lem:Y-der-exp}. In order to show that a contraction with $z$ does not increase the weight in $\jap{z}$, we note that $\phi_{ij}(v-v_*) (z-z_*)_i = \phi_{ij}(v-v_*) (z-z_*)_j = 0$ so that for instance
	\begin{equation}
		\sigma_{(l)ij} z_j  = \int_{\RR^3} \phi_{ij}(v-v_*) z_{*,l} z_{*,j} \mu(v_*)\ud v_*
	\end{equation} 
	and the term does not contribute additional $\jap{z}$ weights.
	
	The other estimates can be obtained in a similar manner. For \eqref{eq:Y.on.coeff.est.2}, we observe that $\rd_{v_i} z_j = \sqrt{\alp +\bt t^2} \de_{ij}$. \qedhere
\end{proof}

\begin{proposition}\label{prop:Y.A}
Let $w = w_{(\vartheta,\kay,\iota)}$. Then
	\begin{equation}
		\begin{split}
			&\: |\langle w^2 Y_l g, [Y_l, A] g \rangle_{L^2_{x,v}}| \\
			\ls &\: \| Y_l g\|_{\calD_{x,v}(\vartheta,\kay,\iota)} \| g\|_{\calD_{x,v}(\vartheta,\kay+1,\iota)} + t^{-2} \| w_{(\vartheta,\kay,\iota)} Y_l g\|_{L^2_{x,v}} \| w_{(\vartheta,\kay+1,\iota)} g\|_{L^2_{x,v}} .
		\end{split}
	\end{equation}
\end{proposition}
\begin{proof}
We compute
	\begin{equation}\label{eq:[Y_l,A].expand.sort.of}
		\begin{split}
			[Y_l, A] g = &\: \rd_{v_i} ((Y_l \sigma_{ij}) \rd_{v_j} g )+ \sqrt{\alp + \bt t^2} (\rd_{v_i} Y_l \sigma_i) g - (\alp + \bt t^2) Y_l (\sigma_{ij} z_i z_j) g \\
			= &\: -\f{2\alp \bt x_l t}{\alp + \bt t^2} A g + \mathrm{error},
		\end{split}
	\end{equation}
	where the error term can be bounded using Proposition~\ref{prop:Y.on.coeff.est}.
	
	The main term corresponding to $-\f{2\alp \bt x_l t}{\alp + \bt t^2} A$ can be bounded using Proposition~\ref{prop:weighted-A-upper-bound} by
	\begin{equation}
		\begin{split}
			&\: \Big|\langle w^2 Y_l g, \f{2\alp \bt x_l t}{\alp + \bt t^2} A g \rangle_{L^2_{x,v}} \Big| \ls \Big\| \f{|x|}{t} \Big\langle w^2 Y_l g,  A g \Big\rangle_{L^2_{v}} \Big\|_{L^1_x} \\
            \ls &\: \| Y_l g\|_{\calD_{x,v}(\vartheta,\kay,\iota)} \| g\|_{\calD_{x,v}(\vartheta,\kay+1,\iota)} + t^{-2} \| w_{(\vartheta,\kay,\iota)} Y_l g\|_{L^2_{x,v}} \| w_{(\vartheta,\kay+1,\iota)} g\|_{L^2_{x,v}} .
		\end{split}
	\end{equation}
	For the error terms from $\sigma_{(l)ij}$ associated with $\rd_{v_i} ((Y_l \sigma_{ij}) \rd_{v_j} g )$, we integrate by parts in $\rd_{v_i}$ so that it is bounded above by
	\begin{equation}
		\begin{split}
			&\: t^{-1}  \Big|\jap{ w^2 \rd_{v_i} Y_l g, \sigma_{(l)ij} \rd_{v_j} g }_{L^2_{x,v}} \Big| + t^{-1}  \Big|\jap{ (\rd_{v_i} w^2) Y_l g, \sigma_{(l)ij} \rd_{v_j} g }_{L^2_{x,v}} \Big| \\
			\ls &\: t^{-1} \times t^{-2} \Big( \jap{ e^{-\f{\alp \bt |x|^2}{\alp+\bt t^2}}w^2 \jap{z}^{-3} |P_z \rd_{v} Y_l g|, |P_z \rd_{v} g| }_{L^2_{x,v}} + \jap{ e^{-\f{\alp \bt |x|^2}{\alp+\bt t^2}}w^2 \jap{z}^{-2} |(I-P_z) \rd_{v} Y_l g|, |P_z \rd_{v} g| }_{L^2_{x,v}}  \\
			&\: \quad + \jap{ e^{-\f{\alp \bt |x|^2}{\alp+\bt t^2}}w^2 \jap{z}^{-2} |P_z \rd_{v} Y_l g|, |(I-P_z) \rd_{v} g| }_{L^2_{x,v}}  \\
            &\: \quad + \jap{e^{-\f{\alp \bt |x|^2}{\alp+\bt t^2}} w^2 \jap{z}^{-1} |(I-P_z) \rd_{v} Y_l g|, |(I-P_z) \rd_{v} g| }_{L^2_{x,v}}  \Big) + t^{-1}  \Big|\jap{ (\rd_{v_i} w^2) Y_l g, \sigma_{(l)ij} \rd_{v_j} g }_{L^2_{x,v}} \Big|\\
			\ls &\: t^{-1} \| Y_l g\|_{\calD_{x,v}(\vartheta,\kay,\iota)} \|g \|_{\calD_{x,v}(\vartheta,\kay,\iota)} + t^{-3} \| w_{(\vartheta,\kay,\iota)} Y_l g\|_{L^2_{x,v}} \|w_{(\vartheta,\kay,\iota)} g \|_{L^2_{x,v}}.
		\end{split}
	\end{equation}
    Here, we used the fact that in \eqref{eq:Y.on.coeff.est.1}, we have the same estimate when contracted with $z$ and so the estimate gains one power of $\jap{z}^{-1}$ when there is a factor of $P_z \rd_v$ and gains $\jap{z}^{-2}$ when there are two factors of $P_z \rd_v$. For the term $t^{-1}  \Big|\jap{ (\rd_{v_i} w^2) Y_l g, \sigma_{(l)ij} \rd_{v_j} g }_{L^2_{x,v}} \Big|$, we have used the $\rd_v w= 0$ if $\vartheta = 0$, and if $\vartheta=1$, the term can be treated in a similar manner as terms $II$, $III$ in the proof of Proposition~\ref{prop:weighted-A-bound}.
    
	The error terms associated with $\sqrt{\alp + \bt t^2} (\rd_{v_i} Y_l \sigma_i) g$ and $(\alp + \bt t^2) Y_l (\sigma_{ij} z_i z_j) g$ can be bounded above after using Proposition~\ref{prop:Y.on.coeff.est} by
	\begin{equation}
		\begin{split}
			&\: t^{-1} \langle e^{-\f{\alp \bt |x|^2}{\alp+\bt t^2}} w^2 |Y_l g|, \jap{z}^{-1} | g| \rangle_{L^2_{x,v}} \ls t^{-1} \| Y_l g\|_{\calD_{x,v}(\vartheta,\kay,\iota)} \|g \|_{\calD_{x,v}(\vartheta,\kay,\iota)}.
		\end{split}
	\end{equation} 
\end{proof}

\subsection{Commutator $[Y,K]$}\label{sec:Y.K.comm}
Next we consider the commutator of $Y$ and the operator $K$. We start with an auxiliary proposition for the following quantity.
\begin{definition}\label{def:frak-K-def}
For any smooth function $g$, let 
\begin{align}
\mathfrak{K}_{ij} g:= &\: \mu^{\f 12} \int_{\R^3}\phi_{ij}(v-v_*)\mu_*^{\f 12} g(v_*)\ud v_*. \label{eq:frak-K-def.1} 
\end{align}
\end{definition}

The following bound is straightforward:
\begin{lemma}\label{lem:fraK.direct}
\begin{align}
	|\mathfrak{K}_{ij} g| \ls \int_{\RR^3} |v-v_*|^{-1} \mu_*^{\f 12} \mu^{\f 12} |g_*| \ud v_*.
\end{align}
\end{lemma}

\begin{proposition}\label{prop:Y-K-comm-aux}
For $Y=t\rd_x+\rd_v$, we have the following:
\begin{align}
\Big| [Y_l,\mathfrak K_{ij} ]g + \f{2 \alp \bt x_l t}{\alp + \bt t^2} \mathfrak K_{ij} g \Big| \ls t^{-1} \int_{\RR^3} |v-v_*|^{-1} \mu_*^{\f 12} \mu^{\f 12} (|z| + |z_*| ) |g_*| \ud v_*. \label{eq:Y-K-comm-aux.1} 
\end{align}
\end{proposition}
\begin{proof}
We compute 
\begin{equation}
\begin{split}
	Y_l \mathfrak K_{ij} g = &\: \f 12 (Y_l \log \mu) \mathfrak K_{ij} g + \mu^{\f 12} \int_{\R^3}\phi_{ij}(v-v_*)(t\rd_{x_l} + \rd_{v_{*,l}})(\mu_*^{\f 12} g(v_*))\ud v_* \\
	= &\: -\alp \mu^{\f 12} \int_{\R^3}\phi_{ij}(v-v_*)\mu_*^{\f 12} (v_l + v_{*,l}) g(v_*)\ud v_* +  \mathfrak K_{ij} Y_l g.
\end{split}
\end{equation} 
Writing $v_l = \f{\bt t x_l}{\alp +\bt t^2} + \f{z_l}{\sqrt{\alp + \bt t^2}}$, $v_{*,l} = \f{\bt t x_l}{\alp +\bt t^2} + \f{z_{*,l}}{\sqrt{\alp + \bt t^2}}$, and using $|\phi_{ij}(v-v_*)|\ls |v-v_*|^{-1}$, we obtain \eqref{eq:Y-K-comm-aux.1}. \qedhere 
\end{proof}

\begin{proposition}\label{prop:Y.K.commute}
$[Y_l,K]$ obeys the following pointwise estimates:
\begin{equation}
	\Big| [Y_l, K] g + \f{2 \alp \bt x_l t}{\alp + \bt t^2} K g\Big|(t,x,v) \ls \calI_1 + \calI_2,
\end{equation}
where
\begin{align}
	\calI_1(t,x,v) := &\:  t^{-1} |z| \mu |g|, \\
	\calI_2(t,x,v) := &\: t^{-1} \times t^2 \int_{\RR^3} |v-v_*|^{-1} \mu^{\f 12} \mu_*^{\f 12} (\jap{z_*}^3 + \jap{z} ) |g_*| \ud v_*.
\end{align}
\end{proposition}
\begin{proof}
We use the formula \eqref{eq:K.representation} for $K$. Using Definition~\ref{def:frak-K-def} and noting $\phi_{ii}(y) = 2|y|^{-1}$, $K$ can be rewritten as
\begin{equation}\label{eq:K.for.Y.commutator}
	K g = 8\pi \mu g  + 4(\alp + \bt t^2)\, \mathfrak K_{ij}\big(z_i z_j g\big) - 2 (\alp + \bt t^2)\, \mathfrak K_{ii}(g).
\end{equation}

For the first term in \eqref{eq:K.for.Y.commutator}, we compute
\begin{equation}
\begin{split}
	 [Y_l, \mu] g + \f{2 \alp \bt x_l t}{\alp + \bt t^2} \mu g = (Y_l \mu) g + \f{2 \alp \bt x_l t}{\alp + \bt t^2} \mu g= -2\alp v_l \mu g + \f{2 \alp \bt x_l t}{\alp + \bt t^2} \mu g = - \f{2 \alp z_l}{\sqrt{\alp + \bt t^2}} \mu g.
\end{split}
\end{equation}
The right-hand side is clearly controlled by $\calI_1$.

For the second term in \eqref{eq:K.for.Y.commutator}, we use \eqref{eq:Y-der-z} so that
\begin{equation}
\begin{split}
	&\: Y_l \Big( (\alp + \bt t^2) \mathfrak K_{ij} \Big(z_i z_j g\Big) \Big) - \Big( (\alp + \bt t^2) \mathfrak K_{ij} \Big(z_i z_j Y_l g\Big) \Big) + \f{2\alp \bt x_l t}{\alp + \bt t^2} \Big( (\alp + \bt t^2) \mathfrak K_{ij} \Big( z_i z_j g\Big) \Big)\\
	= &\: (\alp + \bt t^2) \Big( [Y_l, \mathfrak K_{ij}] + \f{2\alp \bt x_l t}{\alp + \bt t^2} \mathfrak K_{ij} \Big)(z_i z_j g) + 2 \alp \sqrt{\alp + \bt t^2} \mathfrak K_{lj} (z_j g).
\end{split}
\end{equation}
By Lemma~\ref{lem:fraK.direct} and Proposition~\ref{prop:Y-K-comm-aux}, each term is bounded above by $\calI_2$.

Similarly,
\begin{equation}
\begin{split}
	&\: (\alp + \bt t^2) Y_l \Big(\mathfrak K_{ii} (g) \Big) - (\alp + \bt t^2)  \mathfrak K_{ii} ( Y_l g) + \f{2\alp \bt x_l t}{\alp + \bt t^2} \Big( (\alp + \bt t^2) \mathfrak K_{ii} (g) \Big) \\
	= &\: (\alp + \bt t^2) \Big(  [Y_l, \mathfrak K_{ii}] + \f{2\alp \bt x_l t}{\alp + \bt t^2} \mathfrak K_{ii}\Big)(g),
\end{split}
\end{equation}
which can be bounded above by $\calI_2$ by Lemma~\ref{lem:fraK.direct} and Proposition~\ref{prop:Y-K-comm-aux}. \qedhere
\end{proof}

\begin{proposition}\label{prop:Y.K}
	\begin{enumerate}
	\item If $\vartheta = 1$, 
	\begin{equation}
		\begin{split}
			&\: \Big| \langle w_{(1,\kay,0)}^2 [Y_l,K] g_1, Y_l g_2  \rangle_{L^2_{x,v}} \Big| \\
			\ls &\: \|  g_1\|_{\calD_{x,v}(1,\kay+1,0)} \| Y_l g_2\|_{\calD_{x,v}(1,\kay,0)} + t^{-2} \| w_{(1,\kay+1,0)} g_1\|_{L^2_{x,v}}\|w_{(1,\kay,0)} Y g_2\|_{L^2_{x,v}}.
		\end{split}
	\end{equation}
	\item If $\vartheta = 0$, 
	\begin{equation}
		\begin{split}
			&\: \Big| \langle w_{(0,\kay,\iota)}^2 [Y_l,K] g_1, Y_l g_2  \rangle_{L^2_{x,v}} \Big| \\
			\ls &\: \| g_1\|_{\calD_{x,v}(0,\kay+1,\iota)} \| Y_l g_2\|_{\calD_{x,v}(0,\kay,\iota)} + t^{-2} \| w_{(0,\kay,\iota)} g_1\|_{L^2_{x,v}}\|w_{(0,\kay,\iota)} Y g_2\|_{L^2_{x,v}}.
		\end{split}
	\end{equation}
	\end{enumerate}
\end{proposition}
\begin{proof}
	Using Proposition~\ref{prop:Y.K.commute}, we have
	\begin{equation}\label{eq:main.Y.K.commute}
		\begin{split}
			\Big| \langle w^2 [Y_l,K] g_1, Y_l g_2  \rangle_{L^2_{x,v}} \Big| \ls \Big| \Big\langle \f{2 \alp \bt x_l t}{\alp + \bt t^2} w^2 K g_1, Y_l g_2 \Big\rangle_{L^2_{x,v}} \Big| + \mathrm{error},
		\end{split}
	\end{equation}
	where the error terms are exactly those controlled by $\calJ_1$, $\calJ_2$ in Proposition~\ref{prop:main.prelim.for.K} (but with $g_2$ replaced by $Y_l g_2$). Thus, by Proposition~\ref{prop:main.prelim.for.K}, for $\vartheta = 0$, we obtain the desired estimate for the error term. When $\vartheta = 1$, Proposition~\ref{prop:main.prelim.for.K} gives
    \begin{equation}
		\begin{split}
			&\: \Big| \langle w_{(1,\kay,0)}^2 [Y_l,K] g_1, Y_l g_2  \rangle_{L^2_{x,v}} \Big| \\
			\ls &\: \| w_{(0,\kay+1,0)} g_1\|_{L^2_{x,v}} \| w_{(0,\kay,0)} Y_l g_2\|_{L^2_{x,v}} + t^{-2} \| w_{(1,\kay+1,0)} g_1\|_{L^2_{x,v}}\|w_{(1,\kay,0)} Y g_2\|_{L^2_{x,v}}.
		\end{split}
	\end{equation}
    Now observe that $w_{(0,\kay,0)} = e^{-\f{\alp \bar{\bt} |x|^2}{2(\alp + \bar{\bt} t^2)}} e^{-\f 12 |\sqrt{\alp + \bar{\bt} t^2} v - \f{x \bar{\bt} t}{\sqrt{\alp + \bar{\bt} t^2}}|^2} w_{(1,\kay,0)}$, and splitting into $|x| \leq t^2$ and $|x|\geq t^2$ as in Proposition~\ref{prop:main.prelim.for.K}, we obtain
    \begin{equation*} \begin{split} &\: \| w_{(0,\kay+1,0)} g_1\|_{L^2_{x,v}} \| w_{(0,\kay,0)} Y_l g_2\|_{L^2_{x,v}} \\ \ls &\:  \|  g_1\|_{\calD_{x,v}(1,\kay+1,0)} \| Y_l g_2\|_{\calD_{x,v}(1,\kay,0)} + t^{-2} \| w_{(1,\kay+1,0)} g_1\|_{L^2_{x,v}}\|w_{(1,\kay,0)} Y g_2\|_{L^2_{x,v}}, \end{split} \end{equation*}
    which implies the desired result.
	
	Finally, for the main term $\langle \f{2 \alp \bt x_l t}{\alp + \bt t^2} w^2 K g_1, Y_l g_2 \rangle_{L^2_{x,v}} $ in \eqref{eq:main.Y.K.commute}, we control it using Proposition~\ref{prop:weighted-K-bound}. Notice that we allow for an extra $\jap{x/t}$ moment to control the factor $\f{2 \alp \bt x_l t}{\alp + \bt t^2}$. \qedhere
\end{proof}

\subsection{Macroscopic quantities and commutator $[Y,\Pi]$}

Next, we also need the commutator between $Y$ and the projection $I-\Pi.$
We first derive a formula for the macroscopic quantities as in Definition~\ref{def:macro-quant}.
\begin{lemma}\label{lem:exp-abc}
For $a^h$, $b^h_i$ and $c^h$ as in Definition~\ref{def:macro-quant}, we have the following expressions
$$a^h= \pi^{-3/2}e^{\f{\alpha\beta |x|^2}{\alpha+\beta t^2}}(\alpha+\beta t^2)^{3/2}\left[\f{5}{2}\int_{\R^3}h\sqrt{\mu}\ud v-\int_{\R^3}|z|^2h\sqrt{\mu}\ud v\right],$$
$$b^h_i=2\pi^{-3/2}e^{\f{\alpha\beta |x|^2}{\alpha+\beta t^2}}(\alpha+\beta t^2)^{3/2}\int_{\R^3}z_i h\sqrt{\mu}\ud v,$$
and 
$$c^h=\pi^{-3/2}e^{\f{\alpha\beta |x|^2}{\alpha+\beta t^2}}(\alpha+\beta t^2)^{3/2}\left[-\int_{\R^3}h\sqrt{\mu}\ud v+\f{2}{3}\int_{\R^3}|z|^2h\sqrt{\mu}\ud v\right].$$

Thus, 
\begin{equation}\label{eq:Pi.formula}
\begin{split}
\Pi h=\pi^{-3/2}(\alpha+\beta t^2)^{3/2}e^{-\f{|z|^2}{2}}&\left\{\int_{\R^3}(\f 52- |z_*|^2) h_*e^{-\f{|z_*|^2}{2}}\ud v_*+2z_i\int_{\R^3}(z_*)_i h_*e^{-\f{|z_*|^2}{2}}\ud v_*\right.\\
&\hspace{6em}\left.+|z|^2\int_{\R^3}(-1+ \f 23 |z_*|^2) h_*e^{-\f{|z_*|^2}{2}}\ud v_*\right\}.
\end{split}
\end{equation}
\end{lemma}
\begin{proof}
First note that 
\begin{equation}\label{eq:mu-expan}
\sqrt{\mu}=e^{-\f{\alpha|v|^2}{2}}e^{-\f{\beta |x-tv|^2}{2}}=e^{-\f{\alpha\beta |x|^2}{2(\alpha+\beta t^2)}}e^{-\f{|z|^2}{2}}.
\end{equation}

Then we can see easily that for every $i$, $z_i\sqrt{\mu}$ is orthogonal to $\sqrt{\mu}$, $z_j\sqrt{\mu}$ for $j\neq i$ and $|z|^2\sqrt{\mu}$. Hence 
$$b^h_i=\f{\int_{\R^3}z_i h\sqrt{\mu}\ud v}{\int_{\R^3}z_i^2\mu\ud v }=2\pi^{-3/2}e^{\f{\alpha\beta |x|^2}{\alpha+\beta t^2}}(\alpha+\beta t^2)^{{3/2}}\int_{\R^3}z_i h\sqrt{\mu}\ud v.$$
To derive $a^h$, $c^h$ we look at the following system of equations
\begin{align*}
\begin{bmatrix}
\int_{\R^3} \mu\ud v& \int_{\R^3}|z|^2\mu\ud v\\
\int_{\R^3}|z|^2\mu\ud v& \int_{\R^3}|z|^4\mu\ud v
\end{bmatrix}
\begin{bmatrix}
a^h\\c^h
\end{bmatrix}
=
\begin{bmatrix}
\int_{\R^3} h\sqrt{\mu}\ud v\\ \int_{\R^3}h |z|^2 \sqrt{\mu} \ud v
\end{bmatrix}.
\end{align*}
Using simple Gaussian integrals we get that 
$$\begin{bmatrix}
\int_{\R^3} \mu\ud v& \int_{\R^3}|z|^2\mu\ud v\\
\int_{\R^3}|z|^2\mu\ud v& \int_{\R^3}|z|^4\mu\ud v
\end{bmatrix}=\pi^{3/2}\f{e^{-\f{\alpha\beta |x|^2}{\alpha+\beta t^2}}}{(\alpha+\beta t^2)^{3/2}}\begin{bmatrix}
1& {3/2} \\
 {3/2}&{15/4 }
\end{bmatrix}.$$
Finally, we get that 
$$
\begin{bmatrix}
\int_{\R^3} \mu\ud v& \int_{\R^3}|z|^2\mu\ud v\\
\int_{\R^3}|z|^2\mu\ud v& \int_{\R^3}|z|^4\mu\ud v
\end{bmatrix}^{-1}=\pi^{-3/2}e^{\f{\alpha\beta |x|^2}{\alpha+\beta t^2}}(\alpha+\beta t^2)^{3/2}{\begin{bmatrix}
5/2& -1 \\
-1&2/3
\end{bmatrix}.}
$$
Hence $$a^h=\pi^{-3/2}e^{\f{\alpha\beta |x|^2}{\alpha+\beta t^2}}(\alpha+\beta t^2)^{3/2}\left[\f{5}{2}\int_{\R^3}h\sqrt{\mu}\ud v-\int_{\R^3}|z|^2h\sqrt{\mu}\ud v\right]$$
and 
$$c^h=\pi^{-3/2}e^{\f{\alpha\beta |x|^2}{\alpha+\beta t^2}}(\alpha+\beta t^2)^{3/2}\left[-\int_{\R^3}h\sqrt{\mu}\ud v+\f{2}{3}\int_{\R^3}|z|^2h\sqrt{\mu}\ud v\right].$$

The final statement can be proved by using \eqref{eq:mu-expan} and Definition~\ref{def:macro-quant}.
\end{proof}
\begin{proposition}\label{prop:comm-Y-Pi}
For $\Pi$ as defined in Definition~\ref{def:macro-quant}, the commutator $[Y_l,\Pi] = -[Y_l,I-\Pi]$ schematically takes the form
\begin{equation}\label{eq:Yl.Pi.schematic}
	[Y_l,\Pi]g \eqs t^{-1}  \times t^3 e^{-\f{|z|^2}{2}} \Big( \mathfrak{h}_l^{(1)} (t,z)\ \int_{\R^3}\mathfrak{h}_l^{(2)} (t,z_*)g_*e^{-\f{|z_*|^2}{2}}\ud v_*\Big), 
\end{equation}
where $\mathfrak{h}_l^{(i)} (t,z)$ satisfies the bound 
\begin{equation}\label{eq:frak.h.schematic.bound}
	|\rd_z^\om \mathfrak{h}^{(i)}_l| \lesssim_\om \jap{z}^{\max\{3-|\om|,0\}}.
\end{equation}
In \eqref{eq:Yl.Pi.schematic}, $\eqs$ means that the equation is schematic, and that the right-hand side consists of a sum of terms of this form.
\end{proposition}
\begin{proof}
We start with \eqref{eq:Pi.formula}. We compute 
\begin{equation}
\begin{split}
	&\: Y_l \Big( \pi^{-3/2}(\alpha+\beta t^2)^{3/2}e^{-\f{|z|^2}{2}}\int_{\R^3}(\f 52- |z_*|^2) g_*e^{-\f{|z_*|^2}{2}}\ud v_* \Big) \\
	= &\: \pi^{-3/2}(\alpha+\beta t^2)^{3/2}(Y_l e^{-\f{|z|^2}{2}})\int_{\R^3}(\f 52- |z_*|^2) g_*e^{-\f{|z_*|^2}{2}}\ud v_* \\
	&\: + \pi^{-3/2}(\alpha+\beta t^2)^{3/2}e^{-\f{|z|^2}{2}} \int_{\R^3} Y_{*l} ((\f 52- |z_*|^2) g_*e^{-\f{|z_*|^2}{2}}) \ud v_*,
\end{split}
\end{equation}
so that the commutator $$\pi^{-3/2}(\alpha+\beta t^2)^{3/2} \Big( Y_l ( e^{-\f{|z|^2}{2}}\int_{\R^3}(\f 52- |z_*|^2) g_*e^{-\f{|z_*|^2}{2}}\ud v_* ) - e^{-\f{|z|^2}{2}} \int_{\R^3}  (\f 52- |z_*|^2) (Y_lg)_*e^{-\f{|z_*|^2}{2}} \ud v_*\Big)$$ 
obeys the desired bound after using \eqref{eq:Y-der-z}. The other terms can be estimated in a similar manner. \qedhere


\end{proof}

We now control terms taking the same general schematic form as the macroscopic quantities. We will prove two general lemmas and then apply them to specific cases.

\begin{lemma}\label{lem:a.b.c.schematic}
	Let $k\in \mathbb Z_{\geq 0}$ and $\ell \geq 0$. Suppose $g$ is regular, and
	$$\mathfrak m^g(t,x) = t^3 \Big(  \int_{\R^3}\mathfrak{h} (t,z_*)g_*e^{-\f{|z_*|^2}{2}}\ud v_*\Big) \quad \hbox{where $| \mathfrak{h}| \lesssim \jap{z}^{k}$}.$$
	Then the following holds for some implicit constants depending on $k$ and $\ell$:
	\begin{align}
		 |  \mathfrak m^g |(t,x) \ls &\: t^{\f 32} \| \jap{z}^{-\ell} g\|_{L^2_{v}}. \label{eq:dv.macro.1} 
	\end{align}
\end{lemma}
\begin{proof}
	We use Cauchy--Schwarz to obtain
	\begin{equation}
	\begin{split}
		|  \mathfrak m^g |(t,x)
		\ls &\:  t^3 \Big( \int_{\RR^3} e^{-|z_{*}|^2}  \jap{z_{*}}^{2k+2\ell} \ud v_{*} \Big)^{\f 12} \Big( \int_{\RR^3} \jap{z_*}^{-2\ell} g_*^2 \ud v_* \Big)^{\f 12}.
	\end{split}
	\end{equation}
	A change of variables shows that
	\begin{equation}\label{eq:macro.stupid.integration}
	\int_{\RR^3} e^{-|z_*|^2} \jap{z_*}^{2k+2\ell} \ud v_*  \ls t^{-3},
	\end{equation}
	which gives \eqref{eq:dv.macro.1}. \qedhere

\end{proof}

\begin{lemma}\label{lem:Pi.schematic}
	Let $k_1,\,k_2 \in \mathbb Z_{\geq 0}$. Suppose $g$ is regular, and
	$$\mathfrak s^g = t^3 e^{-\f{|z|^2}{2}} \Big( \mathfrak{h}^{(1)} (t,z) \int_{\R^3}\mathfrak{h}^{(2)} (t,z_*)g_*e^{-\f{|z_*|^2}{2}}\ud v_*\Big) \quad \hbox{where $|\rd_z^\om \mathfrak{h}^{(i)}| \lesssim_\om \jap{z}^{k_1},\quad i =1,2.$}$$
	Then for any multi-index $\gamma$, the following holds for some implicit constants depending on $k_1$, $k_2$ and $\gamma$:
	\begin{align}
		 \|  \jap{z}^{k_2} \rd_v^\gamma \mathfrak s^g \|_{L^2_{v}} \ls &\: t^{|\gamma|}  \| \jap{z}^{-\f 12} g\|_{L^2_{v}}. \label{eq:dv.[Y,Pi].1} 
	\end{align}
\end{lemma}
\begin{proof}
	We write 
	$$\mathfrak s^g = e^{-\f{|z|^2}{2}} \mathfrak{h}^{(1)} (t,z) \mathfrak m^g(t,x),$$
	where $\mathfrak m^g$ obeys the assumptions of Lemma~\ref{lem:a.b.c.schematic}.
	We then estimate
	\begin{equation}
	\begin{split}
		\| \jap{z}^{{k_2}} \rd_v^\gamma \mathfrak s^g \|_{L^2_{v}}
		\ls &\:  t^{|\gamma|}  \Big(  \sup_x \int_{\RR^3} e^{-|z|^2} \jap{z}^{2(|\gamma|+ k_1+k_2)} \ud v \Big)^{\f 12}  |  \mathfrak m^g |.
	\end{split}
	\end{equation}
	Using the result of Lemma~\ref{lem:a.b.c.schematic} and a computation similar to \eqref{eq:macro.stupid.integration}, we obtain the desired estimate.  \qedhere
	
\end{proof}

\begin{corollary}\label{cor:abc}
For any $\ell \geq 0$, the following holds with implicit constants depending on $\ell$:
\begin{align}
	|  a^g |,\, | b^g |,\, |  c^g | \ls &\:  t^{\f 32} e^{\f{\alp \bt |x|^2}{2(\alp + \bt t^2)}} \|\jap{z}^{-\ell} g \|_{L^2_v}, \label{eq:abc.est} \\
	|Y_l  a^g |,\, | Y_l b^g |,\, | Y_l c^g | \ls &\: t^{\f 32} e^{\f{\alp \bt |x|^2}{2(\alp + \bt t^2)}} \Big(\jap{x/t}\|\jap{z}^{-\ell}  g \|_{L^2_v} + \| \jap{z}^{-\ell} Y g \|_{L^2_v}\Big). \label{eq:Yabc.est}
\end{align}
\end{corollary}
\begin{proof}
	By Lemma~\ref{lem:exp-abc}, $e^{-\f{\alp \bt |x|^2}{2(\alp + \bt t^2)}} a^g$, $e^{-\f{\alp \bt |x|^2}{2(\alp + \bt t^2)}} b^g$, $e^{-\f{\alp \bt |x|^2}{2(\alp + \bt t^2)}} c^g$ take the form of $\mathfrak m^g$ of Lemma~\ref{lem:a.b.c.schematic}. Thus \eqref{eq:abc.est} follows from Lemma~\ref{lem:a.b.c.schematic}.
	
	Using Lemma~\ref{lem:exp-abc}, \eqref{eq:Y-der-z} and $Y_l \f{|x|^2}{\alp + \bt t^2} = \f{2 t x_l}{\alp + \bt t^2}$, it follows that $Y_l  a^g$, $Y_l b^g$, $Y_l c^g $ take the schematic form
	\begin{equation}\label{eq:Y.abc.schematic}
		Y_l  a^g,\, Y_l b^g, \, Y_l c^g \eqs e^{\f{\alp\bt |x|^2}{2(\alp + \bt t^2)}} \jap{x/t} \mathfrak m^g + e^{\f{\alp\bt |x|^2}{2(\alp + \bt t^2)}}  \mathfrak m^{Yg}.
	\end{equation}
	From this we deduce \eqref{eq:Yabc.est} using Lemma~\ref{lem:a.b.c.schematic}. \qedhere
\end{proof}

\begin{corollary}\label{cor:Pi.in.weighted.space}
    For any $\kay \in \mathbb Z_{\geq 0}$ and $\ell \geq 0$,
    \begin{align}
        \| \jap{z}^{\ell} \Pi g \|_{\calD_{x,v}(0,\kay,1)} \ls &\: \| g\|_{\calD_{x,v}(0,\kay,1)},  \label{eq:Pi.in.weighted.space.2} \\
        \| w_{(0,\kay,0)} \jap{z}^{\ell} \Pi g \|_{L^2_{x,v}} \ls &\: \| w_{(0,\kay,0)} g\|_{L^2_{x,v}}.  \label{eq:Pi.in.weighted.space.3}
    \end{align}
\end{corollary}
\begin{proof}
    We compute using \eqref{eq:Pi.formula} and Corollary~\ref{cor:abc} that
    \begin{equation}\label{eq:Pi.in.weighted.space.1.prelim}
    \begin{split}
        &\: \| e^{-\f{\alp \bt|x|^2}{2(\alp + \bt t^2)}} w_{(0,\kay,1)} \jap{z}^{\ell} \Pi g \|_{L^2_{v}} \\
        \ls &\: \Big(  \| \jap{z}^{\ell+ 2} \sqrt{\mu} \|_{L^2_v} \Big)  e^{-\f{\alp^2(\bt - \bar{\bt})|x|^2}{2(\alp + \bt t^2)(\alp + \bar{\bt} t^2)}} \jap{x/t}^{\kay} |(a^{g},b^{g}, c^{g})|  \\
        \ls &\: (t^{-\f 32} e^{-\f{\alp \bt |x|^2}{2(\alp + \bt t^2)}} ) e^{-\f{\alp^2(\bt - \bar{\bt})|x|^2}{2(\alp + \bt t^2)(\alp + \bar{\bt} t^2)}}  \jap{x/t}^{\kay} (t^{\f32} e^{\f{\alp \bt |x|^2}{2(\alp + \bt t^2)}})\| \jap{z}^{-\f 12} g\|_{L^2_{v}} \\
        \ls &\: e^{-\f{\alp^2(\bt - \bar{\bt})|x|^2}{2(\alp + \bt t^2)(\alp + \bar{\bt} t^2)}} \jap{x/t}^{\kay} \|\jap{z}^{-\f 12} g\|_{L^2_{v}},
    \end{split}
\end{equation}   
where we used a change of variables together with Lemma~\ref{lem:z.x-tv.com} to compute
\begin{equation}\label{eq:decay.for.Pi.in.weighted.space}
    \| \jap{z}^{\ell + 2} \sqrt{\mu} \|_{L^2_v} \ls e^{-\f{\alp \bt |x|^2}{2(\alp + \bt t^2)}} \| e^{-|z|^2/4}\|_{L^2_v} \ls t^{-\f 32} e^{-\f{\alp \bt |x|^2}{2(\alp + \bt t^2)}}.
\end{equation}
    A similar argument gives
    \begin{equation}\label{eq:Pi.in.weighted.space.1.prelim.dv}
    \begin{split}
        &\: t^{-1} \| e^{-\f{\alp \bt|x|^2}{2(\alp + \bt t^2)}} w_{(0,\kay,1)} \jap{z}^{\ell} \rd_v \Pi g \|_{L^2_{v}} \\
        \ls &\: \Big(  \| \jap{z}^{\ell+ 3} \sqrt{\mu} \|_{L^2_v} \Big) e^{-\f{\alp^2(\bt - \bar{\bt})|x|^2}{2(\alp + \bt t^2)(\alp + \bar{\bt} t^2)}} \jap{x/t}^{\kay} |(a^{g},b^{g}, c^{g})|  \\
        \ls &\: e^{-\f{\alp^2(\bt - \bar{\bt})|x|^2}{2(\alp + \bt t^2)(\alp + \bar{\bt} t^2)}} \jap{x/t}^{\kay} \|\jap{z}^{-\f 12} g\|_{L^2_{v}}.
    \end{split}
\end{equation}  
    Recalling the definitions in Definition~\ref{def:dissipation-norm} and Definition~\ref{def:calD}, we have
    \begin{equation}
    \begin{split}
        \| \Pi g \|_{\calD_{x,v}(0,\kay,1)} 
        \ls &\: \| e^{-\f{\alp \bt |x|^2}{2(\alp + \bt t^2)}} w_{(0,\kay,1)}\Pi g \|_{L^2_{x,v}} + t^{-1} \| e^{-\f{\alp \bt |x|^2}{2(\alp + \bt t^2)}} w_{(0,\kay,1)} \rd_v \Pi g \|_{L^2_{x,v}}.
    \end{split}
\end{equation}
Combining this with the $L^2_x$ integrated version of \eqref{eq:Pi.in.weighted.space.1.prelim}, and \eqref{eq:Pi.in.weighted.space.1.prelim.dv}, we obtain \eqref{eq:Pi.in.weighted.space.2}. Essentially the same (but slightly simpler) argument as in \eqref{eq:Pi.in.weighted.space.1.prelim} also gives \eqref{eq:Pi.in.weighted.space.3}. \qedhere
\end{proof}


\begin{proposition}\label{prop:L.[Y,Pi]}
	For any $k \in \mathbb Z_{\geq 0}$, it holds that
	\begin{equation}
		\|\jap{z}^k L [Y_l,\Pi] g \|_{L^2_{v}} \ls t^{-1}  e^{-\f{\alp \bt |x|^2}{\alp + \bt t^2}} \|  \jap{z}^{-\f 12} g\|_{L^2_{v}} 
	\end{equation}
	for some implicit constant depending on $k$.
\end{proposition}
\begin{proof}
	Recalling \eqref{eq:A}, Proposition~\ref{prop:comm-Y-Pi} and Lemma~\ref{lem:Pi.schematic} give the bound for $\|\jap{z}^k A [Y_l,\Pi] g \|_{L^2_{v}}$. Using Proposition~\ref{prop:weighted-K-bound} and duality, we have $\|\jap{z}^k K [Y_l,\Pi] g \|_{L^2_{v}} \ls e^{-\f{\alp \bt |x|^2}{\alp + \bt t^2}} \| [Y_l,\Pi] g \|_{L^2_{v}}$, which implies the desired estimate after also using Proposition~\ref{prop:comm-Y-Pi} and Lemma~\ref{lem:Pi.schematic}. \qedhere
\end{proof}

\begin{lemma}
The following formula holds:
\begin{equation}\label{eq:L.energy.I.term.2.1}
\begin{split}
T \Pi g 
=&\:  \Big\{(\rd_t+\f{x_i\beta t}{\alpha+\beta t^2}\rd_{x_i})a^g +(\rd_t+\f{x_i\beta t}{\alpha+\beta t^2}\rd_{x_i})b^g_l\cdot z_l+(\rd_t+\f{x_i\beta t}{\alpha+\beta t^2}\rd_{x_i})c^g\cdot |z|^2 \Big\}\sqrt{\mu}\\
&\:+\f{1}{t\sqrt{\alpha+\beta t^2}}\{Y_ia^g +Y_i b^g_l \cdot z_l+Y_i c^g\cdot |z|^2\}z_i \sqrt{\mu} + \{b^g_l Tz_l+c^g T |z|^2\}\sqrt{\mu}.
\end{split}
\end{equation}
Moreover, 
\begin{equation}\label{eq:I-Pi.T.Pi}
\| (I-\Pi) T \Pi g \|_{L^2_v} \ls t^{-2}  \Big(\jap{x/t}\|  g \|_{L^2_v} + \| Y g \|_{L^2_v}\Big).
\end{equation}
\end{lemma}
\begin{proof}
\eqref{eq:L.energy.I.term.2.1} is an explicit computation after writing $\Pi g = a^g\sqrt{\mu}+b^g_l z_l \sqrt{\mu}+c^g|z|^2\sqrt{\mu}$. To obtain \eqref{eq:I-Pi.T.Pi}, note that after using $T z_i = -\f{\alp \bt x_i}{(\alp + \bt t^2)^{\f 32}}$, many terms are annihilated by $(I-\Pi)$, specifically, 
\begin{equation}
\begin{split}
	(I - \Pi) T \Pi g  = \f{1}{t\sqrt{\alpha+\beta t^2}}\Big\{Y_i b^g_l (I-\Pi) (z_l z_i \sqrt{\mu}) + Y_i c^g (I-\Pi)( |z|^2 z_i \sqrt{\mu}) \Big\}.
\end{split}
\end{equation}
Since $(I-\Pi)$ is (by definition) bounded on $L^2_v$, it follows that 
\begin{equation*}
	\|(I - \Pi) T \Pi g\|_{L^2_v} \ls t^{-2} \Big( |Yb^g| \|z_l z_i \sqrt{\mu}\|_{L^2_v} + |Yc^g| \||z|^2 z_i \sqrt{\mu}\|_{L^2_v}\Big) \ls t^{-2} \times t^{-\f 32} e^{-\f{\alp \bt |x|^2}{2(\alp + \bt t^2)}} (|Yb^g| + |Yc^g|),
\end{equation*}
where we used Proposition~\ref{prop:aux-integral} to estimate $\|z_l z_i \sqrt{\mu}\|_{L^2_v} $, $\||z|^2 z_i \sqrt{\mu}\|_{L^2_v}$. The estimate \eqref{eq:I-Pi.T.Pi} then follows from \eqref{eq:Yabc.est}. \qedhere
\end{proof}




\subsection{Main $[Y,L]$ commutation}

\begin{proposition}\label{prop:Y.L}
	\begin{enumerate}
	\item	For $\vartheta =1$, $\kay \in \mathbb Z_{\geq 0}$ and $\iota = 0$, the following holds: 
	\begin{equation}\label{eq:main.Y.L}
		\begin{split}
			&\: \Big| \langle w_{(\vartheta,\kay,\iota)}^2 [L,Y_l]g, Y_l g \rangle_{L^2_{x,v}} \Big| \\
			\ls &\: \| g \|_{\calD_{x,v}(\vartheta,\kay+1,\iota)} \|  Y_l g \|_{\calD_{x,v}(\vartheta,\kay,\iota)} + t^{-2} \| w_{(\vartheta,\kay+1,\iota)}g \|_{L^2_{x,v}} \| w_{(\vartheta,\kay,\iota)} Y_l g \|_{L^2_{x,v}}.
		\end{split}
	\end{equation}	
	\item If $\vartheta = 0$, we have the refined estimate
	\begin{equation}\label{eq:main.Y.L.I-Pi}
		\begin{split}
			&\: \Big| \langle w^2_{(0,\kay,\iota)} [L,Y_l]g, Y_l g \rangle_{L^2_{x,v}} \Big| \\
			\ls &\: \Big(  \| (I-\Pi) g \|_{\calD_{x,v}(0,\kay+1,\iota)} + t^{-1} \| g \|_{\calD_{x,v}(0,\kay,\iota)} \Big) \| (I-\Pi) Y_l g \|_{\calD_{x,v}(0,\kay,\iota)} \\
	&\: + t^{-1} \| (I-\Pi) g \|_{\calD_{x,v}(0,\kay,\iota)} \| Y_l g \|_{\calD_{x,v}(0,\kay,\iota)} + t^{-2} \| w_{(0,\kay+1,\iota)} g \|_{L^2_{x,v}} \| w_{(0,\kay,\iota)} Y_l g \|_{L^2_{x,v}}.
		\end{split}
	\end{equation}
	\end{enumerate}
\end{proposition}
\begin{proof}
The first estimate is an immediate consequence of Proposition~\ref{prop:Y.A} and Proposition~\ref{prop:Y.K}.

We turn to the refined estimate, where $w = w_{(0,\kay,\iota)}$. We use that $w$ depends only on $(t,x)$ and integrate by parts in $Y_l$. In the following, we will repeatedly use the fact that $L h = L(I-\Pi)h$ and $\langle w^2 h_1, L h_2 \rangle_{L^2_v} = \langle w^2 (I-\Pi) h_1, L h_2 \rangle_{L^2_v}$ (which follows from the self-adjointness of $L$).
\begin{equation}
\begin{split}
	&\:  \langle w^2 [L,Y_l]g, Y_l g \rangle_{L^2_{x,v}} \\
	= &\:  \langle w^2 L(I-\Pi) Y_l g, (I-\Pi) Y_l g \rangle_{L^2_{x,v}} - \langle w^2 Y_l L (I-\Pi) g,  Y_l g \rangle_{L^2_{x,v}} \\
	= &\: \langle w^2 L(I-\Pi) Y_l g, (I-\Pi) Y_l g \rangle_{L^2_{x,v}}  + \langle w^2  L (I-\Pi) g,  Y_l^2 g \rangle_{L^2_{x,v}} + \langle (t \rd_{x_l} w^2)  L (I-\Pi) g,  Y_l g \rangle_{L^2_{x,v}} \\
	= &\: \langle w^2 L(I-\Pi) Y_l g, (I-\Pi) Y_l g \rangle_{L^2_{x,v}}  + \langle w^2  L (I-\Pi) g,  (I-\Pi) Y_l^2 g \rangle_{L^2_{x,v}} \\
	&\: + \langle (t \rd_{x_l} w^2)  L (I-\Pi) g,  (I-\Pi) Y_l g \rangle_{L^2_{x,v}} \\
	= &\: \langle w^2 L(I-\Pi) Y_l g, (I-\Pi) Y_l g \rangle_{L^2_{x,v}}  + \langle w^2  L (I-\Pi) g,  Y_l (I-\Pi) Y_l g \rangle_{L^2_{x,v}} \\
	&\: + \langle w^2  L (I-\Pi) g,  [Y_l, \Pi] Y_l g \rangle_{L^2_{x,v}} + \langle (t \rd_{x_l} w^2)  L (I-\Pi) g,  (I-\Pi) Y_l g \rangle_{L^2_{x,v}} \\
	= &\: \langle w^2 L(I-\Pi) Y_l g, (I-\Pi) Y_l g \rangle_{L^2_{x,v}}  - \langle w^2  Y_l L (I-\Pi) g,   (I-\Pi) Y_l g \rangle_{L^2_{x,v}} \\
	&\: + \langle w^2  L (I-\Pi) g,  [Y_l, \Pi] Y_l g \rangle_{L^2_{x,v}} \\
	= &\: \langle w^2 [L,Y_l] (I-\Pi) g, (I-\Pi) Y_l g \rangle_{L^2_{x,v}}  + \langle w^2 L [Y_l,\Pi] g, (I-\Pi) Y_l g \rangle_{L^2_{x,v}}  \\
	&\: + \langle w^2  L (I-\Pi) g,  [Y_l, \Pi] Y_l g \rangle_{L^2_{x,v}} =: I + II + III.
	\end{split}
\end{equation}

The term $I$ is the main term, for which we use the commutator estimates in Proposition~\ref{prop:Y.A} and Proposition~\ref{prop:Y.K} to obtain
\begin{equation}
	\begin{split}
	|I| \ls &\: |\langle w^2 [L,Y_l] (I-\Pi) g, (I-\Pi) Y_l g \rangle_{L^2_{x,v}}| \\
	\ls &\: \| (I-\Pi) g \|_{\calD_{x,v}(0,\kay+1,\iota)} \| (I-\Pi) Y_l g \|_{\calD_{x,v}(0,\kay,\iota)} \\
	&\: + t^{-2} \| w_{(0,\kay+1,\iota)} (I-\Pi) g \|_{L^2_{x,v}} \| w (I-\Pi) Y_l g \|_{L^2_{x,v}} \\
	\ls &\: \| (I-\Pi) g \|_{\calD_{x,v}(0,\kay+1,\iota)} \| (I-\Pi) Y_l g \|_{\calD_{x,v}(0,\kay,\iota)} + t^{-2} \| w_{(0,\kay+1,\iota)} g \|_{L^2_{x,v}} \| w Y_l g \|_{L^2_{x,v}},
	\end{split}
\end{equation}
where in the last line we used that the weights are $v$-independent.

For the term $II$, we use Cauchy--Schwarz together with Proposition~\ref{prop:L.[Y,Pi]} to obtain
\begin{equation}
	\begin{split}
		|II| \ls &\: |\langle w^2 L [Y_l,\Pi] g, (I-\Pi) Y_l g \rangle_{L^2_{x,v}}| \\
		\ls &\: \| w e^{\f{\alp\bt |x|^2}{2(\alp + \bt t^2)}} \jap{z}^{\f 12} L [Y_l,\Pi] g \|_{L^2_{x,v}} \| w e^{-\f{\alp\bt |x|^2}{2(\alp + \bt t^2)}} \jap{z}^{-\f 12} (I-\Pi) Y_l g \|_{L^2_{x,v}} \\
		\ls &\: t^{-1} \| w e^{-\f{\alp\bt |x|^2}{2(\alp + \bt t^2)}} \jap{z}^{-\f 12} g \|_{L^2_{x,v}}  \| w e^{-\f{\alp\bt |x|^2}{2(\alp + \bt t^2)}} \jap{z}^{-\f 12} (I-\Pi) Y_l g \|_{L^2_{x,v}} \\
		\ls &\: t^{-1} \| g \|_{\calD_{x,v}(0,\kay,\iota)} \| (I-\Pi) Y_l g \|_{\calD_{x,v}(0,\kay,\iota)}.
	\end{split}
\end{equation}

For $III$, we use the bound for $L = -A-K$ in Proposition~\ref{prop:weighted-K-bound} and Proposition~\ref{prop:weighted-A-upper-bound} to obtain
\begin{equation}
	\begin{split}
		|III| \ls &\: \Big| \langle w^2  L (I-\Pi) g,  [Y_l, \Pi] Y_l g \rangle_{L^2_{x,v}} \Big| \\
		\ls &\: \| (I-\Pi) g \|_{\calD_{x,v}(0,\kay,\iota)} \| [Y_l, \Pi] Y_l g \|_{\calD_{x,v}(0,\kay,\iota)} \\
		&\: + t^{-2} \| w (I-\Pi) g \|_{L^2_{x,v}} \| w [Y_l, \Pi] Y_l g \|_{L^2_{x,v}} \\
		\ls &\: t^{-1} \| (I-\Pi) g \|_{\calD_{x,v}(0,\kay,\iota)} \| Y_l g \|_{\calD_{x,v}(0,\kay,\iota)} + t^{-3} \| w g \|_{L^2_{x,v}} \| w Y_l g \|_{L^2_{x,v}},
	\end{split}
\end{equation}
where we used Proposition~\ref{prop:comm-Y-Pi} and Lemma~\ref{lem:Pi.schematic}, as well as $ \| w (I-\Pi) g \|_{L^2_{x,v}} \ls  \| w g \|_{L^2_{x,v}}$ (since $w$ is independent of $v$). Combining all the estimates yields \eqref{eq:main.Y.L.I-Pi}. \qedhere
\end{proof}

\subsection{Commutator of $L$ with the transport operator}
In this subsection we will compute the commutator $[L,T]= [L,\rd_t+v_i\rd_{x_i}].$

\begin{lemma}\label{lem:T-der-exp}
Define 
\begin{equation}\label{eq:hat.sigma.def}
	\hat{\sigma}_{ij} := - 2\alp t^{-1} \int_{\RR^3} \phi_{ij}(v-v_*) (v-v_*)\cdot v_* \mu_* \ud v_*.
\end{equation}
Then the following holds: 
\begin{align}
	T \sigma_{ij} = &\: -2t^{-1}\sigma_{ij} + \hat{\sigma}_{ij},  \label{eq:T.on.coeff.1}\\
	T (\sigma_{ij} z_j) = &\: -2t^{-1} \sigma_{ij} z_j + \hat{\sigma}_{ij} z_j - \f{\alp \bt \sigma_{ij} x_j}{(\alp + \bt t^2)^{\f 32}}, \label{eq:T.on.coeff.2}\\
	T (\sigma_{ij} z_i z_j) =&\: - 2 t^{-1} \sigma_{ij} z_i z_j + \hat{\sigma}_{ij} z_i z_j - \f{2 \alp \bt \sigma_{ij} x_i z_j}{(\alp + \bt t^2)^{\f 32}}. \label{eq:T.on.coeff.3}
\end{align}
\end{lemma}
\begin{proof}
Since $(T\mu)(v_*) = 0$, we have
\begin{equation}
	\begin{split}
		T \sigma_{ij} = &\: \int_{\RR^3} \phi_{ij}(v-v_*) (v-v_*)_l \rd_{x_l} \mu_* \ud v_* = - 2\bt \int_{\RR^3} \phi_{ij}(v-v_*) (v-v_*)\cdot (x-tv_*) \mu_* \ud v_* \\
		= &\: -  t^{-1} \int_{\RR^3} \phi_{ij}(v-v_*) (v-v_*)\cdot  \nabla_{v_*} \mu_* \ud v_* - 2\alp t^{-1} \int_{\RR^3} \phi_{ij}(v-v_*) (v-v_*)\cdot  v_* \mu_* \ud v_*.
	\end{split}
\end{equation}
The second term is exactly $\hat{\sigma}_{ij}$. For the first term, we integrate by parts and use
\begin{equation}\label{eq:div.comp}
\rd_{v_{*,l}} (\phi_{ij}(v-v_*) (v-v_*)_l) = - 2 \phi_{ij}(v-v_*)
\end{equation}
to obtain \eqref{eq:T.on.coeff.1}. The two identities \eqref{eq:T.on.coeff.2}--\eqref{eq:T.on.coeff.3} follow from combining the above with 
$T z_i = -\f{\alp \bt x_i}{(\alp + \bt t^2)^{\f 32}}$. \qedhere
\end{proof}

\begin{lemma}\label{lem:hat.sigma}
The quantity $\hat{\sigma}_{ij}$ satisfies the following estimates:
\begin{equation}
|\hat{\sigma}_{ij}| ,\, |\hat{\sigma}_{ij} z_j |, \, |\hat{\sigma}_{ij} z_i z_j| \ls (t^{-5} + t^{-4} \f{|x|}{t}) e^{-\f{\alp\bt |x|^2}{\alp + \bt t^2}},\quad |\rd_{v_j} \hat{\sigma}_{ij}|,\, |(\rd_{v_j} \hat{\sigma}_{ij}) z_i | \ls (t^{-4} + t^{-3} \f{|x|}{t}) e^{-\f{\alp\bt |x|^2}{\alp + \bt t^2}}.
\end{equation}
Moreover,
\begin{equation}\label{eq:hat.sigma.A}
\Big| \jap{ w^2_{(0,\kay,1)} g_1, \rd_{v_i} (\hat{\sigma}_{ij} \rd_{v_j} g_2) }_{L^2_{x,v}} \Big| \ls t^{-2} \| g_1 \|_{\calD_{x,v}(0,\kay,1)}  \| \jap{z} g_2 \|_{\calD_{x,v}(0,\kay+1,1)}.
\end{equation}
\end{lemma}
\begin{proof}
 	Note that $|\phi_{ij}(v-v_*) (v-v_*)| \ls 1$. Hence, after writing $v_* = \f{z_*}{\sqrt{\alp + \bt t^2}} + \f{x \bt t}{\alp + \bt t^2}$, the bound for $|\hat{\sigma}_{ij}|$ follows from \eqref{eq:int-no-phi}. To show the improvements when contracting with $z$, we observe that $\phi_{ij}(v-v_*) z_i = \phi_{ij}(v-v_*) z_{*i}$, $\phi_{ij}(v-v_*) z_i z_j = \phi_{ij}(v-v_*) z_{*i} z_{*j}$, and then again argue with Proposition~\ref{prop:aux-integral}. These estimates imply \eqref{eq:hat.sigma.A} after integrating by parts in $\rd_{v_i}$ and decomposing $\rd_v g_k = P_z(\rd_v g_k) + (I - P_z) (\rd_v g_k)$.
    
    For $\rd_{v_i} \hat{\sigma}_{ij}$, we compute
	\begin{equation}
\begin{split}
\rd_{v_i} \hat{\sigma}_{ij} =&\:  - 2\alp t^{-1} \int_{\RR^3} \phi_{ij}(v-v_*) (v-v_*)_l \rd_{v_{*i}} (v_{l,*} \mu_*) \ud v_*  \\
=&\:  4\alp t^{-1}\sqrt{\alp + \bt t^2}\int_{\RR^3} \phi_{ij}(v-v_*) (v-v_*)\cdot v_* z_{*i} \mu_* \ud v_*,
\end{split}
\end{equation}
%
where we used $\phi_{ij}(v-v_*) (v-v_*)_i = 0$. Now we write $v_* = \f{z_*}{\sqrt{\alp + \bt t^2}} + \f{x \bt t}{\alp + \bt t^2}$ so that we have the desired estimate after using Proposition~\ref{prop:aux-integral}. The improvement when contracted with $z_i$ can be obtained as before. \qedhere

\end{proof}

\begin{lemma}\label{lem:Tdvsigma}
	For $\sigma_i = \sigma_{ij} z_j$, it holds that
	$$|T (\rd_{v_i} \sigma_i) + t^{-1} \rd_{v_i} \sigma_i| \ls (t^{-3}  \f{|x|}{t}+ t^{-3} )e^{-\f{\alp\bt |x|^2}{\alp + \bt t^2}}.$$
\end{lemma}
\begin{proof}
	By \eqref{eq:T.on.coeff.2}, we have
	\begin{equation}
		\rd_{v_i} T \sigma_i + 2 t^{-1} \rd_{v_i} \sigma_i = \rd_{v_i} (\hat{\sigma}_{ij} z_j) - \f{\alp \bt \rd_{v_i} \sigma_{ij} x_j }{(\alp + \bt t^2)^{\f 32}}.
	\end{equation}
	The first term is controlled by $(t^{-4} + t^{-3}\f{|x|}{t}) e^{-\f{\alp \bt |x|^2}{\alp + \bt t^2}}$ using Lemma~\ref{lem:hat.sigma} and the second term is controlled by $t^{-3} \jap{z}^{-2} \f{|x|}{t} e^{-\f{\alp \bt |x|^2}{\alp + \bt t^2}}$. 
	Now, observe
	\begin{equation}\label{eq:T.dv.comm}
[T, \rd_{v_l}]= -\rd_{x_l} = t^{-1}(-Y_l + \rd_{v_l})
\end{equation}
and thus
	\begin{equation}
		T (\rd_{v_i} \sigma_i)  = \rd_{v_i} T \sigma_i + t^{-1}(-Y_i \sigma_i  + \rd_{v_i} \sigma_i).
	\end{equation}
	The term $t^{-1}Y_i \sigma_i$ can be bounded by $(t^{-3} \jap{z}^{-2} \f{|x|}{t}+ t^{-3} \jap{z}^{-1} )e^{-\f{\alp\bt |x|^2}{\alp + \bt t^2}}$ using Lemma~\ref{lem:Y-der-exp}. The term $t^{-1} \rd_{v_i} \sigma_i$ is to be treated as a main term and so altogether we have the desired conclusion.
\end{proof}

\begin{proposition}\label{prop:T.A.commute}
	\begin{equation}
		\begin{split}
			\Big| \langle w_{(0,\kay,1)}^2 g_1, [T,A] g_2 \rangle_{L^2_{x,v}} \Big| \ls t^{-2} \| g_1 \|_{\calD_{x,v}(0,\kay,1)} \Big(  \| \jap{z} g_2 \|_{\calD_{x,v}(0,\kay+1,1)} + \| Y g_2 \|_{\calD_{x,v}(0,\kay+1,1)} \Big).
		\end{split}
	\end{equation}
\end{proposition}
\begin{proof}
We control the commutator for each term constituting $A$. Using \eqref{eq:T.dv.comm} and  Lemma~\ref{lem:T-der-exp}, we obtain
\begin{equation}
\begin{split}
	&\: T \rd_{v_i} (\sigma_{ij} \rd_{v_j} g) - \rd_{v_i} (\sigma_{ij} \rd_{v_j} T g) \\
	= &\: \rd_{v_i} (\hat{\sigma}_{ij} \rd_{v_j} g) - t^{-1} Y_i (\sigma_{ij} \rd_{v_j} g) -t^{-1} \rd_{v_i} (\sigma_{ij} Y_j g),
\end{split}
\end{equation}
where we note the cancellation for $t^{-1} \rd_{v_i} (\sigma_{ij} \rd_{v_j} g)$. The term $| \langle w_{(0,\kay,1)}^2 g_1, \rd_{v_i} (\hat{\sigma}_{ij} \rd_{v_j} g_2) \rangle_{L^2_{x,v}} |$ can be estimated by Lemma~\ref{lem:hat.sigma}. For the other two terms, we use Proposition~\ref{prop:Y.on.coeff.est} to obtain

\begin{equation}
\begin{split}
&\: t^{-1} \Big| \langle w_{(0,\kay,1)}^2 g_1, Y_i (\sigma_{ij} \rd_{v_j} g_2) \rangle_{L^2_{x,v}} \Big| + t^{-1} \Big| \langle w_{(0,\kay,1)}^2 g_1, \rd_{v_i} (\sigma_{ij} Y_j g_2) \rangle_{L^2_{x,v}} \Big| \\
\ls &\: e^{-\f{\alp\bt |x|^2}{\alp + \bt t^2}} t^{-3} \f{|x|}{t} \Big( \langle w_{(0,\kay,1)}^2 |g_1|, \jap{z}^{-2} |P_z\rd_{v} g_2| \rangle_{L^2_{x,v}} + \langle w_{(0,\kay,1)}^2 |g_1|, \jap{z}^{-1} |(I-P_z)\rd_{v} g_2| \rangle_{L^2_{x,v}} \Big)\\
&\: + e^{-\f{\alp\bt |x|^2}{\alp + \bt t^2}}  \Big( t^{-4} \langle w_{(0,\kay,1)}^2 |g_1| , \jap{z}^{-2} |P_z\rd_v g_2| \rangle_{L^2_{x,v}} + t^{-4} \langle w_{(0,\kay,1)}^2 |g_1| , \jap{z}^{-1} |(I-P_z)\rd_v g_2| \rangle_{L^2_{x,v}}    \\
&\: + t^{-3} \langle w_{(0,\kay,1)}^2|g_1| , \jap{z}^{-2} |P_z \rd_v Y g_2| \rangle_{L^2_{x,v}} + t^{-3} \langle w_{(0,\kay,1)}^2|g_1| , \jap{z}^{-1} |(I-P_z) \rd_v Y g_2| \rangle_{L^2_{x,v}}   \Big) \\
&\: + e^{-\f{\alp\bt |x|^2}{\alp + \bt t^2}} \Big( t^{-3}  \langle w_{(0,\kay,1)}^2 |P_z \rd_v g_1|, \jap{z}^{-2} |Yg_2| \rangle_{L^2_{x,v}} + t^{-3}  \langle w_{(0,\kay,1)}^2 |(I-P_z)\rd_v g_1|, \jap{z}^{-1} |Yg_2| \rangle_{L^2_{x,v}} \Big)\\
\ls &\: \| g_1 \|_{\calD_{x,v}(0,\kay,1)} \Big( t^{-2}  \| g_2 \|_{\calD_{x,v}(0,\kay+1,1)} + t^{-2} \| Y g_2 \|_{\calD_{x,v}(0,\kay+1,1)} \Big).
\end{split}
\end{equation} Here, we have in particular integrated by parts in $\rd_{v_i}$ for the term $t^{-1} \rd_{v_i} (\sigma_{ij} Y_j g_2)$.

Next, we consider
\begin{equation}\label{eq:T.A.main.2}
\begin{split}
	&\: \Big| T \Big( \sqrt{\alp + \bt t^2} \rd_{v_i} \sigma_i g_2 \Big)  - \sqrt{\alp + \bt t^2} \rd_{v_i} \sigma_i (Tg_2) \Big| \\
	= &\: \Big| \f{\bt t}{\alp + \bt t^2} \sqrt{\alp + \bt t^2} \rd_{v_i} \sigma_i g_2 + \sqrt{\alp + \bt t^2} (T\rd_{v_i} \sigma_i) g_2 \Big| \\
	\ls &\: (t^{-2}  \f{|x|}{t}+ t^{-2} )e^{-\f{\alp\bt |x|^2}{\alp + \bt t^2}}|g_2|,
\end{split}
\end{equation}
where we used Lemma~\ref{lem:Tdvsigma} and noted a cancellation ($|\f{\bt t}{\alp + \bt t^2} - t^{-1} | \ls t^{-3} $) of the slowest decaying terms.

Finally, we consider the following term:
\begin{equation}\label{eq:T.A.main.3}
\begin{split}
	&\: \Big| T \Big( (\alp + \bt t^2) \sigma_{ij} z_i z_j g_2 \Big)  -  (\alp + \bt t^2) \sigma_{ij} z_i z_j (Tg_2)  \Big| \\
	= &\: \Big| - 2 \alp t^{-1} \sigma_{ij} z_i z_j g_2 + (\alp + \bt t^2)\Big( \hat{\sigma}_{ij} z_i z_j - \f{2 \alp \bt \sigma_{ij} x_i z_j}{(\alp + \bt t^2)^{\f 32}} \Big) g_2\Big| \\
	\ls &\: \Big( t^{-3} + t^{-2} \f{|x|}{t} \Big) e^{-\f{\alp \bt |x|^2}{\alp + \bt t^2}} |g_2|,
\end{split}
\end{equation}
where we used Lemma~\ref{lem:T-der-exp} and Lemma~\ref{lem:hat.sigma}.

For terms in \eqref{eq:T.A.main.2} and \eqref{eq:T.A.main.3}, we have
\begin{equation}
\begin{split}
	\Big| \jap{ w_{(0,\kay,1)}^2 g_1, \hbox{RHS of \eqref{eq:T.A.main.2}, \eqref{eq:T.A.main.3}} }_{L^2_{x,v}} \Big| \ls t^{-2}\| g_1 \|_{\calD_{x,v}(0,\kay,1)} \| \jap{z} g_2 \|_{\calD_{x,v}(0,\kay+1,1)} , 
\end{split}
\end{equation}
after using Lemma~\ref{lem:weight.compare}. Hence all terms are acceptable. \qedhere

\end{proof}

We now turn to the commutator $[T,K]$. In analogy with Section~\ref{sec:Y.K.comm}, it is convenient to first consider the commutator of $T$ with $\mathfrak{K}_{ij}$.
\begin{proposition}\label{prop:T.fraK.comm}
Let $\mathfrak{K}_{ij}$ be as in Definition~\ref{def:frak-K-def}. Then
\begin{align}
[T, \mathfrak{K}_{ij}] g + 2 t^{-1} \mathfrak{K}_{ij} g = &\:  t^{-1} \mu^{\f 12} \int_{\R^3}(v-v_*)_l \phi_{ij}(v-v_*)    \mu_*^{\f 12}( (Y_lg)_*-\alp v_{*l} g_*)\ud v_*.
\end{align}
\end{proposition}
\begin{proof}
Noting that $T\mu = 0$ (and $(T\mu)_*=0$), we have
\begin{equation}
\begin{split}
T \mathfrak{K}_{ij} g = &\: T \Big( \mu^{\f 12} \int_{\R^3}\phi_{ij}(v-v_*)\mu_*^{\f 12} g(v_*)\ud v_*\Big) \\
= &\: \mu^{\f 12} \int_{\R^3}\phi_{ij}(v-v_*)\mu_*^{\f 12} (Tg)(v_*)\ud v_* + \mu^{\f 12} \int_{\R^3}\phi_{ij}(v-v_*)    (v-v_*)_l  \rd_{x_l} (\mu_*^{\f 12} g_*)\ud v_*.
\end{split}
\end{equation}
Thus, rewriting $\rd_{x_l} = t^{-1} (Y_{*l} - \rd_{v_{*l}})$, we have
\begin{equation}
\begin{split}
[T, \mathfrak{K}_{ij}]g = &\:  t^{-1} \mu^{\f 12} \int_{\R^3}(v-v_*)_l \phi_{ij}(v-v_*)  (Y_{*l} - \rd_{v_{*l}}) (\mu_*^{\f 12} g(v_*))\ud v_*.
\end{split}
\end{equation}
Observing that $Y_{*l} \mu_*^{\f 12}=-\alp v_{*l} \mu_*^{\f 12}$ and using \eqref{eq:div.comp}, we obtain
\begin{equation}
\begin{split}
[T, \mathfrak{K}_{ij}]g = &\:  t^{-1} \mu^{\f 12} \int_{\R^3}(v-v_*)_l \phi_{ij}(v-v_*)    \mu_*^{\f 12}( (Y_lg)_*-\alp v_{*l} g_*)\ud v_* - 2  t^{-1} \mathfrak{K}_{ij} g.
\end{split}
\end{equation}
\end{proof}

\begin{proposition}\label{prop:T.K.commute}
	\begin{equation}
		\begin{split}
			\Big| \langle w_{(0,\kay,1)}^2 g_1, [T,K] g_2 \rangle_{L^2_{x,v}} \Big| \ls &\: t^{-2} \| g_1 \|_{\calD_{x,v}(0,\kay,1)} \Big(  \| g_2 \|_{\calD_{x,v}(0,\kay+1,1)} + \| Y g_2 \|_{\calD_{x,v}(0,\kay,1)} \Big).
		\end{split}
	\end{equation}
\end{proposition}
\begin{proof}
We start with \eqref{eq:K.for.Y.commutator}, which we recall:
\begin{equation}\label{eq:K.for.Y.commutator.in.T.K}
	K g_2 = 8\pi \mu g_2  + 4(\alp + \bt t^2)\, \mathfrak K_{ij}\big(z_i z_j g_2\big) - 2 (\alp + \bt t^2)\, \mathfrak K_{ii}(g_2).
\end{equation}
Clearly $[T, 8\pi \mu] g_2= 0$ since $T\mu = 0$. For the second term, we compute using $T z_i = -\f{\alp \bt x_i}{(\alp + \bt t^2)^{\f 32}}$ that
\begin{equation}\label{eq:T.K.commute.2}
\begin{split}
	&\: \Big| T ((\alp + \bt t^2) \mathfrak K_{ij} (z_i z_j g_2)) - (\alp + \bt t^2) \mathfrak K_{ij} (z_i z_j T g_2) \Big| \\
	= &\: \Big| 2\bt t \mathfrak K_{ij} (z_i z_j g_2) + (\alp + \bt t^2) [T, \mathfrak K_{ij}] (z_i z_j g_2) + (\alp + \bt t^2) \f{\alp \bt x_i}{(\alp + \bt t^2)^{\f 32}} \mathfrak K_{ij} (z_j g_2) \Big|.
\end{split}
\end{equation}
Using Proposition~\ref{prop:T.fraK.comm} and \eqref{eq:Y-der-z}, and bounding $\phi_{ij}(v-v_*)$ by $|v-v_*|^{-1}$, and $v$ by $t^{-1}|z| + |x|/t$, we have 
\begin{equation}
    \begin{split}
        &\: \Big| 2\bt t \mathfrak K_{ij} (z_i z_j g_2) + (\alp + \bt t^2) [T, \mathfrak K_{ij}] (z_i z_j g_2) \Big| \\
        \ls &\: \Big|  t^{-1}  \mathfrak K_{ij} (z_i z_j g_2) \Big| + t \int_{\bbR^3} \mu^{\f 12} \mu_*^{\f 12} \Big( (t^{-1} \jap{z_*}^3 + \jap{z_*}^2|x|/t) |(g_2)_*| + \jap{z_*}^2 |(Yg_2)_*| \Big) \ud v_*.
    \end{split}
\end{equation}
Thus using also Lemma~\ref{lem:fraK.direct}, we have 
\begin{equation}
    \begin{split}
        &\: \hbox{RHS of \eqref{eq:T.K.commute.2}} \\
        \ls &\: \int_{\bbR^3} \mu^{\f 12} \mu_*^{\f 12} |v-v_*|^{-1} \Big( t^{-1} \jap{z_*}^2 + \jap{z_*} |x|/t\Big) |(g_2)_*| \ud v_* \\
        &\: + t \int_{\bbR^3} \mu^{\f 12} \mu_*^{\f 12}   \Big((t^{-1}\jap{z_*}^3 + \jap{z_*}^2 |x|/t) |(g_2)_*| + \jap{z_*}^2 |(Yg_2)_*| \Big) \ud v_* .
    \end{split}
\end{equation}
Thus,
\begin{equation}
    \begin{split}
        &\: \Big| \jap{w^2_{(0,\kay,1)} g_1, \hbox{RHS of \eqref{eq:T.K.commute.2}}}_{L^2_{x,v}} \Big| \\
        \ls &\: \int_{\bbR^3\times \bbR^3 \times \bbR^3} w^2_{(0,\kay,1)}\mu^{\f 12} \mu_*^{\f 12}  |v-v_*|^{-1} \Big(t^{-1}\jap{z_*}^2 + \jap{z_*} |x|/t \Big) |g_1| |(g_2)_*| \ud v_* \ud v \ud x \\
        &\: + t \int_{\bbR^3\times \bbR^3 \times \bbR^3}  w^2_{(0,\kay,1)}\mu^{\f 12} \mu_*^{\f 12} |g_1|\Big( (t^{-1} \jap{z_*}^3 + \jap{z_*}^2 |x|/t) |(g_2)_*| + \jap{z_*}^2 |(Yg_2)_*| \Big) \ud v_* \ud v \ud x \\
        =:&\: I + II.
    \end{split}
\end{equation}
The term $I$ can be bounded by using the estimates for \eqref{eq:calJ3.def} in \eqref{eq:weighted-K-0} so that
\begin{equation}
    \begin{split}
        I \ls &\: t^{-3} \int_{\bbR^3} \| e^{-\f{\alp \bt |x|^2}{2(\alp + \bt t^2)}} w_{(0,\kay,1)} \jap{z}^{-\f 12} g_1 \|_{L^2_v} \| e^{-\f{\alp \bt |x|^2}{2(\alp + \bt t^2)}} w_{(0,\kay,1)} \jap{z}^{-\f 12} g_2 \|_{L^2_v} \ud x \\
        &\: + t^{-2} \int_{\bbR^3} \| e^{-\f{\alp \bt |x|^2}{2(\alp + \bt t^2)}} w_{(0,\kay,1)} \jap{z}^{-\f 12} g_1 \|_{L^2_v} \| e^{-\f{\alp \bt |x|^2}{2(\alp + \bt t^2)}} w_{(0,\kay+1,1)} \jap{z}^{-\f 12} g_2 \|_{L^2_v} \ud x \\
        \ls &\: t^{-2}\| g_1 \|_{\calD_{x,v}(0,\kay,1)} \| g_2 \|_{\calD_{x,v}(0,\kay+1,1)}.
    \end{split}
\end{equation}

For $II$, we bound 
$$\mu^{\f 12} \ls e^{-\f{\alp \bt |x|^2}{2(\alp + \bt t^2)}}e^{-|z|^2/4} \jap{z}^{-\f 12},\quad  \mu_*^{\f 12} \jap{z_*}^3 \ls e^{-\f{\alp \bt |x|^2}{2(\alp + \bt t^2)}}e^{-|z_*|^2/4} \jap{z_*}^{-\f 12}$$
and use Cauchy--Schwarz in both $v$ and $v_*$ so that
\begin{equation}
    \begin{split}
        II\ls &\: t \Big(\sup_x \int_{\bbR^3} e^{-|z|^2/4} \ud v\Big)\\
        &\: \quad \times \int_{\bbR^3} e^{-\f{\alp \bt |x|^2}{\alp + \bt t^2}}\| w_{(0,\kay,1)} \jap{z}^{-\f 12} g_1 \|_{L^2_v} (\| w_{(0,\kay+1,1)} \jap{z}^{-\f 12} g_2 \|_{L^2_v} + \| w_{(0,\kay,1)} \jap{z}^{-\f 12} Y g_2 \|_{L^2_v}) \ud x \\
        \ls &\: t^{-2} \| g_1 \|_{\calD_{x,v}(0,\kay,1)} \Big(  \| g_2 \|_{\calD_{x,v}(0,\kay+1,1)} + \| Y g_2 \|_{\calD_{x,v}(0,\kay,1)} \Big).
    \end{split}
\end{equation}
Thus, altogether the commutator term associated with the second term in \eqref{eq:K.for.Y.commutator.in.T.K} obeys the desired bound.

The final term in \eqref{eq:K.for.Y.commutator.in.T.K} can be treated in a similar manner and is in fact a bit simpler. We omit the details. \qedhere
\end{proof}

Now we can finally estimate the full $[L,T]$ commutator.
\begin{proposition}\label{prop:comm-L-T}
For $\kay \in \mathbb Z_{ \geq 0}$, the following estimate holds:
\begin{align*}
&\: \Big| \jap{ w^2_{(0,\kay,1)}g_1, [L,T] g_2 }_{L^2_{x,v}} \Big| \ls t^{-2} \| g_1 \|_{\calD_{x,v}(0,\kay,1)} \Big(  \| \jap{z} g_2 \|_{\calD_{x,v}(0,\kay+1,1)} + \| Y g_2 \|_{\calD_{x,v}(0,\kay+1,1)} \Big).
\end{align*}
\end{proposition}
\begin{proof}
This is a consequence of Propositions~\ref{prop:T.A.commute} and \ref{prop:T.K.commute}. \qedhere
\end{proof}

\section{Energy estimates for $f$ and $Y_l f$}\label{sec:lower.order.energy}
In this section, we use the energy estimates in Section~\ref{sec:eng-est-lin} and the commutator estimates in Section~\ref{sec:commutator} to prove weighted energy estimates for $f$ and $Y_lf$.
\begin{proposition}\label{prop:f.final}
The following holds for $f$ satisfying the assumptions of Theorem~\ref{thm:main}:
\begin{equation}
	\begin{split}
		&\: \sup_{t \in [t_0,\infty)} \left(\norm{w_{(0,K_{0}-1,1)} f}_{L^2_{x,v}}^2 + \norm{w_{(0,K_0,0)} f}_{L^2_{x,v}}^2 + t^{-1.1} \norm{w_{(1,K_{0},0)}f}^2_{L^2_{x,v}}\right)(t) \\
		&\: + \int_{t_0}^\infty \left( \norm{(I-\Pi) f}^2_{\calD_{x,v}(0,K_{0}-1,1)} + t^{-1.1} \norm{f}_{\calD_{x,v}(1,K_{0},0)}^2 \right)(t) \, \ud t \ls  E_{K_{0}}.
	\end{split}
\end{equation}
\end{proposition}
\begin{proof}
Since $f$ satisfies the linear equation with no source term, this estimate follows from integrating the equation in Proposition~\ref{prop:Gaussian-x-weights}. Note that the data terms are exactly the first three groups in \eqref{eq:EK.def}. \qedhere
\end{proof}

We derive the equation satisfied by $Y_l f$. Since $[T,Y_l]= 0$ for $T = \partial_t +v_i\partial_{x_i}$, the following holds:
\begin{lemma}\label{lem:Y.comm-Landau}
Let $f$ be a solution to \eqref{eq:linear.intro}. Then $Y_l f$ solves the following equation:
\begin{equation}\label{eq:comm-Landau}
\partial_t Y_l f+v_i\partial_{x_i} Y_l f+LY_l f= [L,Y_l]f.
\end{equation}
\end{lemma}

Next, we prove an energy estimate for $Y_l f$ using the equation \eqref{eq:comm-Landau}.
\begin{proposition}\label{prop:Yf.final}
Under the assumptions of Theorem~\ref{thm:main}, for any $l \in \{1,2,3\}$, the following holds:
\begin{equation}\label{eq:Yf.final}
	\begin{split}
		&\: \sup_{t \in [t_0,\infty)} \left(\norm{w_{(0,K_{0}-2,1)} Y_l f}_{L^2_{x,v}}^2 + \norm{w_{(0,K_{0}-1,0)} Y_l f}_{L^2_{x,v}}^2 + t^{-1.1} \norm{w_{(1,K_{0}-1,0)}Y_l f}^2_{L^2_{x,v}}\right)(t) \\
		&\: + \int_{t_0}^\infty \left( \norm{(I-\Pi)Y_l f}^2_{\calD_{x,v}(0,K_{0}-2,1)} + t^{-1.1} \norm{Y_l f}_{\calD_{x,v}(1,K_{0}-1,0)}^2 \right)(t) \, \ud t \ls E_{K_{0}}.
	\end{split}
\end{equation} 
\end{proposition}
\begin{proof}
By Proposition~\ref{prop:Gaussian-x-weights} (with $\kay = K_{0}-2$ and Lemma~\ref{lem:Y.comm-Landau}), 
\begin{equation}
    \begin{split}
        \hbox{LHS of \eqref{eq:Yf.final}} \ls &\: \int_{t_0}^\infty \Big( t^{-1.1} \Big| \jap{ w_{(1,K_{0}-1,0)}^2 Y_l f, [L,Y_l] f }_{L^2_{x,v}}\Big| + 
        \Big|
 \jap{ w_{(0,K_{0}-1,0)}^2 Y_l f, [L,Y_l] f }_{L^2_{x,v}} \Big| \\
 &\: \qquad \qquad +
\Big| \jap{ w_{(0,K_{0}-2,1)}^2 Y_l f, [L,Y_l] f }_{L^2_{x,v}}\Big|\Big) \ud t.
    \end{split}
\end{equation}

All these commutator terms are bounded using Proposition~\ref{prop:Y.L}. More precisely, for the first term we use \eqref{eq:main.Y.L} so that 
\begin{equation}
    \begin{split}
        &\: \int_{t_0}^\infty \Big( t^{-1.1} \Big| \jap{ w_{(1,K_{0}-1,0)}^2 Y_l f, [L,Y_l] f }_{L^2_{x,v}} \Big| \Big) \ud t \\
        \ls &\:  \int_{t_0}^\infty t^{-1.1} \| f \|_{\calD_{x,v}(1,K_{0},0)} \|  Y_l f \|_{\calD_{x,v}(1,K_{0}-1,0)}(t) \, \ud t \\
        &\: + \int_{t_0}^\infty t^{-3.1} \| w_{(1,K_{0},0)}f \|_{L^2_{x,v}} \| w_{(1,K_{0}-1,0)} Y_l f \|_{L^2_{x,v}}(t) \, \ud t.
    \end{split}
\end{equation}
For the other two terms, we use \eqref{eq:main.Y.L.I-Pi} so as to obtain
\begin{equation}
    \begin{split}
        &\: \int_{t_0}^\infty \Big|
        \jap{ w_{(0,K_{0}-1,0)}^2 Y_l f, [L,Y_l] f }_{L^2_{x,v}} \Big| \ud t \\
        \ls &\: \int_{t_0}^\infty \Big(  \| (I-\Pi) f \|_{\calD_{x,v}(0,K_{0},0)} + t^{-1} \| f \|_{\calD_{x,v}(0,K_{0}-1,0)} \Big) \| (I-\Pi) Y_l f \|_{\calD_{x,v}(0,K_{0}-1,0)}(t) \ud t \\
        &\: +  \int_{t_0}^\infty t^{-1} \| (I-\Pi) f \|_{\calD_{x,v}(0,K_{0}-1,0)} \| Y_l f \|_{\calD_{x,v}(1,K_{0}-1,0)}(t) \ud t \\
	&\: + \int_{t_0}^\infty t^{-2} \| w_{(0,K_{0},0)} f \|_{L^2_{x,v}} \| w_{(0,K_{0}-1,0)} Y_l f \|_{L^2_{x,v}}(t) \ud t.
    \end{split}
\end{equation}
and 
\begin{equation}
    \begin{split}
        &\: \int_{t_0}^\infty \Big| \jap{ w_{(0,K_{0}-2,1)}^2 Y_l f, [L,Y_l] f }_{L^2_{x,v}}\Big| \ud t \\
        \ls &\: \int_{t_0}^\infty \Big(  \| (I-\Pi) f \|_{\calD_{x,v}(0,K_{0}-1,1)} + t^{-1} \| f \|_{\calD_{x,v}(0,K_{0}-2,1)} \Big) \| (I-\Pi) Y_l f \|_{\calD_{x,v}(0,K_{0}-2,1)}(t) \ud t \\
        &\: +  \int_{t_0}^\infty t^{-1} \| (I-\Pi) f \|_{\calD_{x,v}(0,K_{0}-2,1)} \| Y_l f \|_{\calD_{x,v}(1,K_{0}-2,0)}(t) \ud t \\
	&\: + \int_{t_0}^\infty t^{-2} \| w_{(0,K_{0}-1,1)} f \|_{L^2_{x,v}} \| w_{(0,K_{0}-2,1)} Y_l f \|_{L^2_{x,v}}(t) \ud t.
    \end{split}
\end{equation}
Using Young's inequality, we absorb terms involving $Y_l f$ to the left-hand side so that we obtain
\begin{equation}
    \begin{split}
        &\: \hbox{LHS of \eqref{eq:Yf.final}} \\
        \ls &\: \int_{t_0}^\infty \Big( t^{-1.1} \| f \|_{\calD_{x,v}(1,K_{0},0)}^2 + \| (I-\Pi) f \|_{\calD_{x,v}(0,K_{0}-1,1)}^2 + \| (I-\Pi) f \|_{\calD_{x,v}(0,K_{0},0)}^2 \Big)(t) \, \ud t \\
        &\: + \int_{t_0}^\infty t^{-2} \Big( \| f \|_{\calD_{x,v}(0,K_{0}-1,0)}^2 + \| f \|_{\calD_{x,v}(0,K_{0}-2,1)}^2\Big)(t) \, \ud t \\
        &\: +\Big( \int_{t_0}^\infty \Big( t^{-2.55} \| w_{(1,K_{0},0)}f \|_{L^2_{x,v}} + t^{-2} \| w_{(0,K_{0},0)} f \|_{L^2_{x,v}} + t^{-2} \| w_{(0,K_{0}-1,1)} f \|_{L^2_{x,v}} \Big)(t) \ud t \Big)^2.
    \end{split}
\end{equation}
Finally, observe that the right-hand side can be bounded by $E_{K_{0}}$ using Proposition~\ref{prop:f.final}. \qedhere
\end{proof}

\section{Weighted $L$ energies }\label{sec:L.energy}

\subsection{Weighted $L$ energy estimates I}
In this subsection we control the energy $t \langle w^2_{(0,\kay,1)} f , L f \rangle_{L^2_{x,v}}$.

\begin{proposition}\label{prop:L-eng-1}
Under the assumptions of Theorem~\ref{thm:main}, the following holds:
\begin{equation}\label{eq:L-eng-1}
\begin{split}
\f{\ud}{\ud t}& \left(t \langle w^2_{(0,\kay,1)} f , L f \rangle_{L^2_{x,v}} \right) + t \norm{w_{(0,\kay,1)}Lf}_{L^2_{x,v}}^2\\
&\quad\lesssim \|(I-\Pi) f\|_{\calD_{x,v}(0,\kay+1,1)}^2 + t^{-2}\Big( \| f\|_{\calD_{x,v}(1,\kay+2,0)}^2 + \|Y f\|_{\calD_{x,v}(1,\kay+1,0)}^2\Big) \\
&\:\quad\qquad + t^{-3} \Big( \| w_{(1,\kay+1,0)} f \|_{L^2_{x,v}}^2 + \| w_{(1,\kay,0)} Y f \|_{L^2_{x,v}}^2\Big).
    \end{split}
\end{equation}
\end{proposition}
\begin{proof}
In the proof, we repeatedly use that
\begin{equation}\label{eq:L.properties}
	\hbox{$L\Pi = 0$, $L$ is self-adjoint, and that $w^2_{(0,\kay,1)}$ is $v$-independent.}
\end{equation}

Using \eqref{eq:L.properties} a few times, we start by computing
\begin{equation}\label{eq:L.energy.I.def}
\begin{split}
&\: \f{\ud}{\ud t} \left(t\langle w^2_{(0,\kay,1)}f , Lf \rangle_{L^2_{x,v}} \right) = \f{\ud}{\ud t} \left(t\langle w^2_{(0,\kay,1)}(I-\Pi)f , L(I-\Pi)f \rangle_{L^2_{x,v}} \right)\\
=&\: \jap{w^2_{(0,\kay,1)}(I-\Pi)f ,L(I-\Pi)f }_{L^2_{x,v}} +t\int_{\R^3}\int_{\R^3}T\{w^2_{(0,\kay,1)}(I-\Pi)f \cdot L(I-\Pi)f \}\ud v\ud x \\
= &\: \jap{ w^2_{(0,\kay,1)}(I-\Pi)f ,L(I-\Pi)f }_{L^2_{x,v}} +t \jap{ (Tw^2_{(0,\kay,1)}) (I-\Pi)f, L(I-\Pi)f }_{L^2_{x,v}} \\ 
&\: + t \jap{ w^2_{(0,\kay,1)} (T(I-\Pi)f),  L(I-\Pi)f }_{L^2_{x,v}} +t \jap{ w^2_{(0,\kay,1)} (I-\Pi)f, T( L(I-\Pi)f) }_{L^2_{x,v}} \\
= &\: \jap{ w^2_{(0,\kay,1)}(I-\Pi)f ,L(I-\Pi)f }_{L^2_{x,v}} +t \jap{ (Tw^2_{(0,\kay,1)}) (I-\Pi)f, L(I-\Pi)f }_{L^2_{x,v}} \\ 
&\: + 2 t \jap{ w^2_{(0,\kay,1)} (T(I-\Pi)f),  L(I-\Pi)f }_{L^2_{x,v}} +t \jap{w^2_{(0,\kay,1)} (I-\Pi)f, [T, L]((I-\Pi)f) }_{L^2_{x,v}} \\
=: &\: I + II + III + IV.
\end{split}
\end{equation}
(Recall $T=\rd_t+v_i\rd_{x_i}.$)

For the term $I$, we break $L$ into $A$ and $K$ so that by \eqref{eq:K.upper.bound.unweighted} and \eqref{eq:A.upper.bound.2}, we have 
\begin{equation}\label{eq:L.I.energy.main.I}
|I| \ls  \| (I-\Pi) f \|_{\calD_{x,v}(0,\kay,1)}^2.
\end{equation}

We turn to the term $II$ in \eqref{eq:L.energy.I.def}. By \eqref{eq:Tw.exact.1}--\eqref{eq:Tw.exact.4} and Lemma~\ref{lem:z.x-tv.com},
\begin{align}
    |Tw^2_{(0,\kay,1)}| \ls &\: \Big(t^{-1.1} + t^{-1} |x-tv| \f{|x|}{t^2} + t^{-1}(\f{|x|}{t^2})^2  \Big) w^2_{(0,\kay,1)} \ls \Big(t^{-1} \jap{\f{x}{t^2}}^2 + t^{-2} \jap{z} \f{|x|}{t}\Big) w^2_{(0,\kay,1)} \label{eq:Tw} \\
    |\rd_v T w^2_{(0,\kay,1)}| \ls &\: (|x|/t^2) w_{(0,\kay,1)}^2.\label{eq:dv.Tw}
\end{align}
We write $L = -A-K$ and use \eqref{eq:K.upper.bound.unweighted} and \eqref{eq:A.upper.bound.2} to bound $II.$ For $K$, we use \eqref{eq:K.upper.bound.unweighted} and \eqref{eq:Tw} to obtain
\begin{equation}\label{eq:L.I.II.K}
    \begin{split}
        &\: \Big| t \jap{ (Tw^2_{(0,\kay,1)}) (I-\Pi)f, K(I-\Pi)f }_{L^2_{x,v}} \Big| \\
        \ls &\: \Big\langle   e^{-\f{\alp \bt |x|^2}{\alp + \bt t^2}}  \jap{x/t^2}^2 w_{(0,\kay,1)}^2 \norm{\jap{z}^{-\f 12} (I-\Pi)f}_{L^2_v}, \norm{\jap{z}^{-\f 12} (I-\Pi)f}_{L^2_v} \Big\rangle_{L^2_x} \\
        \ls &\: \| (I-\Pi) f \|_{\calD_{x,v}(0,\kay+1,1)}^2,
    \end{split}
\end{equation}
where we used $\jap{x/t^2} \leq \jap{x/t}$ and redistributed the $v$-independent weights.

For $A$, we first use \eqref{eq:A.upper.bound.2} and then use \eqref{eq:Tw}, \eqref{eq:dv.Tw} to obtain
\begin{equation*}
\begin{split}
&\: \Big| t \jap{ (Tw^2_{(0,\kay,1)}) (I-\Pi)f, A(I-\Pi)f }_{L^2_{v}} \Big| \\
\ls &\:  e^{-\f{\alp \bt |x|^2}{\alp + \bt t^2}} t \Big( \Big| \jap{ \jap{z}^{-1}  (I-\Pi)f, (Tw_{(0,\kay,1)}^2) (I-\Pi)f}_{L^2_v} \Big| \\
&\: \qquad \qquad +  t^{-2} \Big| \jap{ \jap{z}^{-3} P_z \rd_v (I-\Pi)f, P_z \rd_v ((Tw_{(0,\kay,1)}^2 (I-\Pi)f)}_{L^2_v} \Big| \\
&\: \qquad \qquad + t^{-2} \Big| \jap{ \jap{z}^{-1} (I - P_z) \rd_v (I-\Pi)f, (I - P_z) \rd_v (Tw_{(0,\kay,1)}^2 (I-\Pi)f)}_{L^2_v} \Big| \Big) \\
\ls &\:  \jap{x/t^2}^2 \| w_{(0,\kay,1)} (I-\Pi) f \|_{\Delta_v(0)}^2  + t^{-1}  \| w_{(0,\kay,1)} (I-\Pi) f \|_{\Delta_v(0)} \| \jap{z} w_{(0,\kay+1,1)} (I-\Pi) f \|_{\Delta_v(0)}. 
\end{split}
\end{equation*}
Taking $L^2_x$ and using Corollary~\ref{cor:Pi.in.weighted.space} and Lemma~\ref{lem:weight.compare}, we obtain
\begin{equation}\label{eq:L.I.II.A}
\begin{split}
&\: \Big| t \jap{ (Tw^2_{(0,\kay,1)}) (I-\Pi)f, A(I-\Pi)f }_{L^2_{x,v}} \Big| \\
\ls &\: \| (I-\Pi) f \|_{\calD_{x,v}(0,\kay+1,1)}^2 + t^{-1} \| (I-\Pi) f \|_{\calD_{x,v}(0,\kay+1,1)} \| \jap{z} (I-\Pi) f \|_{\calD_{x,v}(0,\kay+1,1)}\\
\ls &\: \| (I-\Pi) f \|_{\calD_{x,v}(0,\kay+1,1)}^2 + t^{-1} \| (I-\Pi) f \|_{\calD_{x,v}(0,\kay+1,1)} \| f \|_{\calD_{x,v}(1,\kay+2,0)}.
\end{split}
\end{equation}

Altogether, combining \eqref{eq:L.I.II.K} and \eqref{eq:L.I.II.A}, and using Cauchy--Schwarz and Lemma~\ref{lem:weight.compare}, we obtain
\begin{equation}\label{eq:L.I.energy.main.II}
\begin{split}
|II| \ls &\: \| (I-\Pi) f \|_{\calD_{x,v}(0,\kay+1,1)}^2 + t^{-1} \| (I-\Pi) f \|_{\calD_{x,v}(0,\kay+1,1)} \| f \|_{\calD_{x,v}(1,\kay+2,0)}.
\end{split}
\end{equation}

For $III$ in \eqref{eq:L.energy.I.def}, we use \eqref{eq:L.properties} to obtain
\begin{equation}\label{eq:L.energy.I.term.2}
\begin{split}
III = &\:  2 t \langle w^2_{(0,\kay,1)}T f, Lf \rangle_{L^2_{x,v}} - 2 t  \langle w^2_{(0,\kay,1)} T \Pi f, Lf \rangle  \\
= &\:  - 2 t \| w_{(0,\kay,1)}Lf \|_{L^2_{x,v}}^2 - 2 t \langle w^2_{(0,\kay,1)}(I-\Pi) T \Pi f, Lf \rangle_{L^2_{x,v}},
\end{split}
\end{equation}
where we used the equation \eqref{eq:linear.intro} in the last line.

To handle the last term on the right-hand side of \eqref{eq:L.energy.I.term.2}, we use the Cauchy--Schwarz inequality and use \eqref{eq:I-Pi.T.Pi} with $g = f$, i.e.,
\begin{equation}
\begin{split}
&\: \Big|  t \jap{ w^2_{(0,\kay,1)}(I-\Pi) T \Pi f, Lf }_{L^2_{x,v}}\Big| \ls t \| w_{(0,\kay,1)} (I-\Pi) T \Pi f \|_{L^2_{x,v}} \| w_{(0,\kay,1)} Lf \|_{L^2_{x,v}}\\
\ls &\: t^{-1}  \Big( \| w_{(0,\kay+1,1)}  f \|_{L^2_{x,v}} + \| w_{(0,\kay,1)}  Y f \|_{L^2_{x,v}} \Big) \| w_{(0,\kay,1)} Lf \|_{L^2_{x,v}}.
 \end{split}
 \end{equation}
Using Young's inequality, we absorb $t^{\f 12} \| w_{(0,\kay,1)} Lf \|_{L^2_{x,v}}$ by the good term in \eqref{eq:L.energy.I.term.2} and so after also using Lemma~\ref{lem:weight.compare}, we have
\begin{equation}\label{eq:L.I.energy.main.III}
\begin{split}
 III \leq - t \| w_{(0,\kay,1)}Lf \|_{L^2_{x,v}}^2 + C t^{-3} \Big( \| w_{(1,\kay+1,0)} f \|_{L^2_{x,v}}^2 + \| w_{(1,\kay,0)} Y f \|_{L^2_{x,v}}^2 \Big).
\end{split}
\end{equation}

Finally, the term $IV$ in \eqref{eq:L.energy.I.def} is a commutator term for which we directly estimate using Proposition~\ref{prop:comm-L-T} with $g_1 = g_2 = (I-\Pi) f$ so that 
\begin{equation}\label{eq:L.II.III.initial}
\begin{split}
	|IV| \ls t^{-1} \|(I-\Pi) f\|_{\calD_{x,v}(0,\kay,1)} \Big( \|\jap{z}(I-\Pi) f\|_{\calD_{x,v}(0,\kay+1,1)} + \|Y (I-\Pi) f\|_{\calD_{x,v}(0,\kay+1,1)}\Big).
\end{split}
\end{equation}
By \eqref{eq:Pi.in.weighted.space.2} and Lemma~\ref{lem:weight.compare}, we have 
\begin{equation}\label{eq:z.I-Pi.f.0.k+1.1}
    \|\jap{z}(I-\Pi) f\|_{\calD_{x,v}(0,\kay+1,1)} \ls \| f\|_{\calD_{x,v}(1,\kay+2,0)}.
\end{equation}
Moreover, by Proposition~\ref{prop:comm-Y-Pi} and \eqref{eq:dv.[Y,Pi].1} in Lemma~\ref{lem:Pi.schematic}, $\| [Y,\Pi] f\|_{\calD_{x,v}(0,\kay+1,1)} \ls t^{-1} \| f\|_{\calD_{x,v}(0,\kay+1,1)}$ so that 
\begin{equation}\label{eq:L.II.III.initial.2}
    \| Y (I-\Pi) f\|_{\calD_{x,v}(0,\kay+1,1)} \ls \|  Yf\|_{\calD_{x,v}(0,\kay+1,1)} + t^{-1}\| f \|_{\calD_{x,v}(0,\kay+1,1)}.
\end{equation}
Thus, returning to \eqref{eq:L.II.III.initial}, and using Lemma~\ref{lem:weight.compare}, we obtain
\begin{equation}\label{eq:L.I.energy.main.IV}
\begin{split}
	|IV| \ls t^{-1} \|(I-\Pi) f\|_{\calD_{x,v}(0,\kay,1)} \Big( \| f\|_{\calD_{x,v}(1,\kay+2,0)} + \|Y f\|_{\calD_{x,v}(1,\kay+1,0)} \Big).
\end{split}
\end{equation}

Combining the above, moving the good term in $III$ to the left-hand side, we thus obtain the desired estimate \eqref{eq:L-eng-1}. \qedhere

\end{proof}

\begin{corollary}\label{cor:fLf.final}
Under the assumptions of Theorem~\ref{thm:main}, the following estimate holds:
\begin{equation}
\begin{split}
	&\: \sup_{t \in [t_0,\infty)} \Big(t \langle w^2_{(0,K_{0}-2,1)} f , L f \rangle_{L^2_{x,v}}(t) \Big) + \int_{t_0}^\infty t \norm{w_{(0,K_{0}-2,1)}Lf}_{L^2_{x,v}}^2(t) \ud t \ls E_{K_{0}}.
\end{split}
\end{equation}
\end{corollary}
\begin{proof}
    We integrate the inequality in Proposition~\ref{prop:L-eng-1} in $t$ for $\kay = K_{0}-2$. We thus need to bound the time integral of the right-hand side of \eqref{eq:L-eng-1}. The time integrals of the terms $\|(I-\Pi) f\|_{\calD_{x,v}(0,\kay+1,1)}^2$ and $t^{-2}\Big( \| f\|_{\calD_{x,v}(1,\kay+2,0)}^2 + \|Y f\|_{\calD_{x,v}(1,\kay+1,0)}^2\Big)$ are already controlled in
    Proposition~\ref{prop:f.final} and Proposition~\ref{prop:Yf.final}. For the remaining terms, we control them by noting that $\| w_{(1,\kay+1,0)} f \|_{L^2_{x,v}}^2$ and $\| w_{(1,\kay,0)} Y f \|_{L^2_{x,v}}^2$ grow no faster than $t^{1.1}$ (by Proposition~\ref{prop:f.final} and Proposition~\ref{prop:Yf.final}) and thus together with the factor $t^{-3}$, the term is integrable. \qedhere
\end{proof}

\subsection{Weighted L energy estimates II}
In this subsection, we consider the energy $t^2 \| w_{(0,\kay,1)} Lf \|_{L^2_{x,v}}^2$.

\begin{proposition}\label{prop:L-eng-2}
Under the assumptions of Theorem~\ref{thm:main}, the following holds:
\begin{equation*}
\begin{split}
\f{\ud}{\ud t}& \left(t^2 \| w_{(0,\kay,1)} Lf \|_{L^2_{x,v}}^2 \right) + t^2 \norm{Lf}^2_{\calD_{x,v}(0,\kay,1)}\\
&\quad\lesssim \Big( t \| w_{(0,\kay,1)} Lf \|_{L^2_{x,v}}^2 + \| (I - \Pi) f\|_{\calD_{x,v}(0,\kay+2,1)}^2 \Big) \\
        &\: \qquad + t^{-2} \Big( \| f\|_{\calD_{x,v}(1,\kay+2,0)}^2 + \| Yf\|_{\calD_{x,v}(1,\kay+1,0)}^2 + \| w_{(0,\kay+1,0)} f\|_{L^2_{x,v}}^2 + \| w_{(0,\kay,0)}  Yf\|_{L^2_{x,v}}^2 \Big).
\end{split}
\end{equation*}
\end{proposition}
\begin{proof}
We use \eqref{eq:L.properties} throughout. We will also use $L = (I-\Pi)L$.

As in the proof of Proposition~\ref{prop:L-eng-1}, we compute using \eqref{eq:L.properties} to obtain
\begin{equation}\label{eq:L.II.main}
    \begin{split}
        &\: \f{\ud}{\ud t} \left(t^2 \| w_{(0,\kay,1)} Lf \|_{L^2_{x,v}}^2 \right) = \f{\ud}{\ud t} \left(t^2 \| w_{(0,\kay,1)} (I-\Pi)Lf \|_{L^2_{x,v}}^2 \right)\\
=&\: 2t \norm{w_{(0,\kay,1)} Lf}_{L^2_{x,v}}^2  + t^2\int_{\R^3}\int_{\R^3}T\{w^2_{(0,\kay,1)}((I-\Pi)L f)^2 \}\ud v\ud x \\
= &\: 2t \norm{w_{(0,\kay,1)} Lf}_{L^2_{x,v}}^2 + t^2 \jap{(Tw^2_{(0,\kay,1)}) Lf, Lf}_{L^2_{x,v}} + 2 t^2 \jap{w^2_{(0,\kay,1)}  T L(I-\Pi)f,  Lf}_{L^2_{x,v}}  \\
= &\: \underbrace{2t \norm{w_{(0,\kay,1)} Lf}_{L^2_{x,v}}^2 }_{=:I}\underbrace{-  2 t^2 \jap{w^2_{(0,\kay,1)}  L^2 f, Lf}_{L^2_{x,v}}}_{=:II} + \underbrace{t^2 \jap{(Tw^2_{(0,\kay,1)}) Lf, L(I-\Pi) f}_{L^2_{x,v}} }_{=:III}  \\
&\: + \underbrace{2 t^2 \jap{w^2_{(0,\kay,1)}  [T,L] (I-\Pi) f, Lf}_{L^2_{x,v}}}_{=:IV}   \underbrace{- 2 t^2 \jap{w^2_{(0,\kay,1)}  L T \Pi f, Lf}_{L^2_{x,v}} }_{=:V}.        
    \end{split}
\end{equation}
In particular, we have used the equation $Tf = -Lf$. The term $I$ can be bounded by 
\begin{equation}\label{eq:L.II.I}
    |I|\ls t \| w_{(0,\kay,1)} Lf \|_{L^2_{x,v}}^2.
\end{equation}

Using Lemma~\ref{lem:lower-bound-L}, the term $II$ gives a good term, i.e., there exists $c>0$ such that
\begin{equation}\label{eq:L.II.II}
    II \leq -2 c t^2\| (I-\Pi) Lf \|_{\calD_{x,v}(0,\kay,1)}^2 = - 2c t^2\| Lf \|_{\calD_{x,v}(0,\kay,1)}^2,
\end{equation}
where we used $(I-\Pi)L = L$.

For $III$, we write the second $L$ as $L=-A-K$ and apply \eqref{eq:K.upper.bound.unweighted} and \eqref{eq:A.upper.bound.2} with $g_1 = (I-\Pi)f$, $g_2 = (T w^2_{(0,\kay,1)} )(I-\Pi) Lf $. We handle the weight $Tw^2_{(0,\kay,1)}$ using \eqref{eq:Tw} in a similar manner as \eqref{eq:L.I.II.K}, \eqref{eq:L.I.II.A} but distributing the $\jap{x/t}$ weights slightly differently so that we obtain
\begin{equation}\label{eq:L.II.III.K}
    \begin{split}
        &\: \Big| t^2 \jap{ (Tw^2_{(0,\kay,1)}) Lf , K(I-\Pi)f }_{L^2_{x,v}} \Big| \\
        \ls &\: t \|(I-\Pi) f\|_{\calD_{x,v}(0,\kay+2,1)} \|Lf\|_{\calD_{x,v}(0,\kay,1)},
    \end{split}
\end{equation}
\begin{equation}\label{eq:L.II.III.A}
\begin{split}
&\: \Big| t^2 \jap{ (Tw^2_{(0,\kay,1)}) Lf, A(I-\Pi)f }_{L^2_{v}} \Big| \\
\ls &\: t \|(I-\Pi) f\|_{\calD_{x,v}(0,\kay+2,1)} \|Lf\|_{\calD_{x,v}(0,\kay,1)} + \| f\|_{\calD_{x,v}(1,\kay+2,0)} \|Lf\|_{\calD_{x,v}(0,\kay,1)}.
\end{split}
\end{equation}

Thus, combining \eqref{eq:L.II.III.K} and \eqref{eq:L.II.III.A}, we obtain
\begin{equation}\label{eq:L.II.III}
\begin{split}
|III| \ls &\: t \|(I-\Pi) f\|_{\calD_{x,v}(0,\kay+2,1)} \|Lf\|_{\calD_{x,v}(0,\kay,1)} + \| f\|_{\calD_{x,v}(1,\kay+2,0)} \|Lf\|_{\calD_{x,v}(0,\kay,1)}.
\end{split}
\end{equation}

For $IV$, we use the commutator estimate in Proposition~\ref{prop:comm-L-T} with $g_2 = (I-\Pi) f$ and $g_1 = Lf$ so that
\begin{equation}
	|IV| \ls \| Lf \|_{\calD_{x,v}(0,\kay,1)} \Big( \| \jap{z} (I-\Pi) f \|_{\calD_{x,v}(0,\kay+1,1)} + \| Y(I-\Pi) f \|_{\calD_{x,v}(0,\kay+1,1)} \Big).
\end{equation}
Using \eqref{eq:z.I-Pi.f.0.k+1.1} and also \eqref{eq:L.II.III.initial.2}, we have
\begin{equation}\label{eq:L.II.IV}
	|IV| \ls \| Lf \|_{\calD_{x,v}(0,\kay,1)} \Big( \| f \|_{\calD_{x,v}(1,\kay+2,0)} + \| Y f \|_{\calD_{x,v}(1,\kay+1,0)} \Big).
\end{equation}


Turning to $V$, we recall that $T\Pi f$ is given by \eqref{eq:L.energy.I.term.2.1}. Since $\sqrt{\mu},\, z_i \sqrt{\mu},\, |z|^2 \sqrt{\mu} \in \mathrm{ker}(L)$, we have
\begin{equation}
\begin{split}
LT \Pi f 
=&\: -\f{1}{t\sqrt{\alpha+\beta t^2}} \Big\{ Y_i b_l^f \cdot L (z_l z_i \sqrt{\mu}) + Y_i c^f\cdot L (|z|^2 z_i \sqrt{\mu}) \Big\} .
\end{split}
\end{equation}
An explicit computation (using in particular \eqref{eq:K.representation} and Proposition~\ref{prop:aux-integral}) shows that 
$$|L (z_l z_i \sqrt{\mu})|,\, |L (|z|^2 z_i \sqrt{\mu})| \ls e^{-\f{\alp \bt |x|^2}{\alp + \bt t^2}} \jap{z}^4 \sqrt{\mu}.$$
Thus, using that $\| \jap{z}^{9/2} \sqrt{\mu} \|_{L^2_v} \ls e^{-\f{\alp \bt |x|^2}{2(\alp +\bt t^2)}} t^{-\f 32}$ and applying Corollary~\ref{cor:abc}, we have
\begin{equation}
\begin{split}
	\| e^{\f{\alp \bt |x|^2}{2(\alp + \bt t^2)}} w_{(0,\kay,1)} \jap{z}^{\f 12} LT\Pi f\|_{L^2_{x,v}} \ls &\: t^{-2} \times t^{-\f 32} \| e^{\f{\alp \bt |x|^2}{2(\alp + \bt t^2)}} w_{(0,\kay,1)} e^{-\f{3\alp \bt |x|^2}{2(\alp +\bt t^2)}} (|Y b^f| + |Y c^f|)  \|_{L^2_x} \\
	\ls &\: t^{-2} \Big( \| w_{(0,\kay+1,0)} f\|_{L^2_{x,v}} + \| w_{(0,\kay,0)} Yf\|_{L^2_{x,v}}\Big),
\end{split}
\end{equation}
where we used $w_{(0,\kay,1)} \jap{x/t} e^{-\f{\alp \bt |x|^2}{2(\alp +\bt t^2)}} \ls w_{(0,\kay+1,0)} e^{-\f{\alp^2 (\bt - \bar{\bt}) |x|^2}{2(\alp+ \bar{\bt} t^2)(\alp +\bt t^2)}} \ls w_{(0,\kay+1,0)}$. 
As a consequence, using also Lemma~\ref{lem:weight.compare}, we obtain
\begin{equation}\label{eq:L.II.V}
\begin{split}
|V| \ls &\: t^2 \| e^{\f{\alp \bt |x|^2}{2(\alp + \bt t^2)}} w_{(0,\kay,1)} \jap{z}^{\f 12} LT\Pi f\|_{L^2_{x,v}} \| e^{-\f{\alp \bt |x|^2}{2(\alp + \bt t^2)}} w_{(0,\kay,1)} \jap{z}^{-\f 12} Lf \|_{L^2_{x,v}} \\
\ls &\:  t^{-1} \Big( \| w_{(0,\kay+1,0)}  f\|_{L^2_{x,v}} + \| w_{(0,\kay,0)} Yf\|_{L^2_{x,v}}\Big) \Big( t \| Lf \|_{\calD_{x,v}(0,\kay,1)}\Big).
\end{split}
\end{equation}

We now combine all the estimates. For the estimates for $III$, $IV$ and $V$ in \eqref{eq:L.II.III}, \eqref{eq:L.II.IV} and \eqref{eq:L.II.V}, we use Young's inequality so that $t \| Lf \|_{\calD_{x,v}(0,\kay,1)}$ can be absorbed by the good term in $II$. Thus, altogether 
\begin{equation*}
    \begin{split}
        &\: \hbox{RHS of \eqref{eq:L.II.main}} \\
        \leq &\: - c t^2 \| Lf \|_{\calD_{x,v}(0,\kay,1)}^2 + C\Big( t \| w_{(0,\kay,1)} Lf \|_{L^2_{x,v}}^2 + \| (I - \Pi) f\|_{\calD_{x,v}(0,\kay+2,1)}^2 \Big) \\
        &\: + C t^{-2} \Big( \| f\|_{\calD_{x,v}(1,\kay+2,0)}^2 + \| Yf\|_{\calD_{x,v}(1,\kay+1,0)}^2 + \| w_{(0,\kay+1,0)} f\|_{L^2_{x,v}}^2 + \| w_{(0,\kay,0)}  Yf\|_{L^2_{x,v}}^2 \Big).
    \end{split}
\end{equation*}
Moving the good term $c t^2 \| Lf \|_{\calD_{x,v}(0,\kay,1)}^2$ to the left-hand side yields the desired conclusion. \qedhere
\end{proof}

\begin{corollary}\label{cor:Lf.final}
Under the assumptions of Theorem~\ref{thm:main}, the following estimate holds:
    $$\sup_{t\in [t_0,\infty)} t^2 \| w_{(0,K_{0}-3,1)} Lf \|_{L^2_{x,v}}^2(t)  + \int_{t_0}^\infty t^2 \norm{Lf}^2_{\calD_{x,v}(0,K_{0}-3,1)}(t)\, \ud t \ls E_{K_{0}}.$$
\end{corollary}
\begin{proof}
    The proof is similar to Corollary~\ref{cor:fLf.final} and so we will be brief. We start with the estimate in Proposition~\ref{prop:L-eng-2} with $\kay = K_{0} -3$ and notice that the time-integral of the right-hand side is bounded by $E_{K_{0}}$ using Proposition~\ref{prop:f.final}, Proposition~\ref{prop:Yf.final} and Corollary~\ref{cor:fLf.final}. \qedhere
\end{proof}

\section{Limit at $t = \infty$}\label{sec:limit}

In this section, we prove parts \eqref{item:part.4} and \eqref{item:part.5} of Theorem~\ref{thm:main}. We first make some simple observations.
\begin{lemma}
There exists $c  >0$ such that the following holds:
    \begin{align}
        \jap{x/t}^{2\kay}\ls &\: w_{(0,\kay,1)}^2 e^{-\f{\alp \bt |x|^2}{\alp + \bt t^2}}   \quad \hbox{if $|x|\leq t^2$}, \label{eq:final.weights.1}\\
        \jap{x/t}^{2\kay} \ls &\: e^{-ct^2} w_{(0,0,1)}^2\quad \hbox{if $|x|\geq t^2$}, \label{eq:final.weights.2}
    \end{align}
    where the implicit constant may depend on $\kay$.
\end{lemma}
\begin{proof}
    We compute
    \begin{equation}\label{eq:xt.exp.combine}
        w_{(0,0,1)}^2 e^{-\f{\alp \bt |x|^2}{\alp + \bt t^2}} = e^{Bt^{-0.1}} e^{-\f{\alp^2 (\bt - \bar{\bt}) |x|^2}{(\alp + \bt t^2)(\alp + \bar{\bt} t^2)}},
    \end{equation}
    from which \eqref{eq:final.weights.1} follows. For \eqref{eq:final.weights.2}, we observe that there exists $c>0$ such that 
    $$\jap{x/t}^{2\kay} w_{(0,0,1)}^{-2} \ls \jap{x/t}^{2\kay} e^{-2 c|x|^2/t^2} \ls \sup_{y \geq t} \jap{y}^{2\kay} e^{-2c|y|^2} \ls e^{-ct^2} \quad \hbox{if $|x| \geq t^2$},$$
    which gives \eqref{eq:final.weights.2}. \qedhere
\end{proof}

We now conclude part \eqref{item:part.4} of Theorem~\ref{thm:main}, i.e., that a limit at $t = \infty$ exists.
\begin{theorem}\label{thm:limit.def}
Under the assumptions of Theorem~\ref{thm:main}, there exists a well-defined limit $f^\sharp_\infty(x,v)$ such that $\lim_{t\to \infty} f(t,x+tv,v) = f^\sharp_\infty(x,v)$ in the norm $\|\jap{x}^{-\f 12} \cdot \|_{L^2_{x,v}}$. Moreover, the convergence holds with the following quantitative rate:
\begin{equation}\label{eq:f.convergence}
\| \brk{x}^{-\f 12} ( f^\sharp(t,x,v) -  f^\sharp_\infty(x,v)) \|_{L^2_{x,v}} \ls t^{-\f 12} E_{3}^{\f 12}.
\end{equation}
\end{theorem}
\begin{proof}
Denote $f^\sharp(t,x,v) = f(t,x+tv,v)$. We first show that for any sequence $t_n \to \infty$, $f^\sharp(t_n,x,v)$ is Cauchy in $\|\brk{x}^{-\f 12} \cdot \|_{L^2_{x,v}}$. To see this, notice that 
$$\rd_t [\jap{x}^{-\f 12} f^\sharp(t,x,v)] = \jap{x}^{-\f 12} (\rd_t f  + v\cdot \nabla_x f)(t,x+tv,v) = - \jap{x}^{-\f 12} Lf(t,x+tv,v).$$ Hence, for $1\leq n < m$, we have
$$\jap{x}^{-\f 12} f^\sharp(t_m,x,v) - \jap{x}^{-\f 12} f^\sharp(t_n,x,v) = - \int_{t_n}^{t_m} \jap{x}^{-\f 12} Lf(t,x+tv,v)\, \ud t.$$
Therefore, by Minkowski's inequality,
\begin{equation}\label{eq:f.limit.Cauchy}
\begin{split}
&\: \Big\| \jap{x}^{-\f 12} \Big( f^\sharp(t_m,x,v) - f^\sharp(t_n,x,v)  \Big) \Big\|_{L^2_{x,v}} \\
\leq &\: \int_{t_n}^{t_m} \| \jap{x}^{-\f 12}Lf(t,x+tv,v)\|_{L^2_{x,v}} \, \ud t \\
\leq &\: \Big( \int_{t_n}^{t_m} \int_{\RR^3\times \RR^3}  t^2\jap{x}^{-1} |Lf|^2(t,x+tv,v) \, \ud v \ud x \ud t \Big)^{1/2} t_n^{-1/2}\\
\leq &\: \Big( \int_{t_n}^{t_m} \int_{\RR^3\times \RR^3}  t^2 \jap{x-tv}^{-1} |Lf|^2(t,x,v) \, \ud v \ud x \ud t \Big)^{1/2} t_n^{-1/2}.
\end{split}
\end{equation}

We split the integral into $|x|\leq t^2$ and $|x| \geq t^2$. For $|x|\leq t^2$, we have $\jap{x-tv}^{-1} \ls \jap{z}^{-1}$ by \eqref{eq:z.x-tv.com.restricted}. Thus, using \eqref{eq:final.weights.1} and Corollary~\ref{cor:Lf.final}, we obtain
\begin{equation}\label{eq:f.limit.Cauchy.1}
\begin{split}
&\: \int_{t_n}^{t_m} \int_{\{|x|\leq t^2\}} \int_{\RR^3}   t^2 \jap{x-tv}^{-1} |Lf|^2(t,x,v) \, \ud v \ud x \ud t \\
\ls &\: \int_{t_n}^{t_m} \int_{\RR^3} \int_{\RR^3}   t^2 \jap{z}^{-1} e^{-\f{\alp \bt |x|^2}{\alp + \bt t^2}} w_{(0,0,1)}^2 |Lf|^2(t,x,v) \, \ud v \ud x \ud t  \\
\ls &\: \int_{t_n}^{t_m} t^2 \| Lf \|_{\calD_{x,v}(0,0,1)}^2 \ud t \ls E_3.
\end{split}
\end{equation}

For $|x| \geq t^2$, we use \eqref{eq:final.weights.2} and the ``sup'' part of Corollary~\ref{cor:Lf.final} to obtain
\begin{equation}\label{eq:f.limit.Cauchy.2}
\begin{split}
&\: \int_{t_n}^{t_m} \int_{\{|x|\geq t^2\}} \int_{\RR^3}   t^2  \jap{x-tv}^{-1} |Lf|^2(t,x,v) \, \ud v \ud x \ud t \\
\ls &\: \int_{t_n}^{t_m} e^{-ct^2} \Big( \int_{\RR^3} \int_{\RR^3} t^2 w_{(0,0,1)}^2 |Lf|^2(t,x,v) \, \ud v \ud x \Big) \ud t  \ls E_3.
\end{split}
\end{equation}

It follows from \eqref{eq:f.limit.Cauchy}--\eqref{eq:f.limit.Cauchy.2} that there exists $f^\sharp_\i:\bbR^3_x\times \bbR^3_v \to \bbR$ such that $\| \brk{x}^{-\f 12} ( f^\sharp(t_n,x,v) -  f^\sharp_\infty(x,v)) \|_{L^2_{x,v}} \to 0$ along any sequence $\{t_n\}_{n=1}^\infty$, which in turn implies
$$\lim_{t\to \infty} \| \brk{x}^{-\f 12} ( f^\sharp(t,x,v) -  f^\sharp_\infty(x,v)) \|_{L^2_{x,v}} =0.$$

Finally, fixing $t_n=t$, taking $t_m\to \infty$ in \eqref{eq:f.limit.Cauchy}, and using the above estimates, we obtain \eqref{eq:f.convergence}. \qedhere
\end{proof}

We now conclude part \eqref{item:part.5} of Theorem~\ref{thm:main}, i.e., that the limit at $t = \infty$ is in general non-trivial.
\begin{theorem}\label{thm:nonvanishing}
Fix $\rho_a, \,(\rho_b)_i,\, \rho_c \in C_c^\infty(\bbR^3)$ so that not all of them are zero. There exists $t_0>0$ sufficiently large (depending on these functions) such that the following holds.

Consider the initial value problem for data on $\{t = t_0\}$ such that 
\begin{equation}
f(t_0,x,v) = a(t_0,x) \sqrt{\mu} + b_i(t_0,x) z_i \sqrt{\mu} + c(t_0,x) |z|^2 \sqrt{\mu}
\end{equation}
with
$$a(t_0,x) = \rho_a(\f x{t_0}),\quad b_i(t_0,x) = (\rho_b)_i(\f x{t_0}),\quad c(t_0,x) = \rho_c(\f x{t_0}).$$

Then the limit $f^\sharp_\infty$, whose existence is guaranteed by Theorem~\ref{thm:limit.def}, is non-vanishing.
\end{theorem}
\begin{proof}
Note that the initial data are chosen so that for any $K_{0} \in \mathbb Z_{\geq 3}$,
\begin{equation}\label{eq:lower.bound.thm.data.check}
    E_{K_{0}} \ls 1,\quad \| \jap{z}^{-\f 12} f \|_{L^2_{x,v}}(t_0) \gtrsim 1    
\end{equation}
independently of $t_0$, with implicit constants depending on the choice of $\rho_a, \,(\rho_b)_i,\, \rho_c \in C_c^\infty(\bbR^3)$. Notice that the definition of $E_{K_{0}}$ contains some growing $t_0$ weights, and in order to show that \eqref{eq:lower.bound.thm.data.check} holds independently of $t_0 \geq 1$, we used that $Lf|_{\{t=t_0\}}=0$, which in turn follows from $(I-\Pi)f|_{\{t=t_0\}}=0$.

Therefore, all the estimates proven above apply, and in particular, \eqref{eq:f.convergence} holds. Hence, there exist $c_1,c_2>0$ independent of $t_0$ such that
\begin{equation}
\begin{split}
\| \jap{x}^{-\f 12} f^\sharp_\infty \|_{L^2_{x,v}} \geq &\: \| \jap{x}^{-\f 12} f^\sharp \|_{L^2_{x,v}}(t_0)  - \| \jap{x}^{-\f 12} (f^\sharp(t_0,x,v) - f^\sharp_\infty(x,v)) \|_{L^2_{x,v}} \\
\geq &\: c_1 - c_2 t_0^{-\f 12}.
\end{split}
\end{equation}
Here, we used $$\| \jap{x}^{-\f 12} f^\sharp \|_{L^2_{x,v}}^2(t_0) = \int_{\bbR^3} \int_{\bbR^3} \jap{x-tv}^{-1} f^2(t_0,x,v)\,\ud v\ud x\gtrsim \| \jap{z}^{-\f 12} f \|_{L^2_{x,v}}^2(t_0),$$ which is true by Lemma~\ref{lem:z.x-tv.com} and the compact support of the data in $|x|\ls t_0^2$. 

Thus, taking $t_0$ sufficiently large, we conclude that $\| \jap{x}^{-\f 12} f^\sharp_\infty \|_{L^2_{x,v}}  \neq 0$, as desired. \qedhere
\end{proof}

\section{Decay estimates for the microscopic part}\label{sec:decay}

\subsection{First decay estimates}

We now prove decay estimates corresponding to \eqref{eq:decay.main.thm}.
\begin{theorem}\label{thm:microscopic}
    Under the assumptions of Theorem~\ref{thm:main},
    \begin{equation}
        \| \jap{z}^{-\f 12} \jap{x/t}^{K_{0}-2} (I-\Pi) f \|_{L^2_{x,v}} \ls t^{-\f 12} E_{K_{0}}^{\f 12}.
    \end{equation}
\end{theorem}
\begin{proof}
Using \eqref{eq:final.weights.1}, for $|x|\leq t^2$, we have
    \begin{equation}
        \int_{\{|x| \leq t^2\}} \int_{\bbR^3} \jap{x/t}^{2(K_{0}-2)} \jap{z}^{-1} ((I - \Pi) f)^2 \ud v\ud x \ls \|w_{(0,K_{0}-2,1)} e^{-\f{\alp \bt |x|^2}{2(\alp + \bt t^2)}} \jap{z}^{-\f 12} (I - \Pi) f\|_{L^2_{x,v}}^2.
    \end{equation}
By Lemma~\ref{lem:lower-bound-L} and Corollary~\ref{cor:fLf.final}, we obtain
    \begin{equation}
        \|w_{(0,K_{0}-2,1)} e^{-\f{\alp \bt |x|^2}{2(\alp + \bt t^2)}} \jap{z}^{-\f 12} (I - \Pi) f\|_{L^2_{x,v}}^2 \ls \langle w^2_{(0,K_{0}-2,1)} f, Lf \rangle_{L^2_{x,v}} \ls E_{K_{0}} t^{-1}.
    \end{equation}
    Combining, we obtain
    \begin{equation}\label{eq:microscopic.final.1}
        \int_{\{|x| \leq t^2\}} \int_{\bbR^3} \jap{x/t}^{2(K_{0}-2)} \jap{z}^{-1} ((I - \Pi) f)^2 \ud v\ud x \ls E_{K_{0}} t^{-1}.
    \end{equation}
    
    On the other hand, using \eqref{eq:final.weights.2} and Proposition~\ref{prop:f.final}, we obtain
    \begin{equation}\label{eq:microscopic.final.2}
        \int_{\{|x| \geq t^2\}} \int_{\bbR^3} \jap{x/t}^{2(K_{0}-2)} \jap{z}^{-1} ((I - \Pi) f)^2 \ud v\ud x \ls e^{-ct^2} \| w_{(0,0,1)} f \|_{L^2_{x,v}}^2 \ls e^{-ct^2} E_{K_{0}}.
    \end{equation}
    Summing \eqref{eq:microscopic.final.1} and \eqref{eq:microscopic.final.2} and taking square roots yield the desired conclusion. \qedhere
\end{proof}

We will now also prove the estimate \eqref{eq:weight.improved.intro} that was mentioned in Remark~\ref{rmk:weight.improved.intro}. In particular, this allows us to improve the $\jap{z}$ weight in Theorem~\ref{thm:microscopic} at the expense of an arbitrarily small loss in time decay.
\begin{proposition}\label{prop:weight.improved}
Under the assumptions of Theorem~\ref{thm:main}, for any $\ell \geq 0$ and $\ep >0$, the following holds for some implicit constants depending on $\ell$ and $\ep$:
    \begin{equation}
        \| \jap{z}^{\ell} \jap{x/t}^{K_{0}-2} (I - \Pi) f \|_{L^2_{x,v}}(t) \ls_{\ep,\ell} t^{-\f 12+\ep} E_{K_{0}}^{\f 12}.
    \end{equation}
\end{proposition}
\begin{proof}
    Given $\ep,\ell$, choose $\ell_*$ such that $\ell = -\f 12(1-\f \ep 4) + \f{\ep \ell_*}{4}$. Then 
    $$\jap{z}^\ell \leq (\jap{z}^{-\f 12})^{1-\f \ep 4} (\jap{z}^{\ell_*})^{\f \ep 4} .$$
    Hence, H\"older's inequality gives
    $$\| \jap{z}^{\ell} \jap{x/t}^{K_{0}-2} (I - \Pi) f \|_{L^2_{x,v}}(t) \ls \| \jap{z}^{-\f 12} \jap{x/t}^{K_{0}-2} (I - \Pi) f \|_{L^2_{x,v}}^{1-\f \ep 4}(t)\| \jap{z}^{\ell_*} w_{(0,K_{0}-2,0)} (I - \Pi) f \|_{L^2_{x,v}}^{\f \ep 4}(t).$$
    The first factor is to be bounded by Theorem~\ref{thm:microscopic}. For the second factor, first note that by \eqref{eq:Pi.in.weighted.space.3},
    $$\| \jap{x/t}^{K_{0}-2} \jap{z}^{\ell_*} \Pi f \|_{L^2_{x,v}} \ls  \|w_{(0,K_{0}-2,0)} f\|_{L^2_{x,v}}. $$
    Thus, using Proposition~\ref{prop:f.final}, we have
    $$\| \jap{z}^{\ell_*} \jap{x/t}^{K_{0}-2} (I - \Pi) f \|_{L^2_{x,v}}(t) \ls \| w_{(0,K_{0}-2,0)} \jap{z}^{\ell_*} f\|_{L^2_{x,v}} \ls \| w_{(1,K_{0}-2,0)} f\|_{L^2_{x,v}} \ls t^{0.55} E_{K_{0}}^{\f 12}.$$
    Putting everything together yields 
    $$\| \jap{z}^{\ell} \jap{x/t}^{K_{0}-2} (I - \Pi) f \|_{L^2_{x,v}}(t) \ls t^{-\f 12(1-\f{\ep}4)} t^{\f{0.55 \ep}{4}}E_{K_0}^{\f 12} \ls t^{-\f 12 + \ep} E_{K_0}^{\f 12}. $$\qedhere
\end{proof}

\subsection{Elliptic estimates and improved decay}

In this subsection, we carry out elliptic estimates to obtain improved decay estimates in localized regions, thus proving \eqref{eq:improved.decay.main.thm}.

We will again start with the normalized collisional operator and rescale. It will be convenient to make the analogous definitions as \eqref{eq:L} and \eqref{eq:A.decomposition} for the normalized collisional operator, i.e., we define
$$A_1g := \mu_1^{-1/2}Q(\mu_1, \mu^{1/2}_1 g),\quad K_1 g := \mu^{-1/2}_1 Q(\mu^{1/2}_1 g,\mu_1),$$
as well as decompose
$$A_1 = \bar{A}_1 + \widetilde{A}_1,$$
with 
\begin{equation}\label{eq:barA.tA}
    \bar{A}_1 = \rd_{v_i}(\sigma^1_{ij}\rd_{v_j}g)-\sigma^1_{ij}v_iv_j g,\quad \widetilde{A}_1 = \rd_{v_i}\sigma^1_i g,
\end{equation}
where $\sigma^1_{ij}$ is as in \eqref{eq:sigma1-integral} and $\sigma^1_i = \sigma^1_{ij} v_j$.

\begin{lemma}
Define
$$X = \f{v_i}{|v|}\rd_i,\quad \slashed{\rd}_{v_i} = \rd_{v_i} - \f{v_iv_j}{|v|^2}\rd_j,$$
where $X$ is the unit radial vector field, while $\slashed{\rd}_{v_i}$ is the orthogonal projection of $\rd_{v_i}$ to the angular directions.

Then the linearized operator can be written as 
\begin{equation}
\bar{A}_1 g = X (\lambda_1^1(v) X g) + \f 2{|v|} \lambda_1^1(v) X g + \de^{ij} \srd_i (\lambda_2^1(v) \srd_j g) - \sigma^1_{ij} v_i v_j g, 
\end{equation}
where $\lambda_1^1$, $\lambda_2^1$ are as in \eqref{eq:sigma.decomposition.DL.1}--\eqref{eq:sigma.decomposition.DL.2}. Moreover, $\lambda_1^1$, $\lambda_2^1$ satisfy the following estimates:
\begin{equation}\label{eq:kernel.weights.in.spherical}
\jap{v}^{-3} \ls \lambda^1_1  \ls \jap{v}^{-3},\quad \jap{v}^{-1} \ls  \lambda^1_2 \ls \jap{v}^{-1},\quad |\rd_{v}\lambda^1_1 | \ls \jap{v}^{-4},\quad |\rd_{v} \lambda^1_2 |\ls \jap{v}^{-2}.
\end{equation}
\end{lemma}
\begin{proof}
It will be convenient to write $\rd_i = \rd_{v_i}$ in this proof. Recall the definition of $\bar{A}_1$ in \eqref{eq:barA.tA}. By \cite[Lemma~4]{StGu08}, we have 
\begin{equation*}
\sigma^1_{ij}=\left[\lambda^1_1(v)\f{v_iv_j}{|v|^2}+\lambda^1_2(v)\left(\delta_{ij}-\f{v_iv_j}{|v|^2}\right)\right].
\end{equation*}
 That $\lambda^1_1$, $\lambda^1_2$ satisfy the first two bounds in \eqref{eq:kernel.weights.in.spherical} follow from \cite[Lemma~4]{StGu08}. For the derivative estimates, we use the representation formula in \eqref{eq:sigma.decomposition.DL.1}--\eqref{eq:sigma.decomposition.DL.2}. Differentiating \eqref{eq:sigma.decomposition.DL.2} immediately gives the bound for $|\rd_{v} \lambda^1_2 |$. For $\lambda^1_1(v)$, we need to capture a cancellation. Differentiating \eqref{eq:sigma.decomposition.DL.1} by $\rd_{v_i}$ and then integrating by parts, we obtain
\begin{equation}
    \begin{split}
        \rd_{v_i}\lambda^1_1(v) = &\: 2\f{v_i}{|v|} \int_{\R^3} |u|^{-1} \Big( 1 - \f{u_1^2}{|u|^2} \Big) (u_1 -|v|)  e^{-((u_1-|v|)^2+ u_2^2+ u_3^2)}  \, \ud u \\
        = &\:  -\f{v_i}{|v|} \int_{\R^3} |u|^{-1} \Big( 1 - \f{u_1^2}{|u|^2} \Big) \rd_{u_1}  e^{-((u_1-|v|)^2+ u_2^2+ u_3^2)}  \, \ud u \\
        = &\:  -\f{3v_i}{|v|} \int_{\R^3} u_1 |u|^{-5} \Big( u_2^2 + u_3^2 \Big)  e^{-((u_1-|v|)^2+ u_2^2+ u_3^2)}  \, \ud u \\
        \ls &\: \int_{\{|u| \leq 1\}} |u|^{-2} e^{-\f 12((u_1-|v|)^2+ u_2^2+ u_3^2)}  \, \ud u + \int_{\{|u|\geq 1\}} |u|^{-4}  e^{-\f 12((u_1-|v|)^2+ u_2^2+ u_3^2)}  \, \ud u \ls \jap{v}^{-4},
    \end{split}
\end{equation}
where at the very last step we used Lemma~\ref{lem:int-exp} with $c = (|v|,0,0)$.

Now observe
\begin{equation}
    \begin{split}
        X (\lambda^1_1(v) Xg) =  &\: \f{v_i}{|v|} \rd_{v_i} (\lambda^1_1(v) \f{v_j}{|v|} \rd_{v_j} g) \\
        = &\:  \rd_{v_i} (\lambda^1_1(v) \f{v_i}{|v|} \f{v_j}{|v|} \rd_{v_j} g) - \f 2{|v|} \lambda^1_1(v) X g
    \end{split}
\end{equation}
and
\begin{equation}
    \begin{split}
        \de^{ij} \srd_i (\lambda_2^1(v) \srd_j g) 
        = &\: (\rd_i - \f{v_i v_{i'}}{|v|^2} \rd_{i'})(\lambda_2^1(v) \de^{ij} (\rd_j - \f{v_j v_{j'}}{|v|^2} \rd_{j'}) g) \\
= &\: \de^{ij} \rd_{i}(\lambda_2^1(v) \rd_j g) - \de^{ij} \rd_i (\f{\lambda_2^1(v) v_j v_{j'}}{|v|^2} \rd_{j'} g) - \de^{ij} \rd_{i'}(\f{\lambda_2^1(v) v_i v_{i'}}{|v|^2} \rd_j g) \\
&\: + \de^{ij} \rd_{i'} (\f{\lambda_2^1(v) v_iv_{i'}v_jv_{j'}}{|v|^4} \rd_{j'} g) + \de^{ij} (\rd_{i'}\f{v_iv_{i'}}{|v|^2}) \lambda_2^1(v) (\rd_j - \f{v_jv_{j'}}{|v|^2} \rd_{j'})g  \\
= &\:  \rd_{i}\Big(\lambda_2^1(v) \Big(\de^{ij} - \f{v_i v_{j}}{|v|^2} \Big) \rd_j g \Big).
    \end{split}
\end{equation}
Rearranging yields the desired result. \qedhere
\end{proof}

We separately perform elliptic estimates for finite $v$ and large $v$: The large-$v$ estimates have to be carried out in spherical coordinates to capture the anisotropy, but spherical coordinates will turn out to be inconvenient for finite $v$ due to their irregularity at $v = 0$. For this purpose, let $\chi: \bbR^3 \to \bbR$ be a radially-symmetric, non-negative, smooth, compactly supported cutoff function such that $\chi \equiv 1$ on $B(0,1)$, $\mathrm{supp}(\chi) \subseteq B(0,2)$.

\begin{proposition}\label{prop:main.Fredholm.est}
The following holds for sufficiently regular functions $g$:
\begin{equation}\label{eq:main.Fredholm.est}
\begin{split}
&\: \| \jap{v}^{-3} (1-\chi) X^2 g\|_{L^2_v}^2  + \|\jap{v}^{-2}(1-\chi) \slashed{\nabla}_v X g \|_{L^2_v}^2 + \| \jap{v}^{-2} (1-\chi) Xg \|_{L^2_v}^2 \\
&\: + \| \jap{v}^{-1} (1-\chi)( |\slashed{\nabla}{}^2_v g| + |\slashed{\nabla}_v g|) \|_{L^2_v}^2 + \|\jap{v}^{-1} g\|_{L^2_v}^2 + \sum_{i,j=1}^3 \| \chi \rd^2_{v_i v_j} g\|_{L^2_v}^2  \\
\ls &\: \| L_1 g\|_{L^2_v}^2 + \| \jap{v}^{-7/2} Xg \|_{L^2_v}^2 + \| \jap{v}^{-3/2} (|\slashed{\nabla}_v g| + |g|) \|_{L^2_v}^2,
\end{split}
\end{equation}
where $\slashed{\nabla}_v$ and $\slashed{\nabla}_v^2$ are defined by 
\begin{align}
    | \slashed{\nabla}_v g |^2 = \sum_{i=1}^3 |\srd_{v_i} g|^2, \quad 
    | \slashed{\nabla}_v^2 g |^2 = \sum_{i,j=1}^3 |\srd_{v_i} \srd_{v_j} g|^2.
\end{align}
\end{proposition}
\begin{proof}
Since this estimate concerns only $v$-derivatives, we denote $\rd_i = \rd_{v_i}$.

Using $L_1 g = \bar{A}_1 g + \widetilde{A}_1 g + K_1 g$ and the triangle inequality, we deduce that
\begin{equation}
 \| \bar{A}_1 g\|_{L^2_v}^2  \ls \|L_1 g\|_{L^2_v}^2 + \| \widetilde{A}_1 g\|_{L^2_v}^2 + \| K_1 g \|_{L^2_v}^2 .
\end{equation}
Our argument then separates into showing coercivity of the $\bar{A}_1 g$ term (Step~1), and controlling the other terms (Step~2).

\pfstep{Step~1(a): Coercivity of $\bar{A}_1 g$ for finite $v$} For finite $v$, we have
\begin{equation}
\int_{\R^3} \chi^2 (\bar{A}_1 g)^2 \, \ud v = \int_{\R^3} \chi^2 \Big( \rd_i (\sigma^1_{ij} \rd_j g) - \sigma^1_{ij} v_i v_j g\Big)^2 \, \ud v.
\end{equation}
Integrating by parts and using that $\sigma^1_{ij}$ is uniformly positive definite on the support of $\chi$, we obtain
\begin{equation}
\sum_{i,j=1}^3 \| \chi (\rd_{ij}^2 g)\|_{L^2_v}^2 + \|\chi g\|_{L^2_v}^2 \ls \| \chi \bar{A}_1 g \|_{L^2_v}^2 + \sum_{i=1}^3 \| \rd_i g\|_{L^2_v(|v|\leq 2)}^2 + \|g\|_{L^2_v(|v|\leq 2)}^2.
\end{equation}
The last two terms are bounded by those on the right-hand side of \eqref{eq:main.Fredholm.est}, since $\sum_{i=1}^3|\rd_{v_i} g|^2 \ls |Xg|^2 + |\srd_v g|^2$ and all the weights are $\approx 1$  on $\{|v|\leq 2\}$.

\pfstep{Step~1(b): Coercivity of $\bar{A}_1 g$ for large $v$} We need some preliminary observations for our elliptic estimates.
For the vector fields $X$ and $\rd_i$, notice that
\begin{equation}\label{eq:spherical.commutator}
\begin{split}
[X,\srd_i] = &\: -\f{v_k}{|v|}\rd_k (\f{v_iv_j}{|v|^2})\rd_j - (\rd_i - \f{v_iv_j}{|v|^2} \rd_j) \f{v_k}{|v|} \rd_k\\
= &\: -\f{v_k}{|v|}(\f{\de_{ik} v_j}{|v|^2} + \f{v_i \de_{jk}}{|v|^2} - \f{2 v_iv_jv_k}{|v|^4})\rd_j  - (\f{\de_{ik}}{|v|} - \f{v_iv_k}{|v|^3} - \f{v_iv_j \de_{jk}}{|v|^3}+ \f{v_iv_j v_j v_k}{|v|^5})\rd_k \\
= &\: - \f{1}{|v|} \rd_i + \f{v_i}{|v|^2}X = -\f{1}{|v|} \srd_i.
\end{split}
\end{equation}
In addition, observe the following identities which will be used to integrate by parts:
\begin{equation}\label{eq:spherical.ibp}
\int_{\R^3} X h \ud v = -2 \int_{\R^3} |v|^{-1} h \ud v,\quad \int_{\R^3} \srd_i (h_i)  \ud v = 2  \int_{\R^3} \f{v_i}{|v|^2} h_i  \ud v
\end{equation}
for $h$ decaying suitably at infinity.
 
\begin{equation}
\begin{split}
&\: \int_{\R^3} (1-\chi)^2 (\bar{A}_1 g)^2 \, \ud v \\
= &\: \int_{\R^3}(1-\chi)^2 \Big( |v|^{-2} X (|v|^2 \lambda^1_1(v) X g) + \de^{ij} \srd_i (\lambda^1_2(v) \srd_j g) - \sigma^1_{ij} v_i v_j g \Big)^2 \ud v \\
= &\: \int_{\R^3} (1-\chi)^2\Big( (|v|^{-2} X (|v|^2\lambda^1_1(v) X g))^2 + (\de^{ij} \srd_i (\lambda^1_2(v) \srd_j g))^2 + (\sigma^1_{ij} v_i v_j g)^2 \Big) \ud v \\
&\: - 2 \int_{\R^3} (1-\chi)^2\Big( |v|^{-2} X (|v|^2 \lambda^1_1(v) X g)\sigma^1_{ij} v_i v_j g + \de^{ij} \srd_i (\lambda^1_2(v) \srd_j g) \sigma^1_{i'j'} v_{i'} v_{j'} g\Big) \ud v \\
&\: + 2 \int_{\R^3} (1-\chi)^2 |v|^{-2} X (|v|^2 \lambda^1_1(v) X g) \de^{ij} \srd_i (\lambda^1_2(v) \srd_j g) \ud v \\
=: &\: I + II +\cdots + VI.
\end{split}
\end{equation}
The terms $I$, $II$, $III$ are manifestly non-negative. Integrating by parts with \eqref{eq:spherical.ibp}, and using \eqref{eq:kernel.weights.in.spherical}, we see that they control
\begin{equation}
\begin{split}
&\: \| \jap{v}^{-3} (1-\chi) X^2 g\|_{L^2_v}^2 + \| \jap{v}^{-1} (1-\chi)( |\slashed{\nabla}{}^2_v g| + |g|) \|_{L^2_v}^2  \\
\ls &\: I + II + III + \| \jap{v}^{-4} X g\|_{L^2_v}^2 + \| \jap{v}^{-3} |\slashed{\nabla}_v g|\|_{L^2_v}^2.
\end{split}
\end{equation}
For terms $IV$, $V$ and $VI$, we integrate by parts. Using \eqref{eq:spherical.ibp}, we obtain
\begin{equation}
\begin{split}
&\: \| \jap{v}^{-2} (1-\chi) Xg \|_{L^2_v}^2 + \| \jap{v}^{-1} (1-\chi) |\slashed{\nabla}_v g| \|_{L^2_v}^2 \\
\ls &\: IV + V + \| \jap{v}^{-7/2} Xg \|_{L^2_v}^2 + \| \jap{v}^{-3/2} (|\slashed{\nabla}_v g| + |g|) \|_{L^2_v}^2.
\end{split}
\end{equation}
For $VI$, we integrate by parts using \eqref{eq:spherical.ibp} and \eqref{eq:spherical.commutator} (first integrate by parts $X$ then commute $[X,\srd_i]$, then integrate by parts $\srd_i$) to obtain
\begin{equation}
\begin{split}
&\: \|\jap{v}^{-2}(1-\chi) \slashed{\nabla}_v X g \|_{L^2_v}^2 \\
\ls &\: VI + (\|\jap{v}^{-1} (1-\chi) \slashed{\nabla}{}^2_v g\|_{L^2_v}+ \|\jap{v}^{-2}(1-\chi) \slashed{\nabla}_v X g \|_{L^2_v}) \| \jap{v}^{-3}  Xg\|_{L^2_v} \\
&\: + \| \jap{v}^{-3}  Xg\|_{L^2_v}^2 + \| \jap{v}^{-3}  \slashed{\nabla}_v g\|_{L^2_v}^2,
\end{split}
\end{equation}
where we have used the fact that on the support of $(1-\chi)$, $|v|^{-1} \sim \jap{v}^{-1}$. Combining all the above estimates and suitably absorbing terms to the left-hand side using Young's inequality, we obtain
\begin{equation}
\begin{split}
&\: \| \jap{v}^{-3} (1-\chi) X^2 g\|_{L^2_v}^2  + \|\jap{v}^{-2}(1-\chi) \slashed{\nabla}_v X g \|_{L^2_v}^2 + \| \jap{v}^{-2} (1-\chi) Xg \|_{L^2_v}^2 \\
&\: + \| \jap{v}^{-1} (1-\chi)( |\slashed{\nabla}{}^2_v g| + |\slashed{\nabla}_v g|+ g) \|_{L^2_v}^2 \\
\ls &\: \| (1-\chi) \bar{A}_1 g \|_{L^2_v}^2 + \| \jap{v}^{-7/2} Xg \|_{L^2_v}^2 + \| \jap{v}^{-3/2} (|\slashed{\nabla}_v g| + |g|) \|_{L^2_v}^2.
\end{split}
\end{equation}

\pfstep{Step~2: Control of the lower order terms $\widetilde{A}_1 g$ and $K_1 g$} We consider $\widetilde{A}_1$ given by \eqref{eq:barA.tA}. By \cite[Lemma~4]{StGu08}, $|\rd_{v_i} \sigma^1_i|\ls \jap{v}^{-2}$ and thus
\begin{equation}
    \| \widetilde{A}_1 g \|_{L^2_v} \ls \| \jap{v}^{-2} g \|_{L^2_v}.
\end{equation}

Arguing as in \eqref{eq:K.representation}, we have
\begin{equation}
\begin{split}
 K_1 g = &\: 8\pi  \mu_1 g + 4 \mu_1^{\f 12}   \int_{\R^3}\Big( \phi_{ij}(v-v_*) v_{*i} v_{*j} - |v-v_*|^{-1} \Big) \mu_{1*}^{\f 12} {g}_* \ud v_*
\end{split}
\end{equation} and
\begin{equation}
\begin{split}
\| K_1 g \|_{L^2_v}^2 \ls \|\mu_1 g \|_{L^2_v}^2 + \norm{\int_{\bbR^3} \mu^{\f 12}_1 |v-v_*|^{-1}  (\mu^{\f 12}_1 \jap{v}^2 |g|)_* \ud v_*}_{L^2_v}^2.  
\end{split}
\end{equation}
By duality and Lemma~\ref{lem:HLS} with $\lambda = 1$ and $p = \f 65$, we have
\begin{equation*}
    \begin{split}
        &\: \norm{\int_{\bbR^3} \mu^{\f 12}_1 |v-v_*|^{-1}  (\mu^{\f 12}_1 \jap{v}^2 |g|)_* \ud v_*}_{L^2_v} \ls \sup_{\| \widetilde{g}\|_{L^2_v} = 1} \int_{\bbR^3} \int_{\bbR^3} \widetilde{g} \mu^{\f 12}_1 |v-v_*|^{-1}  (\mu^{\f 12}_1 \jap{v}^2 |g|)_* \ud v_* \ud v \\
        \ls &\: \sup_{\| \widetilde{g}\|_{L^2_v} = 1}  \| \mu^{\f 12}_1 \widetilde{g}\|_{L^{6/5}_v} \| \mu^{\f 12}_1 \jap{v}^2 g\|_{L^{6/5}_v} \ls \|\mu^{\f 12}_1\|_{L^{3}_v} \| \mu^{\f 14}_1 \jap{v}^2 \|_{L^{3}_v} \|\mu_1^{\f 14} g \|_{L^2_v} \ls \|\mu_1^{\f 14} g \|_{L^2_v}.
    \end{split}
\end{equation*}
Hence, 
\begin{equation}
\begin{split}
\| K_1 g \|_{L^2_v}^2 \ls \|\mu_1^{\f 14} g \|_{L^2_v}^2 \ls \| \jap{v}^{-2} g\|_{L^2_v}^2,
\end{split}
\end{equation}
where we used $\mu_1^{\f 14} \ls \jap{v}^{-2}$. Thus, both $\widetilde{A}_1g$ and $K_1g$ obey the desired estimate. \qedhere
\end{proof}

\begin{lemma}\label{lem:Fredholm}
Suppose $\calX,\calY,\calZ$ are separable Hilbert spaces, $\calL:\calX \to \calY$ is a bounded linear map, and $\calT: \calX \to \calZ$ is a compact linear map such that the following estimate holds:
\begin{equation}\label{eq:Fredholm}
\| g \|_{\calX} \ls \|\calL g\|_{\calY} + \| \calT g \|_{\calZ}\quad \forall g \in \calX.
\end{equation}
Assume that $\mathrm{ker}(\calL) = \{0\}$. Then the following estimate holds:
\begin{equation}\label{eq:Fredholm.conclusion}
\| g \|_{\calX} \ls \|\calL g\|_{\calY} \quad \forall g \in \calX.
\end{equation}
\end{lemma}
\begin{proof}
If we had the estimate
\begin{equation}\label{eq:Fredholm.goal}
\| \calT g \|_{\calZ} \ls \|\calL g\|_{\calY} \quad \forall g \in \calX,
\end{equation}
we would have proven \eqref{eq:Fredholm.conclusion} after using \eqref{eq:Fredholm}. We thus assume that \eqref{eq:Fredholm.goal} is false, i.e., that there exists a sequence $\{g_n\}_{n=1}^\infty \subseteq \calX$ such that 
\begin{equation}
\| \calT g_n \|_{\calZ} \geq n \|\calL g_n \|_{\calY}.
\end{equation}
Without loss of generality, we rescale so that $\| \calT g_n \|_{\calZ} = 1$ and $\| \calL g_n \|_{\calY} \leq \f 1 n$. Plugging these two bounds into \eqref{eq:Fredholm}, we thus obtain that $g_n$ is uniformly bounded in $\calX$. In particular, by the Banach--Alaoglu theorem \cite[Theorem~3.2.1]{tBdaS2018}, there exist a subsequence $\{g_{n_k}\}_{k=1}^\infty$ and a $g_\infty \in \calX$ such that $g_{n_k} \to g_\infty$ weakly in $\calX$ as $k\to \infty$. Since $\calL$ is bounded, $\calL g_{n_k} \to \calL g_\infty$ weakly in $\calY$ as $k \to \infty$. Now by the weak convergence, it holds that $0 \leq \liminf_{k\to \infty} \| \calL g_{n_k} - \calL g_\infty \|_{\calY}^2 = (\liminf_{k\to \infty} \|\calL g_{n_k}\|_{\calY}^2) -  \|\calL g_\infty \|_{\calY}^2$, thus  $\|\calL g_\infty\|_{\calY} \leq \liminf_{k\to \infty} \| \calL g_{n_k}\|_{\calY} = 0$, i.e., $\calL g_\infty = 0$. Since $\calL$ has a trivial kernel, it follows that $g_\infty = 0$.

Since $\calT$ is compact, it maps weakly convergent sequences to norm convergent ones. Hence $\lim_{k\to \infty} \|\calT g_{n_k} \|_{\calZ} = 0$. However, this contradicts the fact that $\|\calT g_{n} \|_{\calZ} = 1$ for all $n \in \mathbb Z_{\geq 1}$. Therefore, \eqref{eq:Fredholm.goal} in fact holds and we have proven \eqref{eq:Fredholm.conclusion}. \qedhere

\end{proof}

\begin{corollary}\label{cor:elliptic.final.scale.1}
    Denote by $\Pi_1$ the $L^2$ orthogonal projection to $\mathrm{span}\{\sqrt{\mu_1}, v_i \sqrt{\mu_1}, |v|^2\sqrt{\mu_1}\}$ with respect to $\ud v$. Then the following estimate holds:
    \begin{equation}
        \| \jap{v}^{-1} (I-\Pi_1) g \|_{L^2_v} \ls \| L_1 g\|_{L^2_v}.
    \end{equation}
\end{corollary}
\begin{proof}
    Let $\calY = L^2_v$, and define $\calZ$ by the norms
    \begin{align*}
        \| g \|_{\calZ}^2 := &\: \| \jap{v}^{-7/2} Xg \|_{L^2_v}^2 + \| \jap{v}^{-3/2} (|\slashed{\nabla}_v g| + |g|) \|_{L^2_v}^2.
    \end{align*}
    For $\calX$, we first define $\calX_0$ by the norm
    \begin{align*}
\| g\|_{\calX_0}^2 := &\: \| \jap{v}^{-3} (1-\chi) X^2 g\|_{L^2_v}^2  + \|\jap{v}^{-2}(1-\chi) \slashed{\nabla}_v X g \|_{L^2_v}^2 + \| \jap{v}^{-2} (1-\chi) Xg \|_{L^2_v}^2 \\
&\: + \| \jap{v}^{-1} (1-\chi)( |\slashed{\nabla}{}^2_v g| + |\slashed{\nabla}_v g|) \|_{L^2_v}^2 + \|\jap{v}^{-1} g\|_{L^2_v}^2 + \sum_{i,j=1}^3 \| \chi \rd^2_{v_i v_j} g\|_{L^2_v}^2.
\end{align*}
    Observe that $\Pi_1: \calX_0 \to \calX_0$ is bounded. Define $\calX = (I-\Pi_1) \calX_0$. 
    
    We now apply Lemma~\ref{lem:Fredholm} with $\calX$, $\calY$, $\calZ$ as above, $\calL = L_1$ and $\calT: \calX \to \calZ$ be the obvious inclusion map. Since the norm $\calX$ has both higher derivatives and more weights, by using Rellich's theorem together with the weights at infinity, we deduce that $\calT$ is compact. It is also easy to check that $\calL: \calX \to \calY$ is bounded. Moreover, Proposition~\ref{prop:main.Fredholm.est} implies that \eqref{eq:Fredholm} holds. Finally, observe that \cite[Lemma~4]{Guo02} shows that $\mathrm{ker}(L_1) \cap \calX_0 = \mathrm{span}\{\sqrt{\mu_1}, v_i \sqrt{\mu_1}, |v|^2\sqrt{\mu_1}\}$. (Indeed, the weighted space $\calX_0$ is sufficient to justify the proof in \cite{Guo02}.) Thus $\mathrm{ker}(\calL) = \{0\}$. Hence, $\|g\|_{\calX} \ls \|\calL g \|_{\calY}$, which implies the desired bound. \qedhere
\end{proof}

By a rescaling argument using Proposition~\ref{prop:normalize} and Corollary~\ref{cor:elliptic.final.scale.1}, we thus obtain
\begin{corollary}\label{cor:elliptic.final}
The following holds for a sufficiently regular $g$:
    \begin{equation}
        e^{-\f{\alp \bt |x|^2}{\alp + \bt t^2}} \| \jap{z}^{-1} (I -\Pi) g \|_{L^2_v} \ls \|L g\|_{L^2_v}.
    \end{equation}
\end{corollary}

We now apply the result to obtain improved decay estimates for $f$, which proves \eqref{eq:improved.decay.main.thm}. 
\begin{theorem}\label{thm:microscopic.improved}
Under the assumptions of Theorem~\ref{thm:main},
    \begin{equation}
        \| e^{-\f{\alp \bt |x|^2}{\alp + \bt t^2}} \jap{z}^{-1} \jap{x/t}^{K_{0}-3} (I - \Pi) f \|_{L^2_{x,v}}(t) \ls t^{-1} E_{K_{0}}^{\f 12}.
    \end{equation}
\end{theorem}
\begin{proof}
Combining Corollary~\ref{cor:elliptic.final} and Corollary~\ref{cor:Lf.final}, we obtain
    \begin{equation}
        \| e^{-\f{\alp\bt |x|^2}{\alp+ \bt t^2}} \jap{z}^{-1} w_{(0,K_{0}-3,0)} (I-\Pi) f \|_{L^2_{x,v}}(t) \ls \| w_{(0,K_{0}-3,1)} Lf \|_{L^2_{x,v}}(t)  \ls t^{-1} E_{K_{0}}^{\f 12},
    \end{equation}
    which implies the desired conclusion. \qedhere
\end{proof}

\bibliographystyle{plain}
\bibliography{VPL}
\end{document}